\documentclass[11pt]{article}
\usepackage{amssymb,latexsym,amsmath,amsbsy,amsthm,amsxtra,amsgen,graphicx,makeidx,dsfont,sidecap,comment,mathtools,caption}
\usepackage[utf8]{inputenc}
\usepackage{tikz}
\usetikzlibrary{decorations.pathreplacing}
\usepackage{upgreek}
\newfont{\msbm}{msbm10 at 11pt}

\newfont{\msbmsm}{msbm10 at 8pt}

\newtheorem{Theo}{Theorem}
\newtheorem{Lemma}[Theo]{Lemma}

\numberwithin{equation}{section}

\begin{document}
\title{The Yule-$\Lambda$ Nested Coalescent: Distribution of the Number of Lineages}
\author{Toni Gui}
\maketitle

\vspace{-.3in}
\begin{abstract}
We study a model of a population with individuals  sampled from different species. The Yule-$\Lambda$ nested coalescent describes the genealogy of the sample when each species merges with another randomly chosen species with a constant rate $c$ and the mergers of individuals in each species follow the $\Lambda$-coalescent. For the Yule-$\Lambda$ nested coalescent with $c<\int_0^1x^{-1}\Lambda(dx)<\infty$, where $\Lambda$ is the measure that characterizes the $\Lambda$-coalescent, we show that under some initial conditions, the distribution of the number of individual lineages belonging to one species converges weakly to the distribution $\nu_c^*$, which is the solution to some recursive distributional equation (RDE) with finite mean. In addition, we show that for some values of $c$, the RDE has another solution with infinite mean.
  \\


\end{abstract}


\section{Introduction}
Coalescent theory is the study of stochastic processes that describe how particles merge over time.
Since particles in coalescent theory are often used to represent individual genes, coalescent processes can help describe the genealogy of a population and trace the common origin of sampled genes. Before formally defining the Yule-$\Lambda$ nested coalescent, we will first introduce the $\Lambda$-coalescent and nested coalescents.

\subsection{The $\Lambda$-coalescent}\label{chap1.2introlambda}

One of the most widely studied stochastic processes of coalescence is Kingman's coalescent \cite{KINGMAN1982235,kingman_1982}. Kingman's coalescent can be used to describe the genealogy of a sample from a population under certain conditions, such as constant population size, neutral evolution and typically small family sizes. But when some of the conditions are not satisfied, a more general coalescent model is needed to describe the genealogy. Kingman's coalescent is a special case of the $\Lambda$-coalescent, which is also known as a coalescent with multiple collisions and was introduced by Pitman in \cite{pitman1999} and Sagitov in \cite{10.2307/3215582}. It was shown in \cite{DURRETT20051628} and \cite{10.1214/17-EJP58} that the genealogy of a population undergoing selection under some appropriate assumptions can be described by the $\Lambda$-coalescent. The $\Lambda$-coalescent can also be used to model populations in which the number of offspring of a random individual can be relatively large. Such models were studied in detail in \cite{10.2307/3215582} and \cite{SCHWEINSBERG2003107}. Some results on the $\Lambda$-coalescents and their applications to theoretical population genetics are surveyed in \cite{berestycki2009recent}. 

We will now formally define the $\Lambda$-coalescent. We can start by defining a process $(\Pi_t^n,t\geqslant 0)$ taking values in the space $\mathcal{P}_n$, which is defined as the set of the partitions of $[n]=\{1,\ldots,n\}$. This process is defined by:
\begin{enumerate}
        \item     Initially $\Pi^n_{0}=\{\{1\},\ldots,\{n\}\}$.
        \item    The merger rate of any given $k$-tuple of blocks when the partition has $b$ blocks is denoted by $\lambda_{b,k}$, and there exists a finite measure $\Lambda$ on the interval $[0,1]$,  such that 
        \begin{equation}\label{posterlambda}
        \lambda_{b,k}=\int_{[0,1]}x^{k-2}(1-x)^{b-k}\Lambda(dx),~2\leqslant k\leqslant b.
        \end{equation}
     \end{enumerate}

The merger rate given in (\ref{posterlambda}) satisfies the recursion 
$\lambda_{b,k}=\lambda_{b+1,k}+\lambda_{b+1,k+1}$. Because this consistency condition holds, there exists a unique in law continuous-time Markov process $(\Pi_t,t\geqslant 0)$ taking values in $\mathcal{P}_\infty$, which is defined as the set of the partitions of $\mathbb{N}^+=\{1,2,3,\ldots\}$, such that the initial state of the process is the partition of $\mathbb{N}^+$ into singletons and the law of $\Pi$, when restricted to $[n]$, equals the law of $\Pi^n$. The law of $\Pi$ is uniquely characterized by the measure $\Lambda$.

Kingman's coalescent is a special case of the $\Lambda$-coalescent in which the measure $\Lambda$ equals a unit mass at 0. If $\Lambda(dx)$ is the Beta$(2-\alpha,\alpha)$ distribution with $0<\alpha<2$, then the resulting coalescent is called a beta coalescent with parameter $\alpha$. Beta coalescents can be used to approximate the genealogy of a population when the offspring distribution has heavy tails \cite{SCHWEINSBERG2003107}.

Suppose $\Lambda(\{0\})=0$. Then the process $(\Pi_t,t\geqslant 0)$ can also be constructed as follows. Consider a Poisson point process on $[0,\infty)\times[0,1]$ with intensity $dt\times x^{-2}\Lambda(dx)$. If $\Pi_{t-}=\{B_1,B_2,\ldots\}$, where the blocks are ranked by their smallest elements, and if $(t,x)$ is a point of the Poisson point process, then by flipping a coin, which has probability $x$ of coming up heads, for each of the blocks in $\Pi_{t-}$ and merging blocks whose coins are heads, we can obtain $\Pi_{t}$. 

Suppose that the finite measure $\Lambda$ is a probability measure; if it is not, it can be normalized to a probability measure. Let $X$ be a random variable on $[0,1]$ with law 
$$P(X\in dx)=\Lambda(dx).$$ 
Let $\lambda_b$ be the rate at which mergers happen when there are $b$ blocks in total. Then
\begin{equation}\label{lambda11}
	\lambda_{b,k}=E\left[X^{k-2}(1-X)^{b-k}\right],~2\leqslant k\leqslant b,
\end{equation}
\begin{equation}\label{lambda22}
	\lambda_{b}=\sum_{k=2}^{b}\binom{b}{k}\lambda_{b,k}=\sum_{k=2}^{b}\binom{b}{k}E\left[X^{k-2}(1-X)^{b-k}\right], ~b\geqslant 2,
\end{equation}
and
\begin{equation}\label{lambda33}
	\lambda_{1}=0.
\end{equation}

For some measures $\Lambda$, there is a positive fraction of singleton blocks in $\Pi_t$ at all times. Such $\Lambda$-coalescents are called coalescents with dust. More formally, suppose that $P(X=1)=0$, and let $E$ denote the event that $\Pi_t$ has some singletons for all $t>0$. Then $P(E)=1$ if and only if $E[1/X]<\infty$ and otherwise $P(E)=0$. For example, the beta coalescent with $0<\alpha<1$ satisfies $E[1/X]<\infty$ and has a positive fraction of singletons for all $t>0$. See \cite{freund_mohle_2017, gnedin2011lambdacoalescents, Mhle2021TheRO} for some further studies of $\Lambda$-coalescents with dust.

\subsection{Nested Coalescents}

A constraint of the $\Lambda$-coalescent is that it is only applicable for approximating the genealogy within a single species, but is not suitable for investigating the inter-species relationships. The challenge of reconstructing species trees from gene trees has received attention, as demonstrated in, for example, \cite{10.2307/2413694}. In recent years, there has been a significant growth in the mathematical theory for estimating species trees, as reviewed in \cite{doi:10.1146/annurev-ecolsys-012121-095340}.

The multispecies coalescent model extends the coalescent model, which is originally only applicable to single populations, to a process on a tree of populations that branches during speciation events \cite{steel2016phylogeny}. In multispecies coalescent models, the species trees are either fixed or drawn from some probability distribution. Nested coalescents were introduced in \cite{blancas2018} to capture the joint dynamics of genes and species and to show a much broader class of Markov models for trees within trees. 
While the gene tree reflects the coalescence of lineages within a species, the species tree describes the coalescence of distinct species.
The nested coalescent describes the genealogy of a population at both the individual and species levels. 

One example of a nested coalescent model is the nested Kingman coalescent, in which both the gene tree and the species tree are given by Kingman's coalescent. That is, each pair of individual lineages belonging to the same species merges independently at rate 1 and each pair of species merges independently at some positive constant rate $c$. It has been shown in \cite{blancas2019} that the number of total lineages in the nested Kingman coalescent decays at rate $O(t^{-2})$ when time  $t\to0$. By solving a McKean-Vlasov equation in \cite{lambert2020coagulation}, it has been shown that the typical number of individual lineages in one species is $O(t^{-1})$ at time $t$ when $t\to0$. Another example of a nested coalescent model is the Yule-Kingman nested coalescent, which is introduced in \cite{YuleKingmanToni}. In this model, the species tree is a Yule tree and individual lineages merge according to Kingman's coalescent.

In this article we study the Yule-$\Lambda$ nested coalescent model.
The species tree is a Yule tree in this model. The Yule process is a process in which each individual gives birth to a new individual independently at some constant rate $c$. It is a standard and commonly used model to describe phylogeny at the species level \cite{steel2016phylogeny}. For the Yule-$\Lambda$ nested coalescent in which the $\Lambda$-coalescent is coalescent with dust, we study the distribution of the number of individual lineages in one species and prove the existence and uniqueness of the distribution $\nu_c^*$ with finite mean that solves some recursive distributional equation. 
We show that the distribution of the number of individual lineages in each species converges weakly yo $\nu_c^*$ under some initial condition. In addition, we prove that for some values of $c$, the recursive distributional equation has another solution with infinite mean.
Some similar results on the Yule-Kingman nested coalescent model are shown in \cite{YuleKingmanToni}.

\begin{figure}[ht]
\centering
\captionsetup{width=.85\linewidth} \includegraphics[width = 80mm]{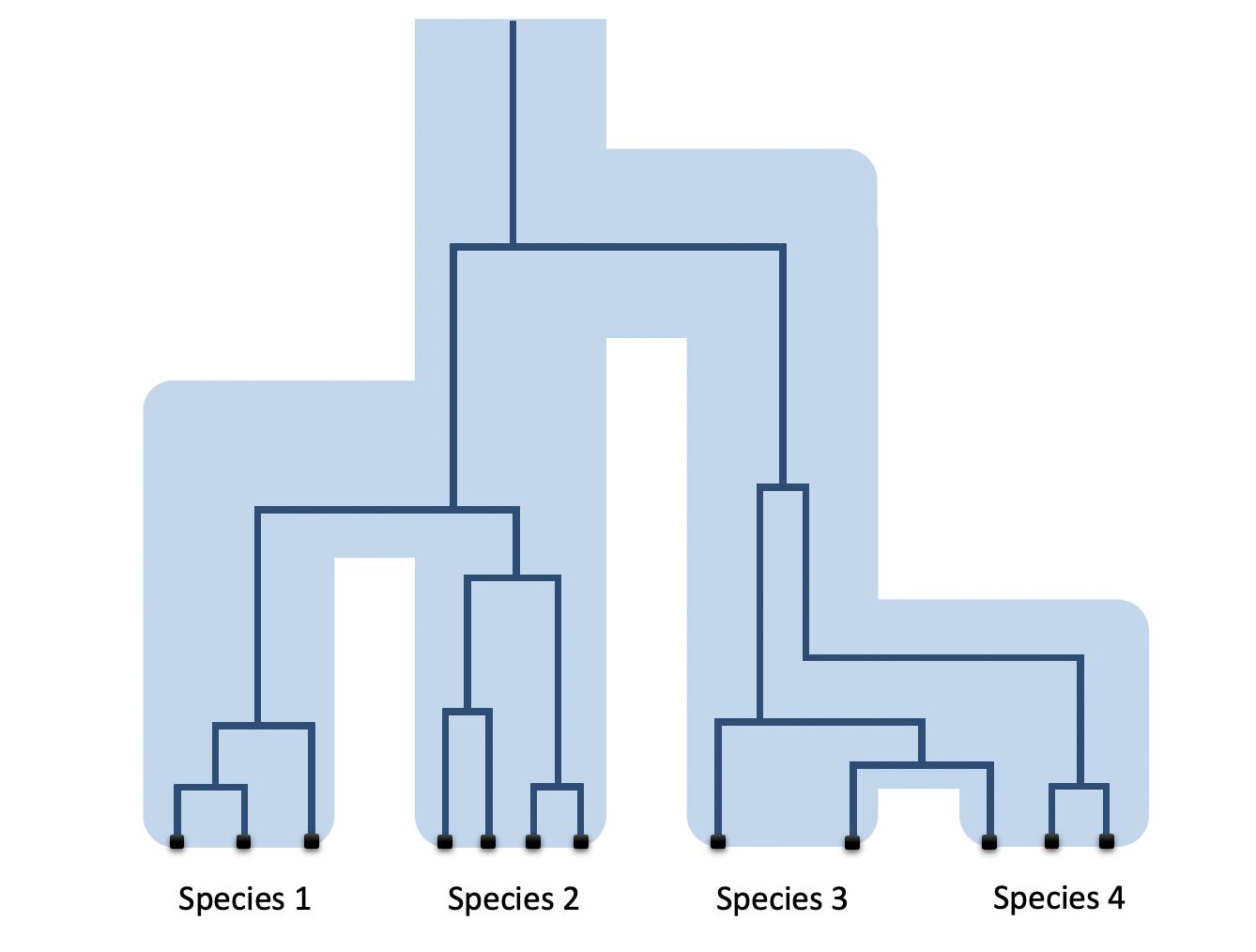}
\caption[A nested coalescent tree which illustrates a model starting from 4 species and 3, 4, 2, and 3 individual lineages in species 1, 2, 3, and 4 respectively.]{(From \cite{YuleKingmanToni}) A nested coalescent tree which illustrates a model starting from 4 species and 3, 4, 2, and 3 individual lineages in species 1, 2, 3, and 4 respectively. The dark lines present a possible ancestral tree for the sampled individual lineages and the light blue tree in the background shows a possible species tree. Only individual lineages belonging to the same species can merge.}
\label{fig.1}
\end{figure}

We will formally define the model and state the main results in the remainder of this section.


\subsection{Definition of the Yule-$\Lambda$ Nested Coalescent Model}\label{sec13}
We consider a nested coalescent model that is similar to  the ones that appear in \cite{blancas2019} and \cite{YuleKingmanToni}.
We start with a sample of $s$ species, and sample $n_k$ individuals from species $k$ ($n_k$ can be infinite).
A continuous-time Markov chain can be defined to formally describe the process. The Markov chain takes its values in the set of labeled partitions of $\{(m,k)\in\mathbb{Z}\times\mathbb{Z}:1\leqslant m\leqslant n_k,1\leqslant k\leqslant s\}$, where every partition block is assigned a label with one of the integers $1,\ldots,s$.
Each label represents one species, and each block corresponds to a lineage. Initially, there are $\sum_{k=1}^s n_k$ singleton blocks, and block $\{(m,k)\}$ is labeled by $k$. There are two types of transitions:\\

\noindent$\textbf{Lineage mergers:}$ When a label has $b$ blocks, any particular set of $k$ blocks with the label may merge independently into a single block with that label at rate given by
$$\lambda_{b,k}=\int_{[0,1]}x^{k-2}(1-x)^{b-k}\Lambda(dx),$$
where $\Lambda(dx)$ is a probability measure on $[0,1]$.\\

\noindent$\textbf{Species mergers:}$ Each label has a constant death rate of $c>0$. If there is more than one species, when label $j$ dies out, all blocks with label $j$ will have their label changed to $i$, which will be chosen from the remaining labels uniformly at random. If there is only one species, its death results in the simultaneous extinction of all blocks.\\

Note that within each species, the lineages merge according to the $\Lambda$-coalescent. The species mergers are the same as the species mergers in \cite{YuleKingmanToni}, where it was observed that if we reverse the direction of the time and redefine the time for the last species in the model to die out as time 0, then the evolutionary process of speciation is a Yule process with rate $c$, and the population has size $s$ when run until the last time. Therefore, we will refer to this model as the nested Yule-$\Lambda$ coalescent.

\subsection{Main Results for the Yule-$\Lambda$ Nested Coalescent}\label{chap1.6introresultYL}
Some of the notations below are similar to or the same as the ones used in \cite{YuleKingmanToni}.
Let $\mathcal{S}$ denote the set of probability distributions on $\mathbb{N}^+\cup\{\infty\}$, and let $\mathcal{S}_{1}$ denote the set of probability distributions on $\mathbb{N}^+$ with finite mean. Let $(\Pi_t,t\geqslant 0)$ be a $\Lambda$-coalescent process. Let $L_n(t)$ denote the number of blocks in $\Pi_t$ restricted to $[n]$. In other words, $L_n(t)$ is the number of lineages surviving to time $t$ in the $\Lambda$-coalescent started with $n$ lineages. Let $G_c:\mathcal{S}\to\mathcal{S}$ be the mapping defined such that $G_c(\mu)$ is the distribution of 
\begin{equation}\label{dist111}
L_{W_{1}+W_{2}}(Y),
\end{equation}
where $W_1$ and $W_2$ have distribution $\mu$, $Y$ is exponentially distributed with rate $c$, and the random variables $Y$, $W_{1}$, $W_{2}$ and the process $\Pi$ are independent. 
Let $G_c^{n}: \mathcal{S}\to\mathcal{S}$ be the map obtained by iterating the map $G_c$ for $n$ times. Let $S(t)$ be the number of species at time $t$. Notice that $S(0)=s$, and the process $(S(t), t\geqslant0)$ 
is a pure death process.  
Let $N_{k}(t)$ denote the number of individuals in species $k$ at time $t$, $1\leqslant k\leqslant s$. Notice that $N_{k}(t)=0$ if species $k$ dies out before time $t$, and there are $S(t)$ species that have a non-zero number of individuals at time $t$. For a fixed $m$, let $\tau_{m}^{s}=\sup\{t:S(t)=m\}$, which is the first time when the number of species reaches $m-1$ starting from $s$. Denote the left limit $\lim_{t\to(\tau_{m}^{s})^{-}}N_k(t)$ by $N_{k}((\tau_{m}^{s})^{-})$, which represents the number of individuals in species $k$ just before the merger that reduces the number of species from $m$ to $m-1$. Let $\{k_l\}_{l=1}^m\subset[s]$ denote the species that survive to $(\tau_{m}^{s})^{-}$. We can show the following results.

\begin{Theo}\label{theo4unilambda}
If $c<E[1/X]<\infty$, then there exists a unique distribution $\nu_c^*$ in $\mathcal{S}_1$ such that 
\begin{equation}\label{theoremequ1}
	G_c(\nu_c^*)=\nu_c^*.
\end{equation}
\end{Theo}

Note that if $P(X=1)=0$, then the Dirac measure on $\{\infty\}$ is a trivial solution to (\ref{theoremequ1}), because the $\Lambda$-coalescent has dust, and therefore if there are infinitely many blocks at time 0, the number of blocks remains infinite at all times. Therefore, the solution to (\ref{theoremequ1}) in $\mathcal{S}$ is not always unique for the Yule-$\Lambda$ nested coalescent model, which is different from the result proved in \cite{YuleKingmanToni} for the Yule-Kingman nested coalescent model.

\begin{Theo}
\label{Theoexistinflambda}
Suppose that $E[1/X]<\infty$ and $m\in\mathbb{N}^+$ is fixed. If $c<E[1/X]$ and $N_k(0)=1$ for all $k\in\{1,\ldots,s\}$, then the random variables $N_{k_l}((\tau_{m}^{s})^{-})$, $l\in\{1,\ldots,m\}$ are asymptotically independent as $s\to\infty$, and 
$$N_{k_l}((\tau_{m}^{s})^{-})\rightarrow_{d}\nu_c^*,~as~s\rightarrow\infty, ~\forall l\in\{1,\ldots,m\}.$$
If $N_k(0)=\infty$ for all $k\in\{1,\ldots,s\}$, then the random variables $N_{k_l}((\tau_{m}^{s})^{-})$, $l\in\{1,\ldots,m\}$ are asymptotically independent as $s\to\infty$, and there exists a distribution $\nu_c'\in\mathcal{S}$ such that
$$N_{k_l}((\tau_{m}^{s})^{-})\rightarrow_{d}\nu_c',~as~s\rightarrow\infty, ~\forall l\in\{1,\ldots,m\}.$$
We have $\nu_c'(\{\infty\})>0$ if $c>P(X=1)$, and $\nu_c'(\{\infty\})=0$ if $c\leqslant P(X=1)$. In addition, if $P(X=1)=0$, then $\nu_c'(\{\infty\})=1$. 
\end{Theo}

Section 2 is devoted to proving Theorem \ref{theo4unilambda} and Theorem \ref{Theoexistinflambda}.

\begin{Theo}
\label{Theoconvlambda1}
	Suppose that $0<c<E[(1-(1-X)^{1/2})/X^2]$ and $E[1/X]<\infty$. If there exists $p<1$ such that
	\begin{equation}\label{conditionYL11}
n^{-p}E\left[\frac{1-(1-X)^n}{X^2}\right]\to 0,\text{ as }n\to\infty,
\end{equation}
then there exists a distribution $\bar{\nu}_c^*\in\mathcal{S}$ with infinite mean and $\bar{\nu}_c^*(\{\infty\})=0$ such that 
\begin{equation*}\label{theomueq1}
G_c(\bar{\nu}_c^*)=\bar{\nu}_c^*.
\end{equation*}
\end{Theo}

Recall from section \ref{chap1.2introlambda} that the beta coalescent is the $\Lambda$-coalescent in which $\Lambda(dx)$ 
is the Beta$(2-\alpha,\alpha)$ distribution with $0<\alpha<2$. For $\alpha\in(0,1)$, we have
\begin{align*}
&\int_0^1\frac{1-(1-x)^n}{x^{2}}x^{1-\alpha}(1-x)^{\alpha-1}dx\\
&\quad\quad\leqslant \left(\frac{1}{2}\right)^{\alpha-1}\int_0^{1/2}(1-(1-x)^n)x^{-1-\alpha}dx 	+ \left(\frac{1}{2}\right)^{-1-\alpha}\int_{1/2}^1(1-x)^{\alpha-1}dx\\
&\quad\quad =-\left(\frac{1}{2}\right)^{\alpha-1}\frac{1}{\alpha}\left[(1-(1-x)^n)x^{-\alpha})\right]_{0}^{1/2}+\left(\frac{1}{2}\right)^{\alpha-1}\frac{1}{\alpha}\int_0^{1/2}n(1-x)^{n-1}x^{-\alpha}dx+\frac{2}{\alpha}\\
&\quad\quad\leqslant \left(\frac{1}{2}\right)^{\alpha-1}\frac{1}{\alpha}\int_0^{1}n(1-x)^{n-1}x^{-\alpha}dx+\frac{2}{\alpha}\\
&\quad\quad = \left(\frac{1}{2}\right)^{\alpha-1}\frac{n}{\alpha}\frac{\Gamma(1-\alpha)\Gamma(n)}{\Gamma(n+1-\alpha)}+\frac{2}{\alpha}.
\end{align*}
It then follows as a consequence of Stirling's formula that
\begin{align*}
n^{-\alpha}\int_0^1\frac{1-(1-x)^n}{x^{2}}\Lambda(dx)=O(1),~\text{as } n\to\infty. 	
\end{align*}
Therefore, the condition (\ref{conditionYL11}) is satisfied for the beta coalescent with $0<\alpha<1$.

Section 3 is devoted to proving Theorem \ref{Theoconvlambda1}.

Theorem \ref{Theoexistinflambda} shows the convergence of the number of individual lineages belonging to each surviving species for certain initial conditions. Theorem \ref{Theoconvlambda1} indicates that the picture is even more complicated compared with the Yule-Kingman nested coalescent because under certain conditions, there exists another solution to (\ref{theoremequ1}) with no mass at infinity other than the solution in Theorem~\ref{theo4unilambda}.

\section{Existence and Uniqueness of the Solution on $\mathcal{S}_1$ to the RDE (\ref{theoremequ1})}
Recall from section \ref{chap1.2introlambda} that $\Lambda$ is a probability measure that uniquely characterizes the law of the $\Lambda$-coalescent and $X$ is a random variable on $[0,1]$ with law $P(X\in dx)=\Lambda(dx)$. Recall from section \ref{chap1.6introresultYL} that $L_n(t)$ is the number of lineages surviving to time $t$ in the $\Lambda$-coalescent started with $n$ lineages and $Y$ is an exponential random variable with rate $c$. 
\subsection{Existence of the Solution to the RDE}

\begin{Lemma}\label{lemma17uniquep1}
	Suppose $E[1/X]<\infty$. Let $b_i=E[L_{i+1}(Y)-L_i(Y)]$. Then 
	$$\lim_{i\to\infty}b_i=\frac{c}{E\left[\frac{1}{X}\right]+c}.$$
\end{Lemma}
\begin{proof} 
Consider the $\Lambda$-coalescent process $(\Pi_t,t\geqslant 0)$ starting from the partition of $\mathbb{N}^+$ into singletons. Let $T_{i}$ be the first time that the block containing the element $i$ merges with some block containing one of the elements $1,2,\ldots,i-1$. Then
\begin{equation}\label{b1restart}
b_i=P(Y<T_{i+1}).
\end{equation}
Let $T_i^*$ be the first time that the block containing $i$ is a non-singleton. Recall from section \ref{chap1.2introlambda} that $(\Pi_t,t\geqslant 0)$ can be constructed by a Poisson point process on $[0,\infty)\times[0,1]$ with intensity $dt\times x^{-2}\Lambda(dx)$. Note that $T_i^*$ is the first time that we get a head for block $\{i\}$ when flipping an $x$-coin. Therefore $T_i^*$ has an exponential distribution with rate 
$$\int_0^1x\cdot x^{-2}\Lambda(dx)=\int_0^1 x^{-1}\Lambda(dx)=E\left[\frac{1}{X}\right],$$
which implies that
\begin{equation}\label{rateofT*} 
T_i^*\sim \text{Exp}\left(E\left[\frac{1}{X}\right]\right).
\end{equation}
Let $\chi_{m}$ be the rate of mergers involving the block $B_{m+1}$ but not the blocks $B_1,\ldots,B_{m}$. Then 
$$\chi_{m}=\int_0^1x(1-x)^{m}\cdot x^{-2}\Lambda(dx)=\int_0^1 x^{-1}(1-x)^{m}\Lambda(dx).$$
Since $E[1/X]<\infty$, we have $\int_0^1 x^{-1}\Lambda(dx)<\infty$. Therefore, by the Dominated Convergence Theorem,\begin{equation}\label{chi11}
\chi_{m}\to0,~\text{as } m\to\infty.
\end{equation}
Let $\epsilon>0$. We can choose $t^*$ such that
\begin{equation}\label{yt*} 
P(Y\geqslant t^*)< \epsilon.
\end{equation}
By $(\ref{chi11})$, for a fixed $t^*$, we can choose $m^*>1$ such that 
\begin{equation}\label{chi111}
\chi_{m}<\frac{\epsilon}{t^*}, ~\forall m\geqslant m^*.
\end{equation}
Let $Z$ be the event that there is a merger before time $t^*$ that involves all blocks. Note that $Z$ is only possible when $\Lambda$ has an atom at 1, which is the same as $P(X=1)>0$. Let 
$$S(t^*)\coloneqq\lim_{n\to\infty}\frac{1}{n}\sum_{i=1}^n\bold{1}_{\{\{i\} \text{ is singleton at time } t\}},$$
where the existence of the limit follows from Theorem 2 in \cite{KINGMAN1982235}. 
Then 
\begin{equation}\label{ZcSt*1}
\{S(t^*)>0\}\subset Z^C. 
\end{equation}
Since $\int_0^1x^{-1}\Lambda(dx)=E[1/X]<\infty$, it then follows from Lemma 25 and Proposition 26 in \cite{pitman1999} that 
\begin{equation}\label{ZcSt*2}
P(S(t^*)>0|Z^C)=1.
\end{equation} 
Note that $L_n(t)$ is the number of blocks in $\Pi_t$ restricted to $[n]$. By the exchangeability of $\Pi_t$ restricted to $[n]$, we have
\begin{equation}\label{ZcSt*}
P(L_{i}(t^*)\geqslant m^*|S(t^*))\geqslant \sum_{n=m^*}^{i}\binom{i}{n}S(t^*)^n(1-S(t^*))^{i-n}.
\end{equation}
By $(\ref{ZcSt*1})$, $(\ref{ZcSt*2})$ and $(\ref{ZcSt*})$, we have
$$P(L_{i}(t^*)\geqslant m^*|Z^C)=P(L_{i}(t^*)\geqslant m^*|S(t^*)>0)\geqslant E\left[\sum_{n=m^*}^{i}\binom{i}{n}S(t^*)^n(1-S(t^*))^{i-n}\bigg|S(t^*)>0\right].$$
By the Dominated Convergence Theorem, we have
\begin{align*}
	\lim_{i\to\infty}&E\left[\sum_{n=m^*}^{i}\binom{i}{n}S(t^*)^n(1-S(t^*))^{i-n}\bigg|S(t^*)>0\right]\\
	&=1-\sum_{n=0}^{m^*-1}E\left[\lim_{i\to\infty}\binom{i}{n}S(t^*)^n(1-S(t^*))^{i-n}\bigg|S(t^*)>0\right]=1.
\end{align*}
Therefore, for $\epsilon>0$, we can choose $i^*$ such that 
\begin{equation}\label{Ln*111} 
P(L_{i}(t^*)< m^*|Z^C)<\epsilon,~\forall i>i^*.
\end{equation}
Note that
\begin{equation}\label{b1restart2}
T_{i+1}^*\leqslant T_{i+1}.
\end{equation}
By $(\ref{b1restart})$ and $(\ref{b1restart2})$, we have
\begin{align}\label{bbb111}
b_i&=P(Y<T_{i+1})\nonumber\\
&=P(Y<T_{i+1}^*)+P(T_{i+1}^*\leqslant Y<T_{i+1})\nonumber\\
&\leqslant P(Y<T_{i+1}^*)+ P(Y\geqslant t^*)+P(T_{i+1}^*<T_{i+1},T_{i+1}^*<t^*).
\end{align}
Note that $P(T_{i+1}^*<T_{i+1},T_{i+1}^*<t^*)$ increases as $t^*$ increases. Let $\bar{T}$ be the time of the first merger that involves all blocks, which is finite only if $\Lambda$ has an atom at 1. Note that 
the events $Z$ and $ \{\bar{T}<t^*\}$ agree almost surely.
Then
\begin{align}\label{sumZc}
	P(T_{i+1}^*<T_{i+1},T_{i+1}^*<t^*)&\leqslant P(T_{i+1}^*<T_{i+1},T_{i+1}^*<t^*|\bar{T}\geqslant t^*)P(\bar{T}\geqslant t^*)\nonumber\\
	&\leqslant P(T_{i+1}^*<T_{i+1},T_{i+1}^*<t^*|Z^C).
\end{align}
Let $A_t$ be the event that there is a merger involving the block containing the element $i+1$ that does not involve the blocks containing any of the elements $1,2,\ldots,i$ before time $t$.  
Recall from section \ref{chap1.2introlambda} that the rate of mergers involving the block $B_{m+1}$ but not the blocks $B_1,\ldots,B_m$ is $\chi_m$. 
Let $E_{m^*}\sim \text{Exp}(\chi_{m^*})$. If $L_{i}(t^*)\geqslant m^*$, then the number of blocks in $\Pi_t$ restricted to $[i]$ is no less than $m^*$ for all $t\leqslant t^*$, which implies the rate for the block containing the element $i+1$ to merge with the blocks containing any of the elements $1,2,\ldots,i$ is smaller than $\chi_{m^*}$ at any time $t\in[0,t^*]$.
Then 
\begin{equation}\label{chi3333}
P(A_{t^*}|L_{i}(t^*)\geqslant m^*)\leqslant P(E_{m^*}<t^*).
\end{equation}
By $(\ref{chi111})$, we have 
\begin{equation}\label{chi2222}
	 P(E_{m^*}<t^*)<\frac{\epsilon}{t^*}\cdot t^*=\epsilon.
\end{equation} 
Note that
\begin{align*}\label{AtZc}
	P&(T_{i+1}^*<T_{i+1},T_{i+1}^*<t^*|Z^C)\\
	&=P(T_{i+1}^*<T_{i+1},T_{i+1}^*<t^*, L_{i}(t^*)\geqslant m^*|Z^C)+P(T_{i+1}^*<T_{i+1},T_{i+1}^*<t^*, L_{i}(t^*)< m^*|Z^C)\nonumber\\
	&\leqslant P(A_{t^*}|\{L_{i}(t^*)\geqslant m^*\}\cap Z^C)P(L_{i}(t^*)\geqslant m^*|Z^C)+P(L_{i}(t^*)< m^*|Z^C).
\end{align*}
Since $m^*>1$, we have 
\begin{equation*}\label{subset123}
\{L_{i}(t^*)\geqslant m^*\}\subset Z^C.
\end{equation*}
Then
\begin{align}\label{AtZc222}
	P(T_{i+1}^*<T_{i+1},T_{i+1}^*<t^*|Z^C)&\leqslant P(A_{t^*}|L_{i}(t^*)\geqslant m^*)+P(L_{i}(t^*)< m^*|Z^C).
\end{align}
Therefore, by $(\ref{bbb111})$, $(\ref{sumZc})$, $(\ref{chi3333})$ and $(\ref{AtZc222})$, we have 
\begin{align}\label{biupperbound1}
	b_i\leqslant P(Y<T_{i+1}^*)+ P(Y\geqslant t^*)+ P(E_{m^*}<t^*)+P(L_{i}(t^*)< m^*|Z^C).
\end{align}
It then follows from $(\ref{yt*})$, $(\ref{Ln*111})$, $(\ref{chi2222})$ and $(\ref{biupperbound1})$ that for all $\epsilon>0$, we can choose $i^*$ large enough such that 
\begin{align*}
	b_i\leqslant P(Y<T_{i+1}^*)+3\epsilon,~\forall i>i^*.
\end{align*}
By $(\ref{b1restart})$, $(\ref{rateofT*})$ and $(\ref{b1restart2})$, we have
\begin{equation}\label{bilowerbound}
	b_i\geqslant P(Y<T_{i+1}^*)=\frac{c}{E\left[\frac{1}{X}\right]+c}, ~\forall i\in\mathbb{N}^+.
\end{equation}
By the squeeze theorem, we have
\begin{equation*}\label{bilimit111}
	\lim_{i\to\infty}b_i=P(Y<T_{i+1}^*)=\frac{c}{E\left[\frac{1}{X}\right]+c}.
\end{equation*}
\end{proof}

\begin{Lemma}\label{lemma18nup1}
	The limit $\lim_{n\to\infty}G_c^n(\delta_1)$ exists. Let $\nu_c^*=\lim_{n\to\infty}G_c^n(\delta_1)$. Then $\nu_c^*$ is a solution to (\ref{theoremequ1}). If $c<E[1/X]<\infty$, then 
	$$\nu_c^*\in\mathcal{S}_1.$$ 
	If $c>E[1/X]$, then $\nu_c^*$ has infinite mean.
\end{Lemma}
\begin{proof}
	Since the expression in $(\ref{dist111})$ increases as $W_1$ and $W_2$ increase, we have that $G_c$ is monotone. It follows that the sequence of iterates $\{G_c^n(\delta_1)\}_{n=1}^\infty$ is increasing and then there exists a distribution $\nu_c^*$ on $\mathbb{N}^+\cup\{\infty\}$ such that
$$G_c^n(\delta_1)\to_d\nu_c^*,\text{ as } n\to\infty.$$

Note that for all $k\in\mathbb{N}^+$,
\begin{align}\label{lemma18proof111}
	\left(G_c(\nu_c^*)\right)(\{k\})&=\sum_{i=k}^\infty \left(\nu_c^* \ast \nu_c^*\right)(\{i\})P(L_i(Y)=k)\nonumber\\
	&= \lim_{n\to\infty}\sum_{i=k}^\infty\left(\left(\nu_c^* \ast \nu_c^*\right)(\{i\})-\left(G_c^{n-1}(\delta_1) \ast G_c^{n-1}(\delta_1)\right)(\{i\})\right)P(L_i(Y)=k)\nonumber\\
	&\quad\quad\quad\quad+\lim_{n\to\infty}\sum_{i=k}^\infty\left(G_c^{n-1}(\delta_1) \ast G_c^{n-1}(\delta_1)\right)(\{i\})P(L_i(Y)=k). 
\end{align}
By Scheffe's theorem and the Dominated Convergence Theorem, the first term on the right-hand side of $(\ref{lemma18proof111})$ is 0. Note that the second term on the right-hand side of $(\ref{lemma18proof111})$ is 
$\nu_c^*(\{k\})$, which implies that $\nu_c^*=G_c(\nu_c^*)$.

Let $X_{n}$ be a random variable with distribution $G_c^{n}(\delta_1)$.  
It follows from the Monotone Convergence Theorem that
\begin{equation}\label{limitofnu*123}
	\int_0^\infty \nu_c^*[x,\infty]dx=\lim_{n\to\infty}E[X_n].
\end{equation}
If $c<E[1/X]$, then $\lim_{i\to\infty}b_i<1/2$ by Lemma \ref{lemma17uniquep1}. It follows that there exist $\delta>0$ and $I\in\mathbb{N}^+$ such that
\begin{equation*}\label{M1}
	b_i<\frac{1}{2}-\delta, ~\forall i\geqslant I.
\end{equation*} 
Let $\bar{X}_{n}$ be a random variable with distribution $G_c^n(\delta_1)$ independent of $X_n$. Then 
\begin{align}
E[X_{n+1}]=E[L_{X_{n}+\bar{X}_{n}}(Y)]=\sum_{i=2}^\infty P(X_{n}+\bar{X}_{n}=i)E[L_{i}(Y)].\label{sumofL}
\end{align}
Note that $L_i(Y)\leqslant i$. Then
\begin{align}
\sum_{i=2}^{I} P(X_{n}+\bar{X}_{n}=i)E[L_{i}(Y)]\leqslant \sum_{i=2}^{I} iP(X_{n}+\bar{X}_{n}=i)\leqslant I,\label{sumofL1}
\end{align}
and 
\begin{align}
	\sum_{i=I+1}^\infty P(X_{n}+\bar{X}_{n}=i)E[L_{i}(Y)]=&\sum_{i=I+1}^\infty P(X_{n}+\bar{X}_{n}=i)\left(E[L_{I}(Y)]+\sum_{n=1}^{i-I}b_{I+n-1}\right)\nonumber\\
< &\sum_{i=I+1}^\infty P(X_{n}+\bar{X}_{n}=i)\left(I+\left(\frac{1}{2}-\delta\right)(i-I)\right)\nonumber\\
<&\left(\frac{1}{2}+\delta\right)I+(1-2\delta)E[X_{n}].\label{sumofL2}
\end{align}
By $(\ref{sumofL})$, $(\ref{sumofL1})$ and $(\ref{sumofL2})$, we have
\begin{equation*}\label{sumofL3}
	E[X_{n+1}]<\left(\frac{3}{2}+\delta\right)I+(1-2\delta)E[X_{n}].
\end{equation*}
It follows that the monotonically increasing sequence $\{E[X_n]\}_{n=1}^\infty$ has an upper bound, which implies that the limit $\lim_{n\to\infty}E[X_n]$ is finite. Therefore, by $(\ref{limitofnu*123})$, when $c<E[1/X]$, we have
$$\nu_c^*\in\mathcal{S}_1.$$
Recall that $b_i \geqslant c/(c+E[1/X])$ for all $i\in\mathbb{N}^+$ by $(\ref{bilowerbound})$. Then
\begin{align}
	E[X_{n+1}]=\sum_{i=2}^\infty P(X_{n}+\bar{X}_{n}=i)E[L_{i}(Y)]
	&=\sum_{i=2}^\infty P(X_{n}+\bar{X}_{n}=i)\left(1+\sum_{n=2}^i b_{n-1}\right)\nonumber\\
	&\geqslant1+\sum_{i=2}^\infty P(X_{n}+\bar{X}_{n}=i)\frac{c(i-1)}{c+E\left[\frac{1}{X}\right]}\nonumber\\
	&=\frac{2c}{c+E\left[\frac{1}{X}\right]}E[X_n]+1-\frac{c}{c+E\left[\frac{1}{X}\right]}.\label{larger1}
\end{align}
If $c>E[1/X]$, then $c/(c+E[1/X])>1/2$. Therefore, by $(\ref{limitofnu*123})$ and $(\ref{larger1})$, we have
that $\nu_c^*$ has infinite mean when $c>E[1/X]$. 
\end{proof}

\subsection{Uniqueness of the Solution to the RDE}

\begin{Lemma}
\label{lemmadistLTY}
	Let $(\Pi_t,t\geqslant 0)$ be a $\Lambda$-coalescent and let $(\Pi^n_t,t\geqslant 0)$ be its restriction to $[n]$. Let $L_n(t)$ denote the number of blocks in $\Pi^n_t$. Let $Y\sim \text{Exp}(c)$, where $c$ is some positive constant. Suppose $T$ is a random variable on $\mathbb{N}^+$, and $T$, $Y$, $\Pi$ are independent. Then
\begin{align*}
P(T=n)=\frac{\lambda_n+c}{c}P(L_T(Y)=n)  -\frac{1}{c}\sum_{i=n+1}^{\infty}  \binom{i}{i-n+1}\lambda_{i, i-n+1}P(L_T(Y)=i), ~\forall n\in\mathbb{N}^+.
\end{align*}
\end{Lemma}
\begin{proof}
Note that 
\begin{align}\label{LTYSplit}
P(L_{T}(Y)=n)=P(L_{T}(Y)=n,T\geqslant n+1)+P(L_{T}(Y)=n,T=n).
\end{align}
Let $A_{i,n}$ be the event that starting from $T$ blocks, the last merger before time $Y$ makes the number of blocks decrease from $i$ to $n$. Then 
\begin{equation}\label{LTYW12i1}
	\{L_{T}(Y)=n\}\cap\{T\geqslant n+1\}=\bigcup_{i=n+1}^\infty\left(\{L_{T}(Y)=n\}\cap A_{i,n}\right).
\end{equation}
Recall from (\ref{lambda11}) and (\ref{lambda22}) that the rate for each $k$-tuple of blocks to merge when there are $b$ blocks in total is $\lambda_{b,k}$ and that the rate at which mergers happen when there are $b$ blocks in total is $\lambda_b$.
Once the number of blocks reaches $i$, because of the memoryless property of exponential random variables, the probability for the number of blocks to stay at $i$ until time $Y$ is $c/(\lambda_i+c)$. 
Let $B_i$ be the event that the number of blocks reaches $i$ before time $Y$. Then
\begin{align}\label{LTYLijY}
	P(L_T(Y)=i)=\frac{c}{\lambda_i+c}P(B_i).
\end{align}
 Since the transition rate from $i$ to $n$ is $\binom{i}{i-n+1}\lambda_{i,i-n+1},$ the number of blocks jumps to $n$ with probability $\binom{i}{i-n+1}\lambda_{i,i-n+1}/(\lambda_i+c).$  
Therefore
\begin{align}\label{LTYLij1}
	P(\{L_{T}(Y)=n\}\cap A_{i,n} )&=P(B_i)\frac{\binom{i}{i-n+1}\lambda_{i,i-n+1}}{\lambda_i+c}\frac{c}{\lambda_n+c}.
\end{align}
For the first probability in $(\ref{LTYSplit})$, by $(\ref{LTYW12i1})$, $(\ref{LTYLijY})$ and $(\ref{LTYLij1})$, we have 
\begin{align}\label{LTYW12i} 
P(L_{T}(Y)=n,T\geqslant n+1)&=\sum_{i=n+1}^\infty P(\{L_{T}(Y)=n\}\cap A_{i,n} )\nonumber\\
&=\sum_{i=n+1}^\infty \binom{i}{i-n+1}\frac{\lambda_{i,i-n+1}}{\lambda_n+c}P(L_{T}(Y)=i).
\end{align}
For the second probability in $(\ref{LTYSplit})$, we have
\begin{align}\label{LTYW12i2}
	P(L_{T}(Y)=n,T=n)=P(L_n(Y)=n)P(T=n)=\frac{c}{\lambda_n+c}P(T=n).
\end{align}
The result follows from  
$(\ref{LTYSplit})$, $(\ref{LTYW12i})$ and $(\ref{LTYW12i2})$. 
\end{proof}

\begin{Lemma}\label{lemmapgflambda}
	Let $\nu_c^*\in\mathcal{S}$ be a distribution that satisfies the recursive distributional equation (\ref{theoremequ1}) $G_c(\nu_c^*)=\nu_c^*$. 
Let $R_c$ be the probability generating function of $\nu^*_c$. Then $R_c(x)$ satisfies
\begin{equation}\label{PGF}
E\left[\frac{R_c(x)-R_c((1-X)x)}{X^2}\right]+c(R_c(x)-R_c^2(x))=xE\left[\frac{R_c(X+(1-X)x)-R_c((1-X)x)}{X^2}\right],
\end{equation}
for all $x\in[0,1]$.
If $E[1/X]<\infty$ and $\nu_c^*\in\mathcal{S}_1$, then both sides of $(\ref{PGF})$ are finite for all $x\in[0,1]$.
\end{Lemma}
\begin{proof}
Let $W$, $W_1$ and $W_2$ be independent random variables with distribution $\nu_c^*$. We have $W=_{d}L_{W_{1}+W_{2}}(Y)$ because $G_c(\nu_c^*)=\nu_c^*$. By Lemma \ref{lemmadistLTY}, we have 
\begin{align*}
P(W_1+W_2=n)=\frac{\lambda_n+c}{c}P(W=n)  -\frac{1}{c}\sum_{i=n+1}^{\infty}  \binom{i}{i-n+1}\lambda_{i, i-n+1}P(W=i), ~\forall n\in\mathbb{N}^+.
\end{align*}
Note that $P(W_1+W_2=1)=0$ because $\nu_c^*$ is a distribution on $\mathbb{N}^+$. It then follows from $(\ref{lambda11})$, $(\ref{lambda22})$, $(\ref{lambda33})$ and the equation above that
\begin{align*}
	\Bigg(\sum_{k=2}^{n}&\binom{n}{k}E\left[X^{k-2}(1-X)^{n-k}\right]+c\Bigg)P(W=n)\nonumber\\
	&=\sum_{i=n+1}^\infty P(W=i)\binom{i}{i-n+1}E\left[X^{i-n-1}(1-X)^{n-1}\right]+cP(W_1+W_2=n), ~i\geqslant 2,
\end{align*}
and 
\begin{align*}
	cP(W=1)=\sum_{i=2}^\infty P(W=i)E[X^{i-2}].
\end{align*}
By multiplying both sides by $x^n$ and summing over $n$, we have
\begin{align}\label{equaa2}
	&\sum_{n=2}^{\infty}P(W=n)x^n\sum_{k=2}^{n}\binom{n}{k}E\left[X^{k-2}(1-X)^{n-k}\right]+\sum_{n=1}^\infty cP(W=n)x^n\nonumber\\
	&~~=\sum_{n=1}^\infty cP(W_1+W_2=n)x^n+\sum_{n=1}^{\infty}x^n\sum_{i=n+1}^\infty P(W=i)\binom{i}{i-n+1}E\left[X^{i-n-1}(1-X)^{n-1}\right].
\end{align}
Note that 
\begin{equation}\label{equaa3}
	\sum_{n=1}^\infty cP(W=n)x^n=cR_c(x),\text{ and }\sum_{n=1}^\infty cP(W_1+W_2=n)x^n=cR_c^2(x).
\end{equation}
By the Monotone Convergence Theorem, we have
\begin{align}\label{equaa5}
	\sum_{n=2}^{\infty}P(&W=n)x^n\sum_{k=2}^{n}\binom{n}{k}E\left[X^{k-2}(1-X)^{n-k}\right]\nonumber\\
	=&E\left[\sum_{n=2}^{\infty}P(W=n)x^n\sum_{k=2}^{n}\binom{n}{k}X^{k-2}(1-X)^{n-k}\right]\nonumber\\
	=&E\left[\sum_{n=2}^{\infty}P(W=n)x^n\frac{1-(1-X)^n-nX(1-X)^{n-1}}{X^2}\right]\nonumber\\
	=&E\left[\sum_{n=1}^{\infty}P(W=n)x^n\frac{1-(1-X)^n-nX(1-X)^{n-1}}{X^2}\right]\nonumber\\
	=&E\left[\frac{R_c(x)-R_c((1-X)x)-\sum_{n=1}^{\infty}nP(W=n)x^nX(1-X)^{n-1}}{X^2}\right],
\end{align}
and
\begin{align}\label{equaa6}
	\sum_{n=1}^{\infty}x^n&\sum_{i=n+1}^\infty P(W=i)\binom{i}{i-n+1}E\left[X^{i-n-1}(1-X)^{n-1}\right]\nonumber\\
	=&\sum_{i=2}^{\infty}P(W=i)x\sum_{n=1}^{i-1}\binom{i}{i-n+1}E\left[X^{i-n-1}(1-X)^{n-1}\right]x^{n-1}\nonumber\\
	=&E\left[\sum_{i=2}^{\infty}P(W=i)x\frac{(X+(1-X)x)^i-((1-X)x)^i-iX((1-X)x)^{i-1}}{X^2}\right]\nonumber\\
	=&E\left[\sum_{i=1}^{\infty}P(W=i)x\frac{(X+(1-X)x)^i-((1-X)x)^i-iX((1-X)x)^{i-1}}{X^2}\right]\nonumber\\
	=&E\left[\frac{xR_c(X+(1-X)x)-xR_c((1-X)x)-\sum_{i=1}^{\infty}iP(W=i)x^iX(1-X)^{i-1}}{X^2}\right].
\end{align}
The result (\ref{PGF}) follows from $(\ref{equaa2})$, $(\ref{equaa3})$, $(\ref{equaa5})$ and $(\ref{equaa6})$.

Since probability generating functions have nonnegative second order derivatives on $[0,1]$ and $X$ is a random variable on $[0,1]$, we have 
$$0\leqslant R_c(x)-R_c((1-X)x)\leqslant xXR_c'(x), ~\forall x\in[0,1].$$
Note that $R_c'(1)$ equals the mean of $W$. Since $\mu^*\in\mathcal{S}_1$, we have $R_c'(1)<\infty$, which implies $R_c'(x)<\infty$ for all $x\in[0,1]$. If $E[1/X]<\infty$, then 
\begin{equation*}
	E\left[\frac{R_c(x)-R_c((1-X)x)}{X^2}\right]\leqslant xR_c'(x)E\left[\frac{1}{X}\right]<\infty, ~\forall x\in[0,1].
\end{equation*}
Therefore, both sides of $(\ref{PGF})$ are finite for all $x\in[0,1]$ when $E[1/X]<\infty$.
\end{proof}

Let $\{c_i\}_{i=1}^\infty$ be a sequence of constants. Let $G_{c_i}:\mathcal{S}\to\mathcal{S}$ be the mapping defined such that $G_{c_i}(\mu)$ is the distribution of 
\begin{equation*}\label{dist22}
L_{W_{1}+W_{2}}(Y_i),
\end{equation*}
where $W_1$ and $W_2$ have distribution $\mu$, $Y_i$ is exponentially distributed with rate $c_i$, and the random variables $Y_i$, $W_{1}$, $W_{2}$ and the $\Lambda$-coalescent process are independent. Let $\nu^*_{c_i}$ denote the distribution $\lim_{n\to\infty}G^n_{c_i}(\delta_1)$. Let $R_{c_i}$ be the probability generating function of $\lim_{n\to\infty}G_{c_i}^n(\delta_1)$.

\begin{Lemma}\label{lemma21now}
If $c_1<c_2<E[1/X]$, then $\nu^*_{c_1}\preceq\nu^*_{c_2}$ and $R_{c_1}(x)>R_{c_2}(x)$ for all $x\in(0,1)$.
\end{Lemma}
\begin{proof}
Since $c_1<c_2$, there exist $Y_1\sim \text{Exp}(c_1)$, $Y_2\sim \text{Exp}(c_2)$ and $Z\geqslant 0$ such that $Y_1=Y_2+Z$. Therefore $L_{T}(Y_1)=L_{T}(Y_2+Z)\preceq L_{T}(Y_2)$ for any positive integer-valued random variable $T$, which implies that 
\begin{equation}\label{lemma21gc1c2}
G_{c_1}(\mu)\preceq	G_{c_2}(\mu),~\forall \mu\in\mathcal{S}.
\end{equation}
Assume that $G_{c_1}^n(\delta_1)\preceq	G_{c_2}^n(\delta_1).$ Since $G_{c_2}$ is monotone, we have $G_{c_2}(G_{c_1}^{n}(\delta_1))\preceq G_{c_2}(G^{n}_{c_2}(\delta_1))$. By (\ref{lemma21gc1c2}), we have $G_{c_1}(G_{c_1}^{n}(\delta_1))\preceq G_{c_2}(G^{n}_{c_1}(\delta_1))$. Therefore,
$$G_{c_1}^{n+1}(\delta_1)\preceq	G_{c_2}^{n+1}(\delta_1).$$
By induction, we have $G_{c_1}^{n}(\delta_1)\preceq G_{c_2}^{n}(\delta_1)$ for all $n\in\mathbb{N}^+$. Then 
$$\nu^*_{c_1}=\lim_{n\to\infty}G_{c_1}^{n}(\delta_1)\preceq\lim_{n\to\infty}G_{c_2}^{n}(\delta_1)=\nu^*_{c_2}.$$ 
If $\nu^*_{c_1}=\nu^*_{c_2}$, then $R_{c_1}(x)=R_{c_2}(x)$ for all $x\in[0,1]$. By Lemma \ref{lemmapgflambda}, we have
$$c_1(R_{c_1}(x)-R_{c_1}^2(x))=c_2(R_{c_2}(x)-R_{c_2}^2(x))=c_2(R_{c_1}(x)-R_{c_1}^2(x)),~\forall x\in[0,1],$$
which implies that $R_{c_1}(x)\in\{0,1\}$ for all $x\in[0,1]$. This is not possible because $\nu^*_{c_1}\in\mathcal{S}_1$ by Lemma~\ref{lemma18nup1}. Therefore, $\nu^*_{c_2}$ and $\nu^*_{c_1}$ are not the same. It follows that 
\begin{align*}\label{strictlylarger2}
R_{c_1}(x)&=\sum_{j=1}^{\infty}\nu_{c_1}^*(\{j\})x^j\nonumber\\
&=x-(1-x)\sum_{j=2}^\infty\nu_{c_1}^*[j,\infty)x^{j-1}\nonumber\\
&>x-(1-x)\sum_{j=2}^\infty\nu_{c_2}^*[j,\infty)x^{j-1}=R_{c_2}(x), ~\forall x\in(0,1).
\end{align*}
\end{proof}

Let $G_0:\mathcal{S}\to\mathcal{S}$ be the mapping defined such that $G_0(\mu)=\delta_1$ for all $\mu\in\mathcal{S}$. 

\begin{Lemma}\label{lemma22diff1}
	Let $0\leqslant c<E[1/X]$. If $\tilde{\nu}_c^*\in\mathcal{S}_1$ is a distribution that satisfies $\tilde{\nu}_c^*=G_c(\tilde{\nu}_c^*)$, then limits $\lim_{n\to\infty}G_{c_1}^n(\tilde{\nu}_c^*)$ and $\lim_{n\to\infty}G_{c_2}^n(\tilde{\nu}_c^*)$ exist in $\mathcal{S}$ for all $c_1, c_2\in(0,E[1/X])$. Let $\tilde{\nu}^*_{c_1}=\lim_{n\to\infty}G_{c_1}^n(\tilde{\nu}_c^*)$ and $\tilde{\nu}^*_{c_2}=\lim_{n\to\infty}G_{c_2}^n(\tilde{\nu}_c^*)$. Let $m_1^*$ and $m_2^*$ be the means of $\tilde{\nu}^*_{c_1}$ and $\tilde{\nu}^*_{c_2}$ respectively. For all $\epsilon>0$, there exists a positive constant $K_{\epsilon}$ such that if $c\in[0,E[1/X]-\epsilon]$ and $c_1, c_2\in[\epsilon,E[1/X]-\epsilon]$, then 
	$$|m_2^*-m_1^*|\leqslant K_{\epsilon}|c_2-c_1|.$$
\end{Lemma}
\begin{proof}
Without loss of generality, suppose $0<c_1<c_2<E[1/X]$. It follows from $(\ref{lemma21gc1c2})$ that $ G_{c_i}(\tilde{\nu}_c^*)\succeq  G_{c}(\tilde{\nu}_c^*)=\tilde{\nu}_c^*$ when $c_i\geqslant c$ and $ G_{c_i}(\tilde{\nu}_c^*)\preceq  G_{c}(\tilde{\nu}_c^*)=\tilde{\nu}_c^*$ when $c\geqslant c_i$ for $i=1,2$.
Since $G_{c_i}$ is monotone, the sequence of iterates $\{G_{c_i}(\tilde{\nu}_c^*)\}_{n=1}^\infty$ is monotone. Therefore, there exist distributions $\tilde{\nu}_{c_i}^*$ on $\mathbb{N}^+\cup\{\infty\}$ such that for $i=1,2$,
\begin{equation}\label{glimitlemma5}
G_{c_i}^n(\tilde{\nu}_c^*)\to_d\tilde{\nu}_{c_i}^*,~\text{as }n\to\infty.
\end{equation}

Consider a binary tree with depth $n$. The root node of the tree has only one child, the leaf nodes of the tree have no children, and every other node of the tree has two children. Each child node is connected to its parent node by one edge. We will view this binary tree as a species tree, and we will keep track of the coalescence of gene lineages within each species. Let the number of genes at each leaf node be independent and identically distributed with distribution $\tilde{\nu}_c^*\ast\tilde{\nu}_c^*$, which is the convolution of $\tilde{\nu}_c^*$ with itself. We define the root of the tree to be at level 0, and the leaves of the tree to be at level $n$.

We consider two processes. In Model 1, the edge lengths are independent and exponentially distributed with parameter $c_1$. In Model 2, the edge lengths are independent and exponentially distributed with parameter $c_2$. For both models, the gene lineages at a child node will follow the $\Lambda$-coalescent along the edge connected to its parent node. When reaching a parent node, the remaining gene lineages from its two children will be added together. Therefore, the distributions of the numbers of gene lineages at the root nodes of Model 1 and Model 2 are $G_{c_1}^n(\tilde{\nu}_c^*)$ and $G_{c_2}^n(\tilde{\nu}_c^*)$ respectively.

We can couple these two processes. For each edge in the same position in Model 1 and Model 2, we choose an $\text{Exp}(c_2)$ random variable, which will be the edge length in Model 2. With probability $c_1/c_2$, the edge length in Model 1 will be the same. With probability $1-c_1/c_2$, the edge length in Model 1 will be the edge length in Model 2 plus another exponential random variable with rate $c_1$. We mark the edges that have different lengths in Model 1 and Model 2, and color the part of the marked edges where Model 1 is longer than Model 2 red in Model 1. Therefore, each unmarked edge and the uncolored part of each marked edge in Model 1 and Model 2 have the same length. 

\begin{figure}[ht]
\centering
\captionsetup{width=.95\linewidth,justification=centering} \includegraphics[width = 120mm]{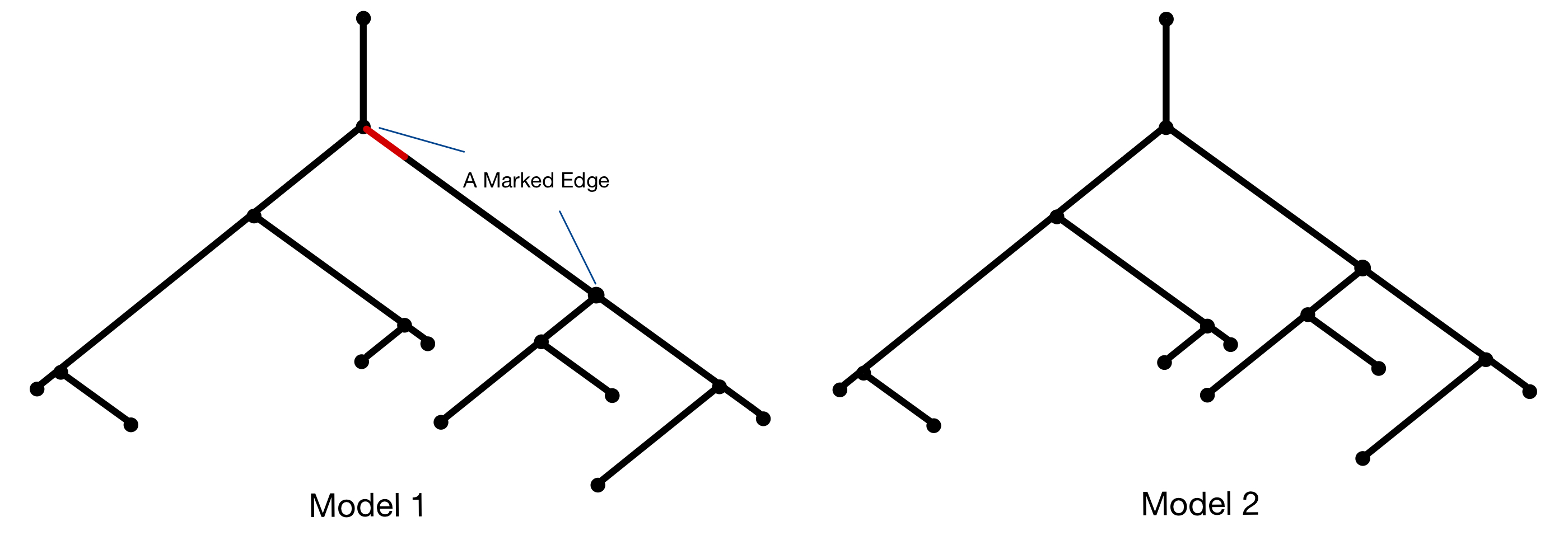}
\caption{Two binary species trees representing Model 1 and Model 2.}
\label{fig.5}
\end{figure}

Recall that the gene lineages in Model 1 follow the $\Lambda$-coalescent along each edge. If some gene lineages are involved in a merger that happens in the colored part of an edge in Model~1, we mark all but one of those gene lineages in Model 2. If some gene lineages are involved in a merger that happens in an unmarked edge or in the uncolored part of a marked edge in Model 1, we let the same number of unmarked gene lineages merge together in the same edge in Model 2. In this way, the unmarked gene lineages in Model 2 also follow the $\Lambda$-coalescent along each edge. Note that marked gene lineages might be involved in mergers to ensure that the gene lineages in Model 2 also follow the $\Lambda$-coalescent along each edge. 

Whenever a marked gene lineage participates in a merger of gene lineages that includes unmarked gene lineages, the mark will be eliminated. Recall that at a parent node at level $k$, the numbers of the gene lineages from the two edges connected to its children at level $k+1$ will be added up. Let a point at level $k+$ denote the point of an edge immediately before two edges merge. Note that there are $2^k$ nodes at level $k$ and $2^{k+1}$ points at level $k+$ for all $k\in\{1,\ldots,n-1\}$.
We can consider the distribution of the number of gene lineages at each point at level $k+$. Note that in Model 1, the number of gene lineages at each point at level $k+$ has distribution $G^{n-k}_{c_1}(\tilde{\nu}_c^*)$. In Model 2, the number of unmarked gene lineages at each point at level~$k+$ has distribution $G^{n-k}_{c_1}(\tilde{\nu}_c^*)$, and the total number of gene lineages at each point at level~$k+$ has distribution $G^{n-k}_{c_2}(\tilde{\nu}_c^*)$. 
 
Let $c_3=E[1/X]-\epsilon$, where $\epsilon$ is chosen to be small enough such that $c,c_1,c_2\leqslant c_3$ and $c_1,c_2>\epsilon$. By a similar argument as the one in (\ref{glimitlemma5}), there exists a distribution $\tilde{\nu}_{c_3}^*$ on $\mathbb{N}^+\cup\{\infty\}$ such that $G_{c_i}^n(\tilde{\nu}_c^*)\to_d\tilde{\nu}_{c_i}^*,~\text{as }n\to\infty$.
Let $m^*$ be mean of $\tilde{\nu}_c^*$. Let $m_{i,n}$ be the mean of $G_{c_i}^n(\tilde{\nu}_c^*)$ for $n\geqslant 1$. Let $m_{i,0}=m^*$. Then the sequence $\{m_{i,n}\}_{n=0}^\infty$ is monotone. Let $m_i^*$ be the mean of $\tilde{\nu}_{c_i}^*$. By the Monotone Convergence Theorem, we have
\begin{equation}\label{limm123}
	\lim_{n\to\infty}m_{i,n}= m_i^*, i=1,2,3.
\end{equation}
Since $c,c_2\leqslant c_3$, we have $m_{2,n}\leqslant \max\{m_2^*,m^*\}\leqslant m_3^*$. 
It follows that the expected number of total gene lineages at each point at level $k+$ is bounded above by $m_3^*$ for all $k\in\{0,\ldots,n-1\}$ in Model~2. Therefore, the expected number of marked gene lineages at each node is bounded above by $2m_3^*$. Conditional on having $i$ marked edges between level $k$ and level $k+1$, the expected number of gene lineages that get marked between level $k$ and level $k+1$ is bounded above by $2im_3^*$. 
Consider the case that the number of gene lineages at each point at level $k$ is 2 and all gene lineages at level~$k$ are unmarked except a randomly chosen one. Then there is one marked gene lineage at level $k$. Let $r_k$ be the probability that this marked gene lineage remains marked at the root of the tree. Note that the probability for a marked gene lineage at level $k$ to survive to the root decreases when the number of unmarked gene lineages at level $k$ increases and the lower bound of the number of unmarked gene lineages at each point at level $k$ is 2. Note that the number of marked gene lineages also decreases when a merger that only involves marked gene lineages happens. Therefore, if there are $j$ marked gene lineages at level $k$, then the expected number of marked gene lineages that could survive to the root of the tree is bounded above by $jr_k$. Let $M_k$ be the number of marked edges between level $k$ and level $k+1$. Then $M_k\sim\text{Binomial}(2^k,1-c_1/c_2)$ and
$$E[M_k]=2^{k}\frac{c_2-c_1}{c_2}.$$
Let $l_k$ be the expected number of the gene lineages that get marked between level $k$ and level~$k+1$ and remain marked at the root of the tree in Model 2. Then
$$l_k\leqslant 2E[M_k]m_3^*r_k.$$
Note that $m_{2,n}-m_{1,n}$ is the expected number of marked gene lineages at the root of the tree in Model 2.  Therefore,
\begin{align}\label{EX123321}
m_{2,n}-m_{1,n}=\sum_{k=0}^{n-1}l_k\leqslant\frac{2(c_2-c_1)m_3^*}{c_2}\sum_{k=0}^{n-1}2^{k}r_k.
\end{align}	

To bound $r_k$, we will introduce three new models, which will be called Model 3, Model~4 and Model 5.  
In these three models, we consider one binary tree which is mostly the same as the binary tree in Model 2, but it has depth $k$ instead of depth $n$. That is, the edge lengths are independently and exponentially distributed with rate $c_2$. Let the number of gene lineages at each leaf node be 2.
The gene lineages follow the $\Lambda$-coalescent along each edge and add together once reaching a parent node. Randomly choose one leaf node and mark one of the gene lineage at that leaf node.

Model 3 and Model 4 are different only in the way that we treat the mark.
In Model~3, when the marked gene lineage participates in a merger, the mark is eliminated. In Model~4, when the marked gene lineage participates in a merger of $l$ gene lineages, the mark is retained with probability $1/l$. Model 5 is the same as Model 4, except that any merger that involves the marked gene lineage but does not cause the mark to be lost gets suppressed altogether. Let $A_i$ be the event that the mark survives to reach the root in Model $i$. 
Note that
\begin{equation}\label{rkA3}
r_k= P(A_3).
\end{equation}
In Model 4, we can consider the gene lineages as labeled with ``marked" or ``unmarked". When gene lineages merge, we choose one of the merging gene lineages at random to have its label preserved. Then each gene lineage at depth $k$ is equally likely to have its label preserved at the root. Since there are $2^k$ leaves in total, if all gene lineages at one leaf node chosen at random are marked, 
the expected number of marked gene lineages at the root is $m_{2,k}/2^{k}$. 
Since in Model 4, only one gene lineage at one leaf node is marked, we have
 \begin{equation}\label{A4111}
P(A_4)\leqslant \frac{m_{2,k}}{2^{k}}\leqslant \frac{\max\{m_2^*,m^*\}}{2^k}\leqslant \frac{m_3^*}{2^k}.
\end{equation}

If the marked gene lineage survives to reach the root in Model 3, then there is no merger that involves the marked gene lineage. It follows that the marked gene lineage also survives to reach the root in Models 4 and 5. If the mark survives in Model 5, then there is either no merger that involves the marked gene lineage or all mergers that involve the marked gene lineage do not cause the mark to be lost, which implies that the mark also survives in Model 4. Therefore
\begin{equation*}\label{A345111}
A_3\subset A_5\subset A_4.
\end{equation*}
Let $T$ be the distance from the leaf that has the marked gene lineage to the root. Since the length of each edge has an $\text{Exp}(c_2)$ distribution, we have $T\sim$ Gamma$(k,c_2)$. 
Let $N(t)$ be the number of gene lineages on the same branch as the marked gene lineage after time $t$ in Model 5. Let $\mathcal{F}$ be the $\sigma-$algebra consisting of all information in Model 5. Let $f(n)$ be the rate, when there are $n$ gene lineages, of mergers that include the marked gene lineage but do not cause the mark to be lost in Model 4. Note that these mergers are not visible in Model 5. It is easy to see that $f(n)$ increases as $n$ increases. Therefore the rate of mergers that get suppressed in Model 5 at time $t$ is $f(N(t))$. If the mark survives in Model 5, then it does not survive in Model 3 if there exist some mergers involving the marked gene lineage, and all mergers involving the marked gene lineage do not cause the mark to be lost. Therefore,
\begin{equation}\label{A3111}
P(A_3| \mathcal{F})\bold{1}_{A_5}=\left(1-P(A_3^C|\mathcal{F})\right)\bold{1}_{A_5}=\exp\left(-\int_0^Tf(N(t))dt\right)\bold{1}_{A_5}.	
\end{equation}

Let $K$ be the number of edges that have length at most $(\log k)^2/k$ between the root and the leaf node which has the marked gene lineage. Then $K\sim\text{Binomial}(k,1-e^{-c_2(\log k)^2/k})$. It follows that for all $\delta>0$,
$$P(K\geqslant \lfloor\delta k\rfloor)\leqslant\binom{k}{\lfloor\delta k\rfloor}(1-e^{-c_2(\log k)^2/k})^{\lfloor\delta k\rfloor}\leqslant\frac{(k(1-e^{-c_2(\log k)^2/k}))^{\lfloor\delta k\rfloor}}{\lfloor\delta k\rfloor!}\leqslant\frac{(c_2(\log k)^2)^{\lfloor\delta k\rfloor}}{\lfloor\delta k\rfloor!}.$$
By Stirling's approximation, we have $\lfloor\delta k\rfloor!\sim\sqrt{2\pi\lfloor\delta k\rfloor }(\lfloor\delta k\rfloor/e)^{\lfloor\delta k\rfloor}$ as $k\to\infty$. Therefore, for any $a>0$, we can choose $k$ large enough such that 
\begin{equation}\label{K123}
	P(K\geqslant \lfloor\delta k\rfloor)\leqslant a^k.
\end{equation}
If $K<\lfloor\delta k\rfloor$, then the number of edges that have length at least $(\log k)^2/k$ is no less than $\lceil(1-\delta)k\rceil$, which implies that $T\geqslant (1-\delta)(\log k)^2$. Therefore,
\begin{equation}\label{KTsubset123}
	\{K< \lfloor\delta k\rfloor\}\subset\left\{T\geqslant (1-\delta)(\log k)^2\right\}.
\end{equation}
If there exists a merger that involves all gene lineages on the same branch as the marked gene lineage at some time in Model 5 and does not cause the mark to be lost, the merger should get suppressed altogether. It follows that on $A_5$,
\begin{equation}\label{Nt123}
N(t)\geqslant2,~\forall t\in[0,T].
\end{equation}
By $(\ref{A3111})$, $(\ref{KTsubset123})$ and $(\ref{Nt123})$, we have
\begin{align}\label{A3KT123}
P(A_3  \cap\{K< \lfloor\delta k\rfloor\}| A_5)&\leqslant P(A_3 \cap\{T\geqslant(1-\delta)(\log k)^2\}|A_5)\nonumber\\
&\leqslant P(A_3  |\{T\geqslant(1-\delta)(\log k)^2\}\cap A_5)\nonumber\\
&\leqslant \exp\left(-f(2)(1-\delta)(\log k)^2\right).
\end{align}
Recall from (\ref{A3111}) that $A_4\subset A_5$. Then
\begin{align}\label{r_k123}
P(A_3)&=P(A_3  \cap\{K\geqslant \lfloor\delta k\rfloor\})+ P(A_3  \cap\{K< \lfloor\delta k\rfloor\})\nonumber\\
&\leqslant P(K\geqslant \lfloor\delta k\rfloor)+P(A_5)P(A_3  \cap\{K<\lfloor\delta k\rfloor\}|A_5)\nonumber\\
&\leqslant P(K\geqslant \lfloor\delta k\rfloor)+P(A_4)P(A_3  \cap\{K<\lfloor\delta k\rfloor\}|A_5).
\end{align}
By $(\ref{rkA3})$, $(\ref{A4111})$, $(\ref{K123})$, $(\ref{A3KT123})$ and $(\ref{r_k123})$, we have for sufficiently large $k$,
\begin{equation*}
	r_k=P(A_3)\leqslant a^k+\frac{m_3^*}{2^k}\exp\left(-f(2)(1-\delta)(\log k)^2\right).
\end{equation*}
Since $a$ can be chosen to be arbitrarily small, and $\exp\left(-f(2)(1-\delta_1)(\log k)^2\right)$ is asymptotically much smaller than $k^{-b}$ for any $b>0$, there exists some constant $d$ such that  
\begin{equation}\label{m3*dbound11}
\sum_{k=0}^{\infty}2^{k}r_k\leqslant m_3^*d.  
\end{equation}
By (\ref{limm123}), $(\ref{EX123321})$ and (\ref{m3*dbound11}), we have 
\begin{align*}
m_2^*-m_1^*=\lim_{n\to\infty}(m_{2,n}-m_{1,n})\leqslant \frac{2(c_2-c_1)(m_3^*)^2}{c_2}d.
\end{align*}	
Let $K_\epsilon=2(m_3^*)^2d/\epsilon$. Then 
$$|m_2^*-m_1^*|\leqslant K_{\epsilon}|c_2-c_1|.$$
The result follows.
\end{proof}

\begin{Lemma}\label{lemma23cont}
	Let $0<c<E[1/X]$. Suppose that $\tilde{\nu}^*_{c}\in\mathcal{S}_1$ is a distribution satisfies $\tilde{\nu}^*_{c}=G_{c}(\tilde{\nu}^*_{c})$. For $c_1,c_2\in(0,E[1/X])$, the limits $\lim_{n\to\infty}G_{c_i}^n(\tilde{\nu}_c^*)$, $i=1,2$ exist. Denote the limits by $\tilde{\nu}^*_{c_i}$, $i=1,2$. 
	For any given $c_1\in(0,E[1/X])$ and $\epsilon>0$, there exists $\delta>0$ such that for all $c_2\in(c_1-\delta,c_1+\delta)$, the distribution $\tilde{\nu}^*_{c_2}$ satisfies the following property.	 
	Let $\tilde{R}_{c_1}$ and $\tilde{R}_{c_2}$ be the probability generating functions of $\tilde{\nu}^*_{c_1}$ and $\tilde{\nu}^*_{c_2}$ respectively. Then 
$$|\tilde{R}_{c_1}(x)-\tilde{R}_{c_2}(x)|<\epsilon, ~\forall x\in[0,1].$$	 
\end{Lemma}
\begin{proof}
By (\ref{glimitlemma5}), we have $\lim_{n\to\infty}G_{c_i}^n(\tilde{\nu}^*_{c})$, $i=1,2$ exists on $\mathcal{S}$. By a similar argument to the one in (\ref{lemma18proof111}), 
we have $\tilde{\nu}_{c_i}^*=G_{c_i}(\tilde{\nu}_{c_i}^*)$.

Now we will show that $\tilde{\nu}_{c_i}^*\in\mathcal{S}_1$ by similar arguments to the one in the proof of Lemma~\ref{lemma18nup1}. Let $Y_i\sim\text{Exp}(c_i),i=1,2$. Let
$$b_{i,j}\coloneqq E[L_{j+1}(Y_i)-L_{j}(Y_i)],~j\in\mathbb{N}^+,~i=1,2.$$ 
Since $c_i<E[1/X]$, it follows from Lemma \ref{lemma17uniquep1} that there exists $\Delta>0$ and $J\in\mathbb{N}^+$ such that
\begin{equation*}\label{M1}
	b_{i,j}<\frac{1}{2}-\Delta, ~\forall j\geqslant J,~i=1,2.
\end{equation*} 
Let $X^*_{i,n}$ and $\bar{X}^*_{i,n}$ be two independent random variables with distribution $G_{c_i}^n(\tilde{\nu}^*_{c})$. Then
\begin{align*}
E[X^*_{i,n+1}]&=\sum_{j=2}^{J} P(X^*_{i,n}+\bar{X}^*_{i,n}=j)E[L_{j}(Y_2)]+\sum_{j=J+1}^\infty P(X^*_{i,n}+\bar{X}^*_{i,n}=j)E[L_{j}(Y_2)].\nonumber\\
&< \sum_{j=2}^J JP(X^*_{i,n}+\bar{X}^*_{i,n}=j)+\sum_{j=J+1}^\infty P(X^*_{i,n}+\bar{X}^*_{i,n}=j)\left( J+\left(\frac{1}{2}-\Delta\right)(j-J)\right)\nonumber\\
&< \left(\frac{3}{2}+\Delta\right)J+(1-2\Delta)E[X^*_{i,n}]. 
\end{align*}
It follows that the sequences $\{E[X^*_{i,n}]\}_{n=1}^\infty$, $i=1,2$ have upper bounds. If $c<c_i<E[1/X]$, by $(\ref{lemma21gc1c2})$, we have 
$$\tilde{\nu}^*_{c}=G_c(\tilde{\nu}_{c}^*)\preceq G_{c_i}(\tilde{\nu}_{c}^*).$$
Since $G_{c_i}$ is monotone, the sequence of iterates $\{G_{c_i}^n(\tilde{\nu}_c^*)\}_{n=0}^\infty$ is increasing, which implies that $\{E[X^*_{i,n}]\}_{n=1}^\infty$ is an increasing sequence. Therefore the limit $\lim_{n\to\infty}E[X^*_{i,n}]$ exists and is finite. By the Monotone Convergence Theorem, the limit $\lim_{n\to\infty}E[X^*_{i,n}]$ equals the mean of $\tilde{\nu}^*_{c_i}$, which implies that
$$\tilde{\nu}^*_{c_i}\in\mathcal{S}_1,~i=1,2.$$
If $0<c_i<c$, by $(\ref{lemma21gc1c2})$, we have
$$\tilde{\nu}^*_{c}=G_{c}(\tilde{\nu}_{c}^*)\succeq G_{c_i}(\tilde{\nu}_{c}^*).$$
Since $G_{c_i}$ is monotone, the sequence of iterates $\{G_{c_i}^n(\tilde{\nu}_c^*)\}_{n=0}^\infty$ is decreasing, 
which implies that $\{E[X^*_{i,n}]\}_{n=1}^\infty$ is a decreasing nonnegative sequence. Therefore the mean of $\tilde{\nu}_{c_i}^*$ is less than the mean of $\tilde{\nu}_c^*$. It follows that 
$$\tilde{\nu}^*_{c_i}\in\mathcal{S}_1,~i=1,2.$$
If $c_i=c$, then
$$\tilde{\nu}^*_{c_i}=\tilde{\nu}^*_c\in\mathcal{S}_1.$$
Note that
\begin{align*}
		\tilde{R}_{c_2}'(x)-&\tilde{R}_{c_1}'(x)\nonumber\\
		&=\sum_{n=1}^\infty(\tilde{\nu}^*_{c_2}(\{n\})-\tilde{\nu}^*_{c_1}(\{n\}))nx^{n-1}\nonumber\\
		&=\sum_{i=1}^\infty\left(\left(\tilde{\nu}^*_{c_2}[i,\infty)-\tilde{\nu}^*_{c_1}[i,\infty)\right)x^{i-1}-(1-x)\sum_{n=i+1}^\infty\left(\tilde{\nu}^*_{c_2}[n,\infty)-\tilde{\nu}^*_{c_1}[n,\infty)\right)x^{n-2}\right).
\end{align*}
If $c_1\leqslant c_2<E[1/X]$,
then $\tilde{\nu}^*_{c_1}=\lim_{n\to\infty}G_{c_1}^n(\tilde{\nu}_c^*)\preceq\lim_{n\to\infty}G_{c_2}^n(\tilde{\nu}_c^*)=\tilde{\nu}^*_{c_2}$, which implies
$$\tilde{\nu}^*_{c_2}[i,\infty)-\tilde{\nu}^*_{c_1}[i,\infty)\geqslant0,~ \forall i\in\mathbb{N}^+.$$
Then $\left(\tilde{\nu}^*_{c_2}[i,\infty)-\tilde{\nu}^*_{c_1}[i,\infty)\right)x^{i-1}$ is nondecreasing and $(1-x)\sum_{n=i+1}^\infty(\tilde{\nu}^*_{c_2}[n,\infty)-\tilde{\nu}^*_{c_1}[n,\infty))x^{n-2}$ is minimized at $x=1$ on $[0,1]$ for all $i\in\mathbb{N}^+$, which implies that $R_{c_2}'(x)-R_{c_1}'(x)$ is maximized at $x=1$ on $[0,1]$. It follows that 
$$0\leqslant \tilde{R}_{c_1}(x)-\tilde{R}_{c_2}(x)\leqslant (\tilde{R}_{c_2}'(1)-\tilde{R}_{c_1}'(1))(1-x), ~\forall x\in[0,1].$$ 
If $0<c_2<c_1$, then $\tilde{\nu}^*_{c_1}=\lim_{n\to\infty}G_{c_1}^n(\tilde{\nu}_c^*)\succeq\lim_{n\to\infty}G_{c_2}^n(\tilde{\nu}_c^*)=\tilde{\nu}^*_{c_2}$. By a similar argument, we get 
$$0\leqslant \tilde{R}_{c_2}(x)-\tilde{R}_{c_1}(x)\leqslant (\tilde{R}_{c_1}'(1)-\tilde{R}_{c_2}'(1))(1-x),~\forall x\in[0,1].$$ 
Therefore, for $c_2\in(0,E[1/X])$, we have 
	\begin{align}\label{lemma4equ11}
		|\tilde{R}_{c_2}(x)-\tilde{R}_{c_1}(x)|\leqslant |\tilde{R}_{c_2}'(1)-\tilde{R}_{c_1}'(1)|(1-x)\leqslant |\tilde{R}_{c_2}'(1)-\tilde{R}_{c_1}'(1)| , ~\forall x\in[0,1].
	\end{align}	
Note that $\tilde{R}'_{c_1}(1)$ and $\tilde{R}'_{c_2}(1)$ are the means of $\tilde{\nu}^*_{c_1}$ and $\tilde{\nu}^*_{c_2}$ respectively. By Lemma \ref{lemma22diff1} and $(\ref{lemma4equ11})$, for all $\epsilon>0$, there exists $\delta>0$, such that for all $c_2\in(c_1-\delta,c_1+\delta)$, 
$$|\tilde{R}_{c_2}(x)-\tilde{R}_{c_1}(x)|\leqslant  |\tilde{R}_{c_2}'(1)-\tilde{R}_{c_1}'(1)|<\epsilon, ~\forall x\in[0,1].$$
\end{proof}

\subsection{Proof of Theorem \ref{theo4unilambda}}
\begin{proof}[Proof of Theorem \ref{theo4unilambda}]
The existence has been proved in Lemma  \ref{lemma18nup1}.
We will show the uniqueness in the rest of the proof. Let the probability generating function of $\nu^*_c$ be $R_c$. Assume that there exists another distribution $\tilde{\nu}_c^*\in\mathcal{S}_1$ such that independent variables $W$, $W_1$ and $W_2$ with distribution $\tilde{\nu}_c^*$ satisfy $(\ref{theoremequ1})$. Then
	 \begin{equation*}\label{theorem5eq11}
	 	\tilde{\nu}_c^*=G_c(\tilde{\nu}_c^*).
	 \end{equation*}
	  Let the probability generating function of  $\tilde{\nu}_c^*$ be $\tilde{R}_c$. By Lemma \ref{lemmapgflambda}, for all $x\in[0,1]$, the function $R_c$ satisfies the equation 
\begin{equation*}
E\left[\frac{R_c(x)-R_c((1-X)x)-xR_c(X+(1-X)x)+xR_c((1-X)x)}{X^2}\right]=c(R_c^2(x)-R_c(x)).
\end{equation*}
Note that $R_c(x)-R_c((1-X)x)-xR_c(X+(1-X)x)+xR_c((1-X)x)$ equals 0 at $x=1$. It follows that 
\begin{align}\label{theo24dct1}
&c(2R_c(1)R_c'(1)-R_c'(1))\nonumber\\
&\quad=\lim_{x\to1^-}\frac{1}{1-x}\left(0-E\left[\frac{R_c(x)-R_c((1-X)x)-xR_c(X+(1-X)x)+xR_c((1-X)x)}{X^2}\right]\right)\nonumber\\
&\quad=	\lim_{x\to1^-}\left(-E\left[\frac{(1-x)(R_c(x)-R_c((1-X)x))}{(1-x)X^2}\right]-E\left[\frac{x(R_c(x)-R_c(X+(1-X)x))}{(1-x)X^2}\right]\right).\nonumber\\
\end{align}
Note that probability generating functions have an increasing first derivative. Also note that if $0\leqslant x\leqslant 1$, then $(1-X)x\leqslant x\leqslant X+(1-X)x$. 
Therefore,
$$ 0\leqslant \frac{(1-x)(R_c(x)-R_c((1-X)x))}{(1-x)X^2}\leqslant \frac{R_c'(1)Xx}{X^2}\leqslant R_c'(1)\frac{1}{X},~\forall x\in[0,1],$$
and
$$ 0\geqslant\frac{x(R_c(x)-R_c(X+(1-X)x))}{(1-x)X^2}\geqslant -\frac{xR_c'(1)(X-Xx)}{(1-x)X^2}\geqslant -R_c'(1)\frac{1}{X},~\forall x\in[0,1].$$
Since $\nu_c^*$ and $\tilde{\nu}_c^*$ are distributions in $\mathcal{S}_1$, they have finite mean, which implies $R_c'(1)$ and $\tilde{R}_c'(1)$ are finite. Then by the Dominated Convergence Theorem, the right-hand side of (\ref{theo24dct1}) equals the expectation of the derivative of $(R_c(x)-R_c((1-X)x)-xR_c(X+(1-X)x)+xR_c((1-X)x))/X^2$ with respect to $x$ at $x=1$, which is 
\begin{align}\label{theo24dct2}
&E\left[\frac{R_c'(1)-(1-X)R_c'((1-X))-R_c(1)+R(1-X)-(1-X)R_c'(1)+(1-X)R_c'(1-X)}{X^2}\right]\nonumber\\
	&\quad\quad\quad\quad\quad\quad\quad\quad\quad\quad\quad\quad\quad\quad\quad\quad\quad\quad\quad =E\left[\frac{XR_c'(1)-R_c(1)+R_c(1-X)}{X^2}\right].
\end{align}
Note that (\ref{theo24dct1}) and  (\ref{theo24dct2}) also hold with $R_c$ replaced by $\tilde{R}_c$. 
Since $R_c$ and $\tilde{R}_c$ are probability generating functions, we have $R_c(1)=1$ and $\tilde{R}_c(1)=1$. Combining (\ref{theo24dct1}) and (\ref{theo24dct2}), we have 
	\begin{equation}\label{exp9}
		R_c'(1)=\frac{E\left[\frac{1-R_c(1-X)}{X^2}\right]}{E\left[\frac{1}{X}\right]-c}, \text{ and } \tilde{R}_c'(1)=\frac{E\left[\frac{1-\tilde{R}_c(1-X)}{X^2}\right]}{E\left[\frac{1}{X}\right]-c}.
	\end{equation}
Because of the monotonicity of $G_c$, we have $\tilde{\nu}_c^*=\lim_{n\to\infty}G_c^n(\tilde{\nu}_c^*)\succeq\lim_{n\to\infty}G_c^n(\delta_1)=\nu_c^*$. Since $\tilde{\nu}_c^*$ and $\nu^*_c$ are different distributions that have different means, we have 
	\begin{equation*}
		\tilde{R}_c(x)<R_c(x), ~\forall x\in(0,1),
	\end{equation*}
	and 
	\begin{equation*}\label{theorem5eq22}
		\tilde{R}_{c}'(1)>R_c'(1).
	\end{equation*}
	For $c_1\in(c,E[1/X])$, let $R_{c_1}$ be the probability generating function of $\lim_{n\to\infty}G_{c_1}^n(\delta_1)$. By Lemma~\ref{lemma21now},  we have $R_{c_1}(x)<R_c(x)$ for all $x\in(0,1)$. 
	Therefore, by $(\ref{exp9})$,
	$$R_{c_1}'(1)=\frac{E\left[\frac{1-R_{c_1}(1-X)}{X^2}\right]}{E\left[\frac{1}{X}\right]-c_1}\geqslant \frac{E\left[\frac{1-R_{c}(1-X)}{X^2}\right]}{E\left[\frac{1}{X}\right]-c_1}=\frac{R_c'(1)\left(E\left[\frac{1}{X}\right]-c\right)}{E\left[\frac{1}{X}\right]-c_1},$$
	which implies $\lim_{c_1\to E[1/X]}R_{c_1}'(1)=\infty$. 
	Recall that $\tilde{R}_{c}'(1)<\infty$ because $\tilde{\nu}_c^*\in\mathcal{S}_1$. Recall the definition of $G_0$ stated before Lemma \ref{lemma22diff1}. Note that $\delta_1$ satisfies $\delta_1=G_0(\delta_1)$.
	By Lemma \ref{lemma22diff1}, we have that $R_{c}'(1)$ is continuously increasing as a function of $c$. It follows that there exists a $c_1>c$ such that $R_{c_1}'(1)= \tilde{R}_{c}'(1)$. By $(\ref{exp9})$, we have 
	$$E\left[\frac{1-R_{c_1}(1-X))}{X^2}\right]=R_{c_1}'(1)\left(E\left[\frac{1}{X}\right]-c_1\right)<\tilde{R}_{c}'(1)\left(E\left[\frac{1}{X}\right]-c\right)=E\left[\frac{1-\tilde{R}_{c}(1-X))}{X^2}\right].$$
Then there exists a $c_2$ slightly larger than $c_1$ such that the probability generating function of $\lim_{n\to\infty}G_{c_2}^n(\delta_1)$, denoted by $R_{c_2}$, satisfies $R_{c_2}'(1)> \tilde{R}_c'(1)$ and
\begin{equation}\label{theo7eq111}
E\left[\frac{1-R_{c_2}(1-X)}{X^2}\right]<E\left[\frac{1-\tilde{R}_c(1-X)}{X^2}\right].
\end{equation}
If $R_{c_2}(x)\leqslant\tilde{R}_c(x)$ for all $x\in[0,1]$, then inequality $(\ref{theo7eq111})$ can not hold.  
Since $R_{c_2}'(1)> \tilde{R}_c'(1)$, there exists $\epsilon_1$ such that $R_{c_2}(x)<\tilde{R}_c(x)$ for all $x\in(1-\epsilon_1,1)$.
Let 
$$x_1\coloneqq \sup\{x\in(0,1):R_{c_2}(x)>\tilde{R}_c(x)\}.$$
Then $R_{c_2}(x)\leqslant \tilde{R}_c(x)$ on $(x_1,1),$ and $R_{c_2}(x)>\tilde{R}_c(x)$ on $(x_1-\epsilon_2,x_1)$ for some $\epsilon_2>0$. Let
$$x_2\coloneqq \sup\left\{\underset{x\in(0,1)}{\arg\max} \left(R_{c_2}(x)-\tilde{R}_c(x)\right)\right\}.$$
Then $0<x_2<x_1$ and $R_{c_2}(x_2)-\tilde{R}_c(x_2)>0$. 
Note that $R_{c_2}(x)$ and $\tilde{R}_c(x)$ both satisfy $(\ref{PGF})$. Then
\begin{align}\label{EQRC2123}
	E&\left[\frac{(R_{c_2}-\tilde{R}_c)(x_2)-x_2(R_{c_2}-\tilde{R}_c)(X+(1-X)x_2)-(1-x_2)(R_{c_2}-\tilde{R}_c)((1-X)x_2)}{X^2}\right]\nonumber\\
	&\quad\quad\quad\quad\quad\quad\quad\quad\quad\quad\quad\quad\quad\quad\quad\quad\quad=-c_2(R_{c_2}(x_2)-R_{c_2}^2(x_2))+c(\tilde{R}_c(x_2)-\tilde{R}_c^2(x_2)).
\end{align}
Since $x_2$ is the point that maximizes $R_{c_2}(x)-\tilde{R}_c(x)$, we have 
$$R_{c_2}(X+(1-X)x_2)-\tilde{R}_c(X+(1-X)x_2)\leqslant R_{c_2}(x_2)-\tilde{R}_c(x_2),$$ 
and 
$$R_{c_2}((1-X)x_2)-\tilde{R}_c((1-X)x_2)\leqslant R_{c_2}(x_2)-\tilde{R}_c(x_2).$$ 
It follows that the left-hand side of $(\ref{EQRC2123})$ is nonnegative.
Therefore,  
\begin{equation}\label{EQRC212345}
	c_2(R_{c_2}(x_2)-R_{c_2}^2(x_2))-c(\tilde{R}_c(x_2)-\tilde{R}_c^2(x_2))\leqslant0.
\end{equation}
We will discuss the following 2 different cases.

\textbf{Case 1:} $x_1\leqslant 1/2$. If $x_1\leqslant 1/2$, then $x_2<1/2$. Note that $R_{c_2}$ has a positive derivative  and is convex on the interval $(0,1)$ with $R_{c_2}(0)=0$ and $R_{c_2}(1)=1$. Then 
$$\tilde{R}_c(x_2)<R_{c_2}(x_2)<R_{c_2}\left(\frac{1}{2}\right)<\frac{1}{2}.$$
Recall that $c_2>c_1>c$. Since the mapping $y\mapsto y-y^2$ is increasing on (0,1/2), we have    
$$c_2(R_{c_2}(x_2)-R_{c_2}^2(x_2))-c(\tilde{R}_c(x_2)-\tilde{R}_c^2(x_2))>0.$$
This contradicts $(\ref{EQRC212345})$.

\textbf{Case 2:} $x_1>1/2$. For $c_3<c$, let $\tilde{R}_{c_3,1}$ be the probability generating function of $G_{c_3}(\tilde{\nu}_c^*)$. It is easy to see that
$$\lim_{c_3\to 0}\int_{1}^\infty x (G_{c_3}(\nu))(dx) = 1,~\forall \nu\in\mathcal{S}_1.$$
Then $\lim_{c_3\to 0}\tilde{R}_{c_3,1}'(1)=1$. Since $R_{c_2}$ is convex on $[0,1]$ with $R_{c_2}(1)=1$, the line connecting the two points $(1/2,R_{c_2}(1/2))$ and $(1,1)$ has slope larger than 1, and is above $R_{c_2}(x)$ for all $x\in[1/2,1]$. Since $\tilde{R}_{c_3,1}$ is convex on $[0,1]$ with $\tilde{R}_{c_3,1}(1)=1$, the line $y=\tilde{R}_{c_3,1}'(1)(x-1)+1$
is below $y=\tilde{R}_{c_3,1}$ for all $x\in[0,1]$. Therefore, there exists $c_3$ sufficiently small that $R_{c_2}(x)<\tilde{R}_{c_3,1}(x)$ for all $x\in[1/2,x_1]$. Note that $\lim_{n\to\infty}G_{c_3}^n(\tilde{\nu}_c^*)\preceq G_{c_3}(\tilde{\nu}_c^*)\preceq \tilde{\nu}_c^*$. Let $\tilde{R}_{c_3}$ be the probability generating function of $\lim_{n\to\infty}G_{c_3}^n(\tilde{\nu}_c^*)$. Then $\tilde{R}_{c_3}(x)\geqslant \tilde{R}_{c_3,1}(x)$ for all $x\in(0,1)$. Therefore,
\begin{equation}\label{TheoCase233}
	\max_{x\in[1/2,x_1]}(R_{c_2}(x)-\tilde{R}_{c_3}(x))<0.
\end{equation}
Recall that there exists $\epsilon_2>0$ such that $R_{c_2}(x)>\tilde{R}_c(x)$ on $(x_1-\epsilon_2,x_1)$. Since $x_1>1/2$, we have 
\begin{equation}\label{TheoCase244}
\max_{x\in[1/2,x_1]}(R_{c_2}(x)-\tilde{R}_{c}(x))>0.
\end{equation}
For $c_3'<c$, let $\tilde{R}_{c_3'}$ be the probability generating function of $\lim_{n\to\infty}G_{c_3'}^n(\tilde{\nu}_c^*)$. Then
\begin{align*}
	\max_{x\in[1/2,x_1]}\left(\tilde{R}_{c_3}(x)-\tilde{R}_{c_3'}(x)\right)+&\max_{x\in[1/2,x_1]}\left(R_{c_2}(x)-\tilde{R}_{c_3}(x)\right)\nonumber\\
	&\geqslant \max_{x\in[1/2,x_1]}\left(\tilde{R}_{c_3}(x)-\tilde{R}_{c_3'}(x)+R_{c_2}(x)-\tilde{R}_{c_3}(x)\right)\nonumber\\
	&=\max_{x\in[1/2,x_1]}\left(R_{c_2}(x)-\tilde{R}_{c_3'}(x)\right),
\end{align*}
and 
\begin{align*}
	\max_{x\in[1/2,x_1]}\left(\tilde{R}_{c_3'}(x)-\tilde{R}_{c_3}(x)\right)+&\max_{x\in[1/2,x_1]}\left(R_{c_2}(x)-\tilde{R}_{c_3'}(x)\right)\nonumber\\
	&\geqslant \max_{x\in[1/2,x_1]}\left(\tilde{R}_{c_3'}(x)-\tilde{R}_{c_3}(x)+R_{c_2}(x)-\tilde{R}_{c_3'}(x)\right)\nonumber\\
	&=\max_{x\in[1/2,x_1]}\left(R_{c_2}(x)-\tilde{R}_{c_3}(x)\right),
\end{align*}
which implies that 
\begin{align}\label{TheoCase211}
\left|\max_{x\in[1/2,x_1]}(R_{c_2}(x)-\tilde{R}_{c_3}(x))-\max_{x\in[1/2,x_1]}(R_{c_2}(x)-\tilde{R}_{c_3'}(x))\right|
\leqslant	\max_{x\in[1/2,x_1]}|\tilde{R}_{c_3}(x)-\tilde{R}_{c_3'}(x)|.
\end{align}
 By Lemma \ref{lemma23cont}, for any given $c_3$, for all $\epsilon>0$, there exists $\delta>0$, such that 
\begin{equation*}\label{TheoCase222}
	|\tilde{R}_{c_3}(x)-\tilde{R}_{c_3'}(x)| <\epsilon, ~\forall x\in[0,1],~\forall c_3'\in(c_3-\delta,c_3+\delta),
\end{equation*}
which implies 
\begin{equation}\label{TheoCase222}
	\max_{x\in[1/2,x_1]}|\tilde{R}_{c_3}(x)-\tilde{R}_{c_3'}(x)| <\epsilon.
\end{equation}
By $(\ref{TheoCase211})$ and $(\ref{TheoCase222})$, we have that $\max_{x\in[1/2,x_1]}(R_{c_2}(x)-\tilde{R}_{c_3}(x))$ is continuous with respect to $c_3$. By $(\ref{TheoCase233})$ and $(\ref{TheoCase244})$, we can find $c_3<c$ such that 
\begin{equation}\label{Case255}
	\max_{x\in[1/2,x_1]}(R_{c_2}(x)-\tilde{R}_{c_3}(x))=0.
\end{equation}

Since $c_3<c$, we have that $\lim_{n\to\infty}G_{c_3}^n(\tilde{\nu}_c^*)$ is stochastically dominated by $\tilde{\nu}_c^*$, and the two distributions are not the same, which implies $\tilde{R}_{c_3}(x)> \tilde{R}_{c}(x)$ for all $x\in(0,1)$. Recall that $R_{c_2}(x)\leqslant \tilde{R}_{c}(x)$ for all $x\in(x_1,1)$. Therefore,
\begin{equation}\label{Case2.13}
	R_{c_2}(x)<\tilde{R}_{c_3}(x),~\forall x\in(x_1,1).
\end{equation}
Let
$$x_3\coloneqq \sup\{x\in(0,1):R_{c_2}(x)>\tilde{R}_{c_3}(x)\}.$$
Let
$$x_4\coloneqq \sup\left\{\underset{x\in(0,1)}{\arg\max} \left(R_{c_2}(x)-\tilde{R}_{c_3}(x)\right)\right\}.$$
By $(\ref{Case255})$ and $(\ref{Case2.13})$, we have $x_3\in[1/2,x_1]$ and $R_{c_2}(x_4)-\tilde{R}_{c_3}(x_4)\geqslant0$. Note that inequality $(\ref{EQRC212345})$ still holds with $c$ replaced by $c_3$ and $x_2$ replaced by $x_4$. That is
\begin{equation}\label{EQRC2123456}
	c_2(R_{c_2}(x_4)-R_{c_2}^2(x_4))-c_3(\tilde{R}_{c_3}(x_4)-\tilde{R}_{c_3}^2(x_4))\leqslant0.
\end{equation}
We will discuss the following two subcases.

\textbf{Case 2.1:} $R_{c_2}(x_4)-\tilde{R}_{c_3}(x_4)>0$. By $(\ref{Case255})$ and $(\ref{Case2.13})$, we have $x_4<1/2$. Note that $R_{c_2}(x)$ has a positive derivative on the interval $(0,1)$. Then $\tilde{R}_{c_3}(x_4)<R_{c_2}(x_4)<R_{c_2}(1/2)<1/2$. Note that $c_3<c<c_2$. Then 
$$c_2(R_{c_2}(x_4)-R_{c_2}^2(x_4))-c_3(\tilde{R}_{c_3}(x_4)-\tilde{R}_{c_3}^2(x_4))>0.$$
This contradicts $(\ref{EQRC2123456})$.

\textbf{Case 2.2:} $R_{c_2}(x_4)-\tilde{R}_{c_3}(x_4)=0$. In this case, we have $x_4=x_3$. Since $0<R_{c_2}(x)<1$ for all $x\in(0,1)$ and $c_3<c_2$, we have
$$c_2(R_{c_2}(x_4)-R_{c_2}^2(x_4))-c_3(\tilde{R}_{c_3}(x_4)-\tilde{R}_{c_3}^2(x_4))=(c_2-c_3)(R_{c_2}(x_4)-R_{c_2}^2(x_4))>0.$$
This contradicts $(\ref{EQRC2123456})$.

Therefore, there won't be another distribution $\tilde{\nu}_c^*\in\mathcal{S}_1$ that satisfies our conditions. 
\end{proof}

\subsection{Proof of Theorem \ref{Theoexistinflambda}}
\begin{Lemma}\label{lemmabefore55}
The limit $\lim_{n\to\infty}G_c^n(\delta_\infty)$ exists. Let $\nu_c'\coloneqq\lim_{n\to\infty}G_c^n(\delta_\infty)$. Then $\nu_c'$ is a solution to (\ref{theoremequ1}). Suppose that $E[1/X]<\infty$. If $P(X=1)=0$, then $\nu_c'=\delta_\infty$. If $c>P(X=1)$, then $\nu_c'(\{\infty\})>0$. If $c\leqslant P(X=1)$, then $\nu_c'(\{\infty\})=0$.	
\end{Lemma}
\begin{proof}
Since $G_c$ is monotone, the sequence of iterates $\{G_c^n(\delta_\infty)\}_{n=1}^\infty$ is decreasing and there exists a distribution $\nu_c'$ on $\mathbb{N}^+\cup\{\infty\}$ such that 
$$	\lim_{n\to\infty}G_c^n(\delta_\infty)\to_d\nu_c',~\text{as }n\to\infty.$$
By a similar argument to the one in (\ref{lemma18proof111}) in Lemma \ref{lemma18nup1}, the distribution $\nu_c'$ is a solution to (\ref{theoremequ1}).

Suppose that $P(X=1)=0$. Let $(\Pi_t,t\geqslant 0)$ be a $\Lambda$-coalescent characterized by the probability measure of $X$. Let $L_\infty(t)$ denote the number of blocks in $\Pi_t$. Note that $P(L_\infty(t)=\infty)=1$ for all $t>0$ because the $\Lambda$-coalescent has dust, and therefore the number of blocks remains infinity at all times if starting from infinitely many blocks at time 0.
It then follows that $P(L_\infty(Y)=\infty)=1$. Therefore $G_c(\delta_\infty)=\delta_\infty$, which implies that $\nu_c'=\delta_\infty$.

Suppose that $P(X=1)>0$. Then the rate for infinitely many blocks all to merge into one is $P(X=1)$, which implies that $P(L_\infty(Y)=\infty)=c/(P(X=1)+c)$.
Since $\nu_c'$ is a solution to (\ref{theoremequ1}), we have 
\begin{equation}\label{lemmabefore5}
\nu_c'(\{\infty\})=G_c(\nu_c'(\{\infty\}))=(\nu_c'\ast\nu_c')(\{\infty\})P(L_\infty(Y)=\infty)=\frac{c\left(1-(1-\nu_c'(\{\infty\}))^2\right)}{P(X=1)+c}.
\end{equation}
If $c\leqslant P(X=1)$, then this equation can only hold when $\nu_c'(\{\infty\})=0.$

Suppose that $c>P(X=1)$. Note that $\{\left(G_c^n(\delta_\infty)\right)(\{\infty\})\}_{n=1}^\infty$ is a decreasing sequence of numbers with
\begin{align*} 
\left(G_c^{n+1}(\nu_c')\right)(\{\infty\})&=\left(2\left(G_c^{n}(\nu_c')\right)(\{\infty\})-\left(\left(G_c^{n}(\nu_c')\right)(\{\infty\})\right)^2\right)\frac{c}{P(X=1)+c}\\
&=\frac{2c}{P(X=1)+c}\left(G_c^{n}(\nu_c')\right)(\{\infty\})-\frac{c}{P(X=1)+c}\left(\left(G_c^{n}(\nu_c')\right)(\{\infty\})\right)^2.
\end{align*}
Since $2c/(P(X=1)+c)>1$, the sequence $\{\left(G_c^n(\delta_\infty)\right)(\{\infty\})\}_{n=1}^\infty$ has a strictly positive limit. By solving (\ref{lemmabefore5}), we get $\nu_c'(\{\infty\})=1-P(X=1)/c$.
\end{proof}

In the proof of Theorem \ref{Theoexistinflambda}, we are going to use the same
notations as in \cite{YuleKingmanToni}. Consider a Yule tree. Denote the first time when the number of branches reaches $m$ as $u=0$. Let the number of branches at time $u$ be $R(u)$. Then $R(0)=m$ and $R(0^-)=m-1$. If we consider the the portion of the tree after time $u=0$, then we have $m$ subtrees. The following notations for the subtree, branching times and number of lineages are the same as in \cite{blancas2019}. We can place the $m$ subtrees in random order and denote them by $\mathcal{T}^{l,m}$, $l\in\{1,\ldots,m\}$. Let $V^{l,m}$ denote the time when the initial branch splits into two. Let $V^{l,m}_{1}$ and $V^{l,m}_{2}$ denote the times when the two branches created at time $V^{l,m}$ split into two respectively. Given $V^{l,m}_{i_{1}\ldots i_{d}}$, define $V^{l,m}_{i_{1}\ldots i_{d}1}$ and $V^{l,m}_{i_{1}\ldots i_{d}2}$ to be the times when the two branches created at time $V^{l,m}_{i_{1}\ldots i_{d}}$ split again. Let $U^{s+1}_{m}=\inf\{u:R(u)=s+1\}$. Lemma 11 in \cite{YuleKingmanToni} says that for a fixed positive integer $d$, we have
$$\lim_{s\to\infty}P\left(U_{m}^{s+1}\geqslant\max\left\{V_{i_{1}\ldots i_{d}}^{l,m}, i_{1},\ldots,i_{d}\in\{1,2\}, 1\leqslant l\leqslant m\right\}\right)=1.$$
Formula (3.3) and (3.5) in \cite{YuleKingmanToni} respectively state that
\begin{equation}\label{usedintheo5222}
P\left(U^{s+1}_{m}<u\right)=e^{-cmu}\sum_{i=s+1}^{\infty}\binom{i-1}{i-m} (1-e^{-cu})^{i-m},
\end{equation}
and
\begin{equation}\label{Taums} 
\tau_{m}^{s}=_d U^{s+1}_{m}.
\end{equation}

\begin{center}
\setlength{\unitlength}{0.75cm}
\begin{picture}(22,8.5)
\linethickness{0.3mm}
\put(0,1){\line(2,1){11}}
\put(11,6.5){\line(2,-1){11}}
\put(11,6.5){\line(0,1){1.5}}
\put(6,4){\line(2,-1){6}}
\put(17.2,3.4){\line(-2,-1){4.8}}
\put(1,1.5){\line(2,-1){1}}
\put(8,3){\line(-2,-1){4}}
\put(20,2){\line(-2,-1){2}}
\put(11,6.5){\circle*{0.2}}
\put(6,4){\circle*{0.2}}
\put(17.2,3.4){\circle*{0.2}}
\put(1,1.5){\circle*{0.2}}
\put(8,3){\circle*{0.2}}
\put(20,2){\circle*{0.2}}
\put(11,8){\circle*{0.2}}
\put(11.2,8){$u=0$}
\put(11.2,6.5){$V^{l,m}$}
\put(4.5,4){$V_1^{l,m}$}
\put(17.5,3.4){$V_2^{l,m}$}
\put(0,1.5){$V_{11}^{l,m}$}
\put(8.5,3){$V_{12}^{l,m}$}
\put(20.3,2){$V_{22}^{l,m}$}
\end{picture}

\vspace{-.2in}
Figure 3: The tree ${\cal T}^{l,m}$.
\end{center}

\begin{proof}[Proof of Theorem \ref{Theoexistinflambda}]
Recall from the beginning of the proof of Lemma \ref{lemma18nup1} that $\{G_c^n(\delta_1)\}_{n=1}^\infty$ is an increasing sequence of distributions. Then by Lemma \ref{lemma18nup1} and Lemma \ref{lemmabefore55}, for every $x>0$, for any $\epsilon>0,$ there exists $\bar{d}\in\mathbb{N}^+$ such that
\begin{equation*}\label{T11theo5new}
-\frac{\epsilon}{2}< \left(G^{d+1}_c(\delta_\infty)\right)([0,x])-\nu_c'([0,x])<0<\left(G^{d+1}_c(\delta_1)\right)([0,x])-\nu_c^*([0,x])<\frac{\epsilon}{2}, ~\forall d\geqslant\bar{d}. 
\end{equation*}
By Lemma 11 in \cite{YuleKingmanToni} and $(\ref{Taums})$, for any $\bar{d}\in\mathbb{N}^+$, there exists $\bar{s}\in\mathbb{N}^+$ such that 
\begin{equation*}\label{T12theo5}
P\left(\tau_{m}^{s}\geqslant\max\left\{V_{i_{1}\ldots i_{\bar{d}}}^{l,m}, i_{1},\ldots,i_{\bar{d}}\in\{1,2\}, 1\leqslant l\leqslant m\right\}\right)>1-\frac{\epsilon}{2},~ \forall s\geqslant \bar{s}.
\end{equation*}
Let 
$$A\coloneqq\left\{\tau_{m}^{s}\geqslant\max\left\{V_{i_{1}\ldots i_{\bar{d}}}^{l,m}, i_{1},\ldots,i_{\bar{d}}\in\{1,2\}, 1\leqslant l\leqslant m\right\}\right\}.$$ 
Then $P\left(A^C\right)<\epsilon/2$.
Since the number of individual lineages belonging to each species at any time is at least 1, we can lower bound the number of individual lineages in each species immediately below any branchpoint at time $V_{i_{1}\ldots i_{\bar{d}}}^{l,m}$, where $i_{1},\ldots,i_{\bar{d}}\in\{1,2\}$ and $1\leqslant l\leqslant m$, by 1. 
Note that on event $A$, all times $V_{i_{1}\ldots i_{\bar{d}}}^{l,m},i_{1},\ldots,i_{\bar{d}}\in\{1,2\},1\leqslant l\leqslant m$ are less than $\tau_{m}^{s}$. Therefore, there exist independent random variables $X_l^{\bar{d}+1}$, $l\in\{1,\ldots,m\}$ with distribution $G^{\bar{d}+1}_c(\delta_1)$ such that $X_l^{\bar{d}+1}\leqslant N_{k_l}\left(\left(\tau_{m}^{s}\right)^-\right)$ for all $l\in\{1,\ldots,m\}$ on event $A$. It follows that for all $x>0$ and $l\in\{1,\ldots,m\}$, we have
$$P\left(\left\{N_{k_l}\left(\left(\tau_{m}^{s}\right)^-\right)\leqslant x\right\}\cap A\right)\leqslant P\left(\left\{X_l^{\bar{d}+1}\leqslant x\right\}\cap A\right)\leqslant P\left(X_l^{\bar{d}+1}\leqslant x\right)= \left(G^{\bar{d}+1}_c(\delta_1)\right)([0,x]),$$
which implies that for any $x>0$,
\begin{align}\label{prooftheo5eq101}
 P\left(N_{k_l}\left(\left(\tau_{m}^{s}\right)^-\right)\leqslant x\right)-\nu_c^*([0,x])&\leqslant  P\left(\left\{N_{k_l}\left(\left(\tau_{m}^{s}\right)^-\right)\leqslant x\right\}\cap A\right)+P\left(A^C\right)-\nu_c^*([0,x])\nonumber\\
&\leqslant \left(G^{\bar{d}+1}_c(\delta_1)\right)([0,x])+P\left(A^C\right)-\nu_c^*([0,x])\nonumber\\
&<\epsilon, ~ \forall s\geqslant\bar{s}, ~\forall l\in\{1,\ldots,m\}.
\end{align}
It follows from (\ref{usedintheo5222}) that 
for any $s$, there exists $\bar{u}>0$ such that 
$$P\left(U^{s+1}_{m}\geqslant \bar{u}\right)<\frac{\epsilon}{2}.$$
Since $V_{i_{1}\ldots i_{d}}^{l,m}$ is a sum of $d+1$ independent exponential random variables with parameter $c$, we have $V_{i_{1}\ldots i_{d}}^{l,m}\sim Gamma(d+1,c)$. For a given time $u$, we have
$$P\left(V_{i_{1}\ldots i_{d}}^{l,m}> u\right)=e^{-cu}\sum_{i=0}^{d}\frac{(cu)^{i}}{i!}.$$
Then for fixed $u$, we have
\begin{align*}
P\left(\min\left\{V_{i_{1}\ldots i_{d}}^{l,m}, i_{1},\ldots,i_{d}\in\{1,2\}, 1\leqslant l\leqslant m\right\}\geqslant u\right)\to 1,~\text{as}~d\to\infty.
\end{align*}
Then for any $\epsilon>0$, there exists $\tilde{d}>0$ such that 
\begin{equation*}\label{barusection31theo5}
P\left(\min\left\{V_{i_{1}\ldots i_{d}}^{l,m}, i_{1},\ldots,i_{d}\in\{1,2\},1\leqslant l\leqslant m\right\}< \bar{u}\right)<\frac{\epsilon}{2},~ \forall d\geqslant \tilde{d}. 
\end{equation*}
Therefore,
\begin{align}\label{prooftheo5eq3}
	&P\left(U_m^{s+1}<\min\left\{V_{i_{1}\ldots i_{d}}^{l,m}, i_{1},\ldots,i_{d}\in\{1,2\},1\leqslant l\leqslant m\right\}\right)\nonumber\\
	&\quad\quad\quad\quad\geqslant 1-P\left(\min\left\{V_{i_{1}\ldots i_{d}}^{l,m}, i_{1},\ldots,i_{d}\in\{1,2\},1\leqslant l\leqslant m\right\}< \bar{u}\right)-P\left(U^{s+1}_{m}\geqslant \bar{u}\right)\nonumber\\
	&\quad\quad\quad\quad>1-\epsilon,~ \forall d\geqslant \tilde{d}. 
\end{align}
Let 
$$B\coloneqq\left\{\tau_{m}^{s}<\min\left\{V_{i_{1}\ldots i_{\tilde{d}}}^{l,m}, i_{1},\ldots,i_{\tilde{d}}\in\{1,2\}, 1\leqslant l\leqslant m\right\}\right\}.$$ 
By $(\ref{Taums})$ and (\ref{prooftheo5eq3}), we have $P(B)>1-\epsilon$. Since the number of individual lineages belonging to each species at any time is at least 1 and $N_k(0)=1$ for all $k\in\{1,\ldots,s\}$, 
there exist independent random variables $X_l^{\tilde{d}+1}$, $l\in\{1,\ldots,m\}$ with distribution $G^{\tilde{d}+1}_c(\delta_1)$ such that $X_l^{\tilde{d}+1}\geqslant N_{k_l}\left(\left(\tau_{m}^{s}\right)^-\right)$ for all $l\in\{1,\ldots,m\}$ on event $B$. It follows that for any $x>0$ and $l\in\{1,\ldots,m\}$, we have
\begin{align*}
	P\left(\left\{N_{k_l}\left(\left(\tau_{m}^{s}\right)^-\right)\leqslant x\right\}\cap B\right)&\geqslant P\left(\left\{X_l^{\tilde{d}+1}\leqslant x\right\}\cap B\right)\\
	&\geqslant P\left(X_l^{\tilde{d}+1}\leqslant x\right)-P\left(B^C\right)\\
	&= \left(G^{\tilde{d}+1}_c(\delta_1)\right)([0,x])-P\left(B^C\right).
\end{align*}
Then
\begin{align}\label{prooftheo5eq202}
 P\left(N_{k_l}\left(\left(\tau_{m}^{s}\right)^-\right)\leqslant x\right)-\nu_c^*([0,x])& \geqslant  P\left(\left\{N_{k_l}\left(\left(\tau_{m}^{s}\right)^-\right)\leqslant x\right\}\cap B\right)-\nu_c^*([0,x])\nonumber\\
&\geqslant  \left(G^{\tilde{d}+1}_c(\delta_1)\right)([0,x])-P\left(B^C\right)-\nu_c^*([0,x])\nonumber\\
& >-\epsilon, ~ \forall s\geqslant\bar{s}, ~\forall l\in\{1,\ldots,m\}.
\end{align}
It follows from (\ref{prooftheo5eq101}) and (\ref{prooftheo5eq202}) that $N_{k_l}\left(\left(\tau_{m}^{s}\right)^-\right)\to_d\nu_c^*$ as $s\to\infty$ for all $l\in\{1,\ldots,m\}$ if $N_k(0)=1$ for all $k\in\{1,\ldots,s\}$. 

For the case $N_k(0)=1$ for all $k\in\{1,\ldots,s\}$ in Theorem \ref{Theoexistinflambda}, it remains to show the asymptotic independence. Let $\{a_1,\ldots,a_m\}$ be an arbitrary point of $(\mathbb{N}^+)^m$. 
Then
\begin{align*}
	P\left(N_{k_{l}}\left(\left(\tau_{m}^{s}\right)^-\right)\leqslant a_l,\forall l\in\{1,\ldots,m\}\right)&\leqslant P\left(\left\{N_{k_{l}}\left(\left(\tau_{m}^{s}\right)^-\right)\leqslant a_l,\forall l\in\{1,\ldots,m\}\right\}\cap A\right)+P\left(A^C\right)\nonumber\\
	&\leqslant P\left(\left\{X^{\bar{d}+1}_l\leqslant a_l,\forall l\in\{1,\ldots,m\}\right\}\cap A\right)+P\left(A^C\right)\nonumber\\
	&\leqslant P\left(X^{\bar{d}+1}_l\leqslant a_l,\forall l\in\{1,\ldots,m\}\right)+P\left(A^C\right)\nonumber\\
	&\leqslant \prod_{l=1}^m \left(G^{\bar{d}+1}_c(\delta_1)\right)([0,a_l])+P(A^C)\nonumber\\
	&\leqslant \prod_{l=1}^m\left(\nu_c^*([0,a_l])+\frac{\epsilon}{2}\right)+\frac{\epsilon}{2}, ~ \forall s\geqslant\bar{s},
\end{align*}
and
\begin{align*}
	P\left(N_{k_{l}}\left(\left(\tau_{m}^{s}\right)^-\right)\leqslant a_l,\forall l\in\{1,\ldots,m\}\right)&\geqslant P\left(\left\{N_{k_{l}}\left(\left(\tau_{m}^{s}\right)^-\right)\leqslant a_l,\forall l\in\{1,\ldots,m\}\right\}\cap B\right)\nonumber\\
	&\geqslant P\left(\left\{X_l^{\tilde{d}+1}\leqslant a_l,\forall l\in\{1,\ldots,m\}\right\}\cap B\right)\nonumber\\
	&\geqslant P\left(X_l^{\tilde{d}+1}\leqslant a_l,\forall l\in\{1,\ldots,m\}\right)-P\left(B^C\right)\nonumber\\
	&\geqslant \prod_{l=1}^m \left(G^{\bar{d}+1}_c(\delta_1)\right)([0,a_l])-P(B^C)\nonumber\\
	&\geqslant \prod_{l=1}^m \nu_c^*([0,a_l])-\epsilon, ~ \forall s\geqslant\bar{s}.
\end{align*}
Since $\epsilon$ can be arbitrarily small, we have
$$\lim_{s\to\infty}P\left(N_{k_{l}}\left(\left(\tau_{m}^{s}\right)^-\right)\leqslant a_l,~\forall l\in\{1,\ldots,m\}\right)= \prod_{l=1}^m\nu_c^*([0,a_l]),$$
which implies that $N_{k_l}((\tau_{m}^{s})^{-})$, $l\in\{1,\ldots,m\}$ are asymptotically independent random variables as $s\to\infty$.

If $N_k(0)=\infty$ for all $k\in\{1,\ldots,s\}$, the result in Theorem \ref{Theoexistinflambda} follows from  a similar argument and Lemma \ref{lemmabefore55}.
\end{proof}

\section{Existence of the Solution on $\mathcal{S}$ with Infinite Mean to the RDE (\ref{theoremequ1})}
In this section, we will show the existence of the solution in $\mathcal{S}$ with infinite mean to the RDE (\ref{theoremequ1}). Lemma \ref{lemmasrfunction} to Lemma \ref{uniqueulemma} will show some results used later in the proof of Lemma \ref{lemmaincrease} and Lemma~\ref{lemmadecrease}. Lemma \ref{lemmaincrease} and Lemma \ref{lemmadecrease} will show that under some conditions, there exists a distribution $\mu_0\in\mathcal{S}$ with infinite mean such that $G_c(\mu_0)\succeq\mu_0$ and a distribution $\nu_0\in\mathcal{S}$ with infinite mean such that $G_c(\nu_0)\preceq\nu_0$. Theorem \ref{Theoconvlambda1} will be proved at the end of this section.

\begin{Lemma}\label{lemmasrfunction}
Suppose $S(x)$ and $R(x)$ are two functions that satisfy the following properties:
\begin{enumerate}
	\item $R(0)=0$ and $R(1)=1$.
	\item $S(x)$ and $R(x)$ have Taylor series $S(x)=\sum_{n=1}^\infty s_nx^n$ and $R(x)=\sum_{n=1}^\infty r_nx^n$ centered at $x=0$ for $|x|<1$. There exists $N\in\mathbb{N}^+$ such that $r_n\geqslant 0$ for all $n\geqslant N$.
	\item For all $x\in[0,1)$, functions $S(x)$ and $R(x)$ satisfy the equation
\begin{equation}\label{SRPGF}
	S(x)=\frac{1}{c}E\left[\frac{R(x)-R((1-X)x)}{X^2}\right]-\frac{x}{c}E\left[\frac{R(X+(1-X)x)-R((1-X)x)}{X^2}\right]+R(x),
\end{equation}
where $c$ is a positive constant and $X$ is a random variable on $[0,1]$ with $E[1/X]<\infty$.
\end{enumerate}
Then 
$$s_n=\frac{\lambda_n+c}{c} r_n-\frac{1}{c} \sum_{i=n+1}^{\infty} \binom{i}{i-n+1}\lambda_{i, i-n+1} r_i, ~\forall n\in\mathbb{N}^+. $$
\end{Lemma}

\begin{proof}
We first show that the two expectations on the right-hand side of $(\ref{SRPGF})$ are finite for all $x\in[0,1)$. 
Let 
$$\bar{R}(x)\coloneqq \sum_{n=1}^\infty |r_n|x^n.$$
Since $r_n\geqslant 0$ for all $n\geqslant N$ and $\sum_{n=1}^\infty r_n=R(1)=1$, the series $\sum_{n=1}^\infty |r_n|$ is convergent.
It follows that 
$$ | R(x)|\leqslant\bar{R}(x)\leqslant \sum_{n=1}^\infty |r_n| ,~\forall x\in[0,1],$$
and
$$\sup_{a\in[0,x]}|R'(a)|=\sup_{a\in[0,x]}\left|\sum_{n=1}^\infty nr_n a^{n-1}\right|\leqslant \sum_{n=1}^\infty n|r_n| x^{n-1}  =\bar{R}'(x)
<\infty,~\forall x\in[0,1).$$
Since $X$ is a random variable on $[0,1]$, we have
$$0\leqslant(1-X)x\leqslant x\leqslant X+(1-X)x\leqslant1,~\forall x\in[0,1].$$
Then
	\begin{equation*}
			\left|E\left[\frac{R(x)-R((1-X)x)}{X^2}\right]\right|\leqslant x\sup_{a\in[0,x]}\bar{R}'(a)E\left[\frac{1}{X}\right]<\infty, ~\forall x\in[0,1),
	\end{equation*}
	and
	\begin{align*}
			\left| E\left[\frac{R(X+(1-X)x)-R((1-X)x)}{X^2}\right]\right|&\leqslant\left|E\left[\frac{R(X+(1-X)x)-R((1-X)x)}{X^2}\bold{1}_{\{X>1/2\}}\right]\right|\\
			&\quad\quad+\left|E\left[\frac{R(X+(1-X)x)-R((1-X)x)}{X^2}\bold{1}_{\{X\leqslant 1/2\}}\right]\right|\\
			&\leqslant \frac{ 2\sum_{n=1}^\infty |r_n|}{(1/2)^2}+\sup_{a\in[0,1/2+x/2]}\bar{R}'\left(a\right)E\left[\frac{1}{X}\right]\\
			&<\infty, ~\forall x\in[0,1).
	\end{align*}

We now prove the result of the lemma. 
Since $R(x)=\sum_{n=1}^\infty r_nx^n$ is finite on $[0,1]$, we have
$$R(x)-R((1-X)x)=\sum_{n=1}^\infty r_nx^n-\sum_{n=1}^\infty r_n(1-X)^nx^n=\sum_{n=1}^\infty r_nx^n(1-(1-X)^n),~\forall x\in[0,1],$$ 
and 
\begin{align*}
	R(X+(1-X)x)-R((1-X)x)&=\sum_{n=1}^\infty r_n(X+(1-X)x)^n-\sum_{n=1}^\infty r_n((1-X)x)^n\\
	&=\sum_{n=1}^\infty r_n\left((X+(1-X)x)^n-((1-X)x)^n\right),~\forall x\in[0,1].
\end{align*}
By the Monotone Convergence Theorem, we have
\begin{align}\label{srdiff1}
	E\left[\frac{R(x)-R((1-X)x)}{X^2}\right]&=E\left[\sum_{n=1}^\infty r_nx^n \frac{1-(1-X)^n}{X^2}\right]\nonumber\\
	&=E\left[\sum_{n=1}^\infty r_nx^n \sum_{k=1}^n \binom{n}{k} X^{k-2}(1-X)^{n-k}\right]\nonumber\\
	&=\sum_{n=1}^\infty\left( \sum_{k=1}^n \binom{n}{k} E\left[X^{k-2}(1-X)^{n-k}\right]\right)r_nx^n,
\end{align}
and
\begin{align}\label{srdiff2}
	xE\left[\frac{R(X+(1-X)x)-R(1-X)x)}{X^2}\right]&=xE\left[\sum_{i=1}^\infty r_i \frac{(X+(1-X)x)^i-((1-X)x)^i}{X^2}\right]\nonumber\\
	&=xE\left[\sum_{i=1}^\infty r_i \sum_{n=0}^{i-1} \binom{i}{i-n} X^{i-n-2}(1-X)^{n}x^n\right]\nonumber\\
	&=x\sum_{n=0}^\infty  \left(\sum_{i=n+1}^{\infty} \binom{i}{i-n} E\left[X^{i-n-2}(1-X)^{n}\right]r_i\right)x^n\nonumber\\
	&=\sum_{n=1}^\infty  \left(\sum_{i=n}^{\infty} \binom{i}{i-n+1} E\left[X^{i-n-1}(1-X)^{n-1}\right]r_i\right)x^n.
\end{align}
Note that the coefficients of $x^n$ for the Taylor series on both sides of $(\ref{SRPGF})$ must be the same. Therefore, from $(\ref{SRPGF})$, 
$(\ref{srdiff1})$ and $(\ref{srdiff2})$, we have
\begin{align*}
	s_n&=\frac{1}{c}\sum_{k=1}^n \binom{n}{k} E\left[X^{k-2}(1-X)^{n-k}\right]r_n-\frac{1}{c}\sum_{i=n}^{\infty} \binom{i}{i-n+1} E\left[X^{i-n-1}(1-X)^{n-1}\right]r_i+r_n\nonumber\\
	&=
	\begin{cases}
      r_1-\frac{1}{c}\sum\limits_{i=2}^{\infty}  E\left[X^{i-2}\right]r_i,&n=1,\\
      \left(\frac{1}{c}\sum\limits_{k=2}^{n}\binom{n}{k}E\left[X^{k-2}(1-X)^{n-k}\right]+1\right)r_n\\
      \quad\quad\quad\quad-\frac{1}{c}\sum\limits_{i=n+1}^\infty \binom{i}{i-n+1}E\left[X^{i-n-1}(1-X)^{n-1}\right]r_i,&n\geqslant 2.
    \end{cases} 
\end{align*}
The result of the lemma follows from $(\ref{lambda11})$, $(\ref{lambda22})$ and $(\ref{lambda33})$.
\end{proof}

 Note that for $\alpha>0$, the Taylor series for the function $1-(1-x)^{\alpha}$ centered at $x=0$, where $|x|\leqslant1$, is 
$$1-(1-x)^{\alpha}=\sum_{n=1}^\infty \binom{\alpha}{n} (-1)^{n+1}x^n,$$
where the generalized binomial coefficients are defined by
$$\binom{\alpha}{n}\coloneqq\frac{\alpha(\alpha-1)\cdots(\alpha-n+1)}{n!}.$$
For $\alpha>0$ and $n\geqslant 1$, let
\begin{equation}\label{def11xalpha122}
\gamma_{\alpha,n}\coloneqq(-1)^{n+1}\binom{\alpha}{n}.
\end{equation}
Then 
\begin{equation}\label{def11xalpha1}
	1-(1-x)^{\alpha}=\sum_{n=1}^\infty \gamma_{\alpha,n}x^n.
\end{equation}
Note that a distribution with probability mass function $p_k=\gamma_{\alpha,n}$ for all $k\in\mathbb{N}^+$ with $\alpha\in(0,1)$ is called a Sibuya distribution, which first appeared in \cite{sibuya1979generalized}. It is well-known (see, for example, Proposition 4 in \cite{kozubowski2018generalized}) that
for any $\alpha>0$ and $\alpha\notin\mathbb{Z}$, we have
\begin{equation}\label{gammasimre}
\lim_{n\to\infty}\frac{\gamma_{\alpha,n}}{n^{-1-\alpha}}=-\frac{1}{\Gamma(-\alpha)}.	
	\end{equation}

We will use this notation in the proofs of Lemmas \ref{sumsn1lemma} -
\ref{lemmadecrease} and Theorem \ref{Theoconvlambda1}.

\begin{Lemma}\label{sumsn1lemma}
	Suppose $E[1/X]<\infty$. Suppose $\{r_n\}_{n=1}^\infty$ is a sequence of nonnegative real numbers satisfying
\begin{align}\label{LTYprob0'rn}
\sum_{n=1}^\infty r_n=1.
\end{align}
Suppose $\{s_n\}_{n=1}^\infty$ is a sequence of real numbers satisfying
\begin{align}\label{LTYprob1'rn}
s_n=\frac{\lambda_n+c}{c} r_n-\frac{1}{c} \sum_{i=n+1}^{\infty} \binom{i}{i-n+1}\lambda_{i, i-n+1} r_i , ~\forall n\in\mathbb{N}^+.
\end{align}
Let $R(x)\coloneqq \sum_{n=1}^\infty r_nx^n$. If $(1-x)R'(x)$ is bounded on $[0,1)$ and $s_n$ is positive for all large enough $n$, then  
$$\sum_{n=1}^\infty s_n=1.$$
\end{Lemma}
\begin{proof}
	By multiplying both sides of (\ref{LTYprob1'rn}) by $x^n$, summing over $n$, and following a similar calculation to the one in (\ref{equaa5}) and (\ref{equaa6}) in the proof of Lemma \ref{lemmapgflambda} with $r_n$ playing the role of $P(W=n)$, we have 
\begin{equation}\label{eq11lemma26}
\sum_{n=1}^\infty s_nx^n= \frac{1}{c}E\left[\frac{R(x)-R((1-X)x)-x(R(X+(1-X)x)-R((1-X)x))}{X^2}\right]+R(x),~\forall x\in[0,1).
\end{equation}
Note that 
\begin{align*}
R(x)-R((1-X)x)&-x(R(X+(1-X)x)-R((1-X)x))\\
&=(1-x)(R(x)-R((1-X)x))+x(R(x)-R(X+(1-X)x)).
\end{align*}
Since $\{r_n\}_{n=1}^\infty$ is a sequence of nonnegative real numbers, we have that $R(x)$ is a nondecreasing function with nondecreasing first derivative. Since $(1-X)x\leqslant x\leqslant X+(1-X)x$, we have
$$0\leqslant R(x)-R((1-X)x)\leqslant R'(x)Xx,$$
and 
$$ 0\geqslant R(x)-R(X+(1-X)x)\geqslant -R'(X+(1-X)x)X(1-x).$$
Therefore,
\begin{align}\label{eq22lemma26}
&\frac{R(x)-R((1-X)x)-x(R(X+(1-X)x)-R((1-X)x))}{X^2}\nonumber\\
&\quad\quad\quad\quad\geqslant -4\cdot\bold{1}_{\{X\geqslant1/2\}}+\bold{1}_{\{X<1/2\}}\frac{-x(1-x)R'(X+(1-X)x)}{X},
\end{align}
and 
\begin{equation}\label{eq33lemma26}
\frac{R(x)-R((1-X)x)-x(R(X+(1-X)x)-R((1-X)x))}{X^2}\leqslant \frac{x(1-x)R'(x)}{X}.
\end{equation}
Note that
$$\bold{1}_{\{X<1/2\}}(1-x)R'(X+(1-X)x)\leqslant (1-x)R'\left(\frac{1}{2}+\frac{x}{2}\right)=2\left(1-\frac{1}{2}-\frac{x}{2}\right)R'\left(\frac{1}{2}+\frac{x}{2}\right).$$
Since $(1-x)R'(x)$ is bounded on $[0,1)$, the right-hand side of (\ref{eq22lemma26}) is bounded below by an integrable random variable and the right-hand side of (\ref{eq33lemma26}) is bounded above by an integrable random variable. It follows from the Dominated Convergence Theorem that 
\begin{equation}\label{eq44lemma26}
 \lim_{x\to1^-}E\left[\frac{R(x)-R((1-X)x)-x(R(X+(1-X)x)-R((1-X)x))}{X^2}\right]=0.
\end{equation}
Since $\{s_n\}_{n=1}^\infty$ is eventually positive, by (\ref{eq11lemma26}) and (\ref{eq44lemma26}), we have 
$$\sum_{n=1}^\infty s_n=\lim_{x\to1^-}\sum_{n=1}^\infty s_nx^n=R(1)=\sum_{n=1}^\infty r_n.$$
The result follows from (\ref{LTYprob0'rn}).
\end{proof}

\begin{Lemma}\label{uniqueulemma}
	Suppose $E[1/X]<\infty$. Suppose $\{r_n\}_{n=1}^\infty$ is a sequence of nonnegative real numbers such that $\sum_{n=1}^\infty r_n$ is convergent  
and there exists $\alpha\in(0,1)$ and $b>0$ such that 
\begin{align}\label{LTYprob00'}
\frac{r_{n-1}-r_n}{n^{-2-\alpha}}\to b, \text{ as }n\to\infty.
\end{align}
Suppose $T$ is a random variable on $\mathbb{N}^+$ with
\begin{align}\label{LTYprob1'}
P(T=n)=\frac{\lambda_n+c}{c} r_n-\frac{1}{c} \sum_{i=n+1}^{\infty} \binom{i}{i-n+1}\lambda_{i, i-n+1} r_i , ~\forall n\in\mathbb{N}^+.
\end{align}
Let $(\Pi_t,t\geqslant 0)$ be a $\Lambda$-coalescent characterized by the probability measure of $X$, and let $(\Pi_t,t\geqslant 0)$ be its restriction to $[n]$. Let $L_n(t)$ denote the number of blocks in $\Pi^n_t$. Suppose $T$, $Y$, and $\Pi$ are independent.
Then 
$$P(L_T(Y)=n)=r_n,~\forall n\in\mathbb{N}^+.$$
\end{Lemma}
\begin{proof}
Note that $L_T(Y)\leqslant T$. By the independence of $T$, $Y$, and $\Pi$, we have 
\begin{equation}\label{L4LTYninfsum1}
P(L_T(Y)=n)=\sum_{i=n}^\infty P(T=i)P(L_T(Y)=n|T=i)=\sum_{i=n}^\infty P(T=i)P(L_i(Y)=n),~\forall n\in\mathbb{N}^+.
\end{equation}

We first derive an iterative formula for $P(L_{i}(Y)=k)$. Suppose that we are starting from $i$ blocks. Let $B_{i,n}$ be the event that the first merger happens before $Y$ and makes the number of blocks decrease from $i$ to $n$. Then for $i>k$, we have
\begin{equation*}\label{SplitBij}
	\{L_i(Y)=k\}=\bigcup_{n=k}^{i-1}\left(\{L_{i}(Y)=k\}\cap B_{i,n}\right).
\end{equation*}
Note that the transition rate from $i$ to $n$ is $\binom{i}{i-n+1}\lambda_{i, i-n+1}$, and the rate for the number of blocks to decrease is $\lambda_i$ when there are $i$ blocks. Recall that $Y$ is an exponential random variable with rate $c$. Therefore, 
\begin{equation*}\label{probBij}
	P(B_{i,n})=\frac{\binom{i}{i-n+1}\lambda_{i, i-n+1}}{\lambda_i+c},
\end{equation*}
and
\begin{equation}\label{probLkk1}
	P(L_i(Y)=i)=\frac{c}{\lambda_i+c}.
\end{equation}
Once the number of blocks reaches $n$, because of the memoryless property of exponential random variables, the probability for the number of blocks to reach $k$ at time $Y$ is $P(L_n(Y)=k)$. Then for $i>k$, we have
\begin{equation}\label{LiYn}
	P(L_{i}(Y)=k)=\sum_{n=k}^{i-1}P\left(L_{i}(Y)=k|B_{i,n}\right)P(B_{i,n})=\frac{1}{\lambda_i+c}\sum_{n=k}^{i-1}\binom{i}{i-n+1}\lambda_{i, i-n+1}P(L_{n}(Y)=k).
\end{equation}

We now prove $P(L_T(Y)=n)=r_n$ for all large enough $n$.
For $n\geqslant 2$, let 
\begin{equation}\label{vndef1new}
	v_n\coloneqq r_{n-1}-r_n.
\end{equation}
It follows from the convergence of $\sum_{n=1}^\infty r_n$ that $\lim_{n\to\infty}r_n=0$, and from $(\ref{LTYprob00'})$ that there exists $N\in\mathbb{N}^+$ such that $0<v_n< 2bn^{-2-\alpha}$ for all $n>N$. Therefore,
\begin{equation}\label{unsumvnrnl5}
r_n=\sum_{i=n+1}^\infty v_i,~\forall n\in\mathbb{N}^+,
\end{equation}
and there exists $M_1>0$ such that $r_n<M_1n^{-1-\alpha}$ for all $n\in\mathbb{N}^+$.
By $(\ref{LTYprob1'})$ and $(\ref{unsumvnrnl5})$, we have
\begin{align}\label{L2PTn1rnnew} 
P(T=n)&=r_n+\frac{\lambda_n}{c} \sum_{j=n+1}^\infty v_j-\frac{1}{c} \sum_{i=n+1}^{\infty} \binom{i}{i-n+1}\lambda_{i, i-n+1} \sum_{j=i+1}^\infty v_j\nonumber\\
&=r_n+\frac{\lambda_n}{c} v_{n+1}+ \frac{1}{c}\sum_{j=n+2}^\infty  v_j \lambda_n-\frac{1}{c}\sum_{j=n+2}^\infty  v_j\sum_{i=n+1}^{j-1} \binom{i}{i-n+1}\lambda_{i, i-n+1}.
\end{align}
Note that $\sum_{j=n+2}^\infty  v_j \lambda_n=\lambda_n r_{n+1}<\infty$. Then the first three terms on the right-hand side of (\ref{L2PTn1rnnew}) are finite and the finiteness of the last term follows from $P(T=n)<\infty$. 
Therefore, by $(\ref{L2PTn1rnnew})$, we have 
\begin{align}\label{L2PTn1new} 
P(T=n)=r_n+\frac{1}{c}\lambda_n v_{n+1}+\frac{1}{c}\sum_{j=n+2}^\infty  v_j \left(\lambda_n- \sum_{i=n+1}^{j-1} \binom{i}{i-n+1}\lambda_{i, i-n+1}\right).
\end{align}
Note that the result $P(L_T(Y)=n)=r_n$ follows by combining $r_n=\sum_{i=n}^\infty P(T=i)P(L_i(Y)=n)$ with (\ref{L4LTYninfsum1}). Now we show $r_n=\sum_{i=n}^\infty P(T=i)P(L_i(Y)=n)$. We will do this by showing that the right-hand side can be separated into the sum of several convergent series, and then show the two sides are equal.
Note that 
\begin{equation*}
	\lambda_{n}=\sum_{k=2}^{n}\binom{n}{k}E\left[X^{k-2}(1-X)^{n-k}\right]=E\left[\frac{1-(1-X)^n-nX(1-X)^{n-1}}{X^2}\right]< nE\left[\frac{1}{X}\right].
\end{equation*}
Recall that  $0<v_n< 2bn^{-2-\alpha}$ for all $n>N$. Then
$$0<\sum_{n=k}^\infty  \lambda_n v_{n+1}<\sum_{n=k}^\infty nE\left[\frac{1}{X}\right] 2bn^{-2-\alpha}<\infty,~\forall k>N.$$ 
Therefore,
\begin{equation}\label{convergesumfl1}
	\sum_{n=k}^\infty  \lambda_n v_{n+1}P(L_n(Y)=k)\text{ is convergent},~\forall k\in\mathbb{N}^+.
\end{equation}
Let 
\begin{equation}\label{deffkl4newnew}
f_k\coloneqq \sum_{n=k}^\infty \sum_{j=n+2}^\infty   \sum_{i=n+1}^{j-1}v_j\binom{i-1}{n-1}E\left[X^{i-n-1}(1-X)^{n-1}\right].
\end{equation}
We now prove that $f_k$ is finite for all $k\in\mathbb{N}^+$. Since $v_j>0$ for all $j>N$,
by Tonelli's Theorem, $(\ref{unsumvnrnl5})$ and $(\ref{deffkl4newnew})$, we have
\begin{equation}\label{boundgkl41rnnew}
f_k=\sum_{n=k}^\infty\sum_{i=n+1}^\infty \binom{i-1}{n-1}E\left[X^{i-n-1}(1-X)^{n-1}\right]r_i\geqslant0,~\forall k>N.
\end{equation}
Since $\alpha\in(0,1)$, we have $\gamma_{\alpha,n}>0$ for all $n\in\mathbb{N}^+$.
It then follows from (\ref{gammasimre}) that there exists $M_2>0$ such that $i^{-1-\alpha}<M_2\gamma_{\alpha,i}$ for all $i\in\mathbb{N}^+$. Recall that $r_i<M_1i^{-1-\alpha}$ for all $i\in\mathbb{N}^+$.    
Then 
\begin{equation}\label{unboundgammarnnew}
	r_i<  M_1M_2\gamma_{\alpha,i}<  M_1M_2\gamma_{\alpha,i-1},~\forall i\in\mathbb{N}^+.
\end{equation}
Let $R(x)=1-(1-x)^\alpha$. Recall from $(\ref{def11xalpha1})$ that $1-(1-x)^\alpha=\sum_{n=1}^\infty \gamma_{\alpha,n}x^n$. Then
$$R(X+(1-X)x)
=1-(1-X)^\alpha(1-x)^\alpha=1-(1-X)^\alpha+(1-X)^\alpha \sum_{n=1}^\infty \gamma_{\alpha,n}x^n,$$ 
and 
\begin{align*}
R(X+(1-X)x)&=\sum_{i=1}^\infty \gamma_{\alpha,i}(X+(1-X)x)^i\\
&=\sum_{n=1}^\infty (1-X)^n x^n\sum_{i=n}^\infty\gamma_{\alpha,i}\binom{i}{n}X^{i-n}+\sum_{i=1}^\infty\gamma_{\alpha,i}X^{i}.
\end{align*} 
Since the coefficient of the term $x^{n-1}$ in the Taylor expansion of $R(X+(1-X)x)$ centered at 0 is unique, for all $n\geqslant2$, we have
$$(1-X)^\alpha \gamma_{\alpha,n-1}=(1-X)^{n-1} \sum_{i=n-1}^\infty\gamma_{\alpha,i}\binom{i}{n-1}X^{i-(n-1)}=(1-X)^{n-1} \sum_{i=n}^\infty\gamma_{\alpha,i-1}\binom{i-1}{n-1}X^{i-n}.$$
It follows from the Monotone Convergence Theorem that for $n\geqslant 2$, 
\begin{align}\label{gammanchangel21rnnew}
	E\left[\frac{(1-X)^\alpha}{X}\right]\gamma_{\alpha,n-1}
	=\sum_{i=n}^\infty\binom{i-1}{n-1}E\left[X^{i-n-1}(1-X)^{n-1}\right]\gamma_{\alpha,i-1}.
\end{align}
By $(\ref{unboundgammarnnew})$ and $(\ref{gammanchangel21rnnew})$, we have
\begin{equation}\label{boundgkl42rnnew}
\sum_{i=n+1}^\infty \binom{i-1}{n-1}E\left[X^{i-n-1}(1-X)^{n-1}\right]r_i\leqslant M_1M_2E\left[\frac{(1-X)^\alpha}{X}\right]\gamma_{\alpha,n-1},~\forall n\geqslant 2.
\end{equation}
Since $\sum_{n=1}^\infty\gamma_{\alpha,n}=R(1)=1$, by $(\ref{boundgkl41rnnew})$ and $(\ref{boundgkl42rnnew})$, we have
\begin{equation*}\label{conevergentfkl4new}
	0\leqslant f_k\leqslant M_1M_2	E\left[\frac{(1-X)^\alpha}{X}\right]\sum_{n=k}^\infty \gamma_{\alpha,n-1}<\infty,~\forall k>N.
\end{equation*}
Note that 
for $n>1$,
\begin{align}\label{beforenegabinom}
0\leqslant \sum_{i=n+1}^{j-1}\binom{i-1}{n-1}E\left[X^{i-n-1}(1-X)^{n-1}\right]
\leqslant \frac{j}{n-1}E\left[\sum_{i=n+1}^{\infty}\binom{i-2}{i-n}\frac{X^{i-n}(1-X)^{n-1}}{X}\right].
\end{align}
It is a well-known result that for $p\in[0,1]$,
\begin{equation}\label{negabinom111}
\sum_{k=0}^\infty\binom{k+r-1}{k}(1-p)^kp^r=1.	
\end{equation}
Combining (\ref{beforenegabinom}) and (\ref{negabinom111}), for $n>1$, we have
$$0\leqslant \sum_{i=n+1}^{j-1}\binom{i-1}{n-1}E\left[X^{i-n-1}(1-X)^{n-1}\right]<\frac{j}{n-1}E\left[\frac{1}{X}\right].$$
Recall that $0<v_j< 2bj^{-2-\alpha}$ for all $j>N$. Then
$$ \sum_{j=n+2}^\infty   \sum_{i=n+1}^{j-1}  |v_j|\binom{i-1}{n-1}E\left[X^{i-n-1}(1-X)^{n-1}\right]<\frac{1}{n-1}E\left[\frac{1}{X}\right]\sum_{j=n+2}^\infty j|v_j|<\infty,~\forall n>1,$$
and 
$$\sum_{j=3}^\infty   \sum_{i=2}^{j-1}|v_j|E\left[X^{i-2}\right]\leqslant \sum_{j=3}^\infty j|v_j|<\infty.$$
Since $|v_j|=v_j$ for all $j>N$ and $f_k$ is finite for all $k>N$, we have
\begin{equation}\label{conevergentfkl4new2}
\sum_{n=k}^\infty \sum_{j=n+2}^\infty   \sum_{i=n+1}^{j-1}|v_j|\binom{i-1}{n-1}E\left[X^{i-n-1}(1-X)^{n-1}\right]<\infty,~\forall k\in\mathbb{N}^+.
\end{equation}
Let 
$$f'_k\coloneqq\sum_{n=k}^\infty \sum_{j=n+2}^\infty   \sum_{i=n+1}^{j-1}v_j\binom{i-1}{n-1}E\left[X^{i-n-1}(1-X)^{n-1}\right]P(L_n(Y)=k).$$
By $(\ref{conevergentfkl4new2})$, we have
\begin{equation}\label{convegegk1}
	f'_k\text{ is finite},~\forall k\in\mathbb{N}^+.
\end{equation}
Let
$$g(n,j)\coloneqq v_j \left(\lambda_n- \sum_{i=n+1}^{j-1} \binom{i}{i-n+1}\lambda_{i, i-n+1}+\sum_{i=n+1}^{j-1}\binom{i-1}{i-n}E\left[X^{i-n-1}(1-X)^{n-1}\right]\right).$$
Therefore, by $(\ref{L2PTn1new})$, $(\ref{convergesumfl1})$, $(\ref{convegegk1})$ and the fact that $\sum_{n=1}^\infty r_n$ is convergent, for all $k\in\mathbb{N}^+$, 
we have
\begin{align}\label{splitnewl51}
\sum_{n=k}^\infty P(T=n)P(L_n(Y)=k)&=\sum_{n=k}^\infty r_nP(L_n(Y)=k)+\frac{1}{c}\sum_{n=k}^\infty  \lambda_n v_{n+1}P(L_n(Y)=k)\nonumber\\
&\quad\quad\quad\quad-\frac{f'_k}{c}+\frac{1}{c}\sum_{n=k}^\infty \sum_{j=n+2}^\infty g(n,j)P(L_n(Y)=k).
\end{align}
Note that 
\begin{align*}
	\lambda_{n}
	=E\left[\frac{1-(1-X)^n-nX(1-X)^{n-1}}{X^2}\right]=E\left[\frac{1-(1-X)^{n-1}-(n-1)X(1-X)^{n-1}}{X^2}\right].
\end{align*}
Using (\ref{negabinom111}),  
it follows that
$$\lambda_n=\sum_{i=n+1}^\infty \binom{i-1}{n-2}E\left[X^{i-n-1}(1-X)^{n-1}\right].$$
Recall from $(\ref{lambda11})$ that $\lambda_{i, i-n+1}=E\left[X^{i-n-1}(1-X)^{n-1}\right]$.
Therefore, 
\begin{align*}\label{Tonelliposirnnew}
	\lambda_n&- \sum_{i=n+1}^{j-1}\binom{i}{i-n+1}\lambda_{i, i-n+1}\nonumber\\
	&=\left(\sum_{i=n+1}^\infty \binom{i-1}{n-2}E\left[X^{i-n-1}(1-X)^{n-1}\right]-\sum_{i=n+1}^{j-1} \binom{i}{i-n+1}E\left[X^{i-n-1}(1-X)^{n-1}\right]\right)\nonumber\\
	&\geqslant \left(\sum_{i=n+1}^{j-1} \binom{i-1}{i-n+1}E\left[X^{i-n-1}(1-X)^{n-1}\right]-\sum_{i=n+1}^{j-1} \binom{i}{i-n+1}E\left[X^{i-n-1}(1-X)^{n-1}\right]\right)\nonumber\\
	&= -\sum_{i=n+1}^{j-1}\binom{i-1}{i-n}E\left[X^{i-n-1}(1-X)^{n-1}\right].
\end{align*}
Since $0<v_j$ for all $j>N$, we have $g(n,j)\geqslant 0$ for all $j>N$. By Tonelli's Theorem, for $k>N$, we have
\begin{align*}
	\sum_{n=k}^\infty& \sum_{j=n+2}^\infty g(n,j)P(L_n(Y)=k)=\sum_{j=k+2}^\infty\sum_{n=k}^{j-2}g(n,j)P(L_n(Y)=k)\\
	&=\sum_{j=k+2}^\infty v_j\left( \sum_{n=k}^{j-2}\lambda_n P(L_n(Y)=k)-\sum_{i=k+1}^{j-1}\sum_{n=k}^{i-1} \binom{i}{i-n+1}\lambda_{i, i-n+1}P(L_n(Y)=k)\right.\\
	&\quad\quad\quad\quad\quad\quad\quad\quad+\left.\sum_{n=k}^{j-2}\sum_{i=n+1}^{j-1}\binom{i-1}{i-n}E\left[X^{i-n-1}(1-X)^{n-1}\right]P(L_n(Y)=k)\right),
\end{align*}
and
$$\sum_{j=k+2}^\infty v_j\sum_{n=k}^{j-2}\sum_{i=n+1}^{j-1}\binom{i-1}{i-n}E\left[X^{i-n-1}(1-X)^{n-1}\right]P(L_n(Y)=k)=f'_k.$$
It follows from $(\ref{probLkk1})$ and $(\ref{LiYn})$ that 
\begin{align*}
\sum_{n=k}^{j-2}\lambda_n &P(L_n(Y)=k)-\sum_{i=k+1}^{j-1}\sum_{n=k}^{i-1} \binom{i}{i-n+1}\lambda_{i, i-n+1}P(L_n(Y)=k)\\
&=\sum_{n=k}^{j-2}\lambda_n P(L_n(Y)=k)-\sum_{i=k+1}^{j-1}(\lambda_i+c)P(L_i(Y)=k)\\
&=\frac{c\lambda_k}{\lambda_k+c}-\lambda_{j-1}P(L_{j-1}(Y)=k)-c\sum_{i=k+1}^{j-1}P(L_i(Y)=k).
\end{align*}
Recall from (\ref{convergesumfl1}) and (\ref{convegegk1}) that $\sum_{n=k}^\infty  \lambda_n v_{n+1}P(L_n(Y)=k)$ and $f'_k$ are finite for all $k\in\mathbb{N}^+$. Then for all $k>N$,
\begin{align*}
	\sum_{n=k}^\infty& \sum_{j=n+2}^\infty  g(n,j)P(L_n(Y)=k)\\
	&=\sum_{j=k+2}^\infty v_j\left(\frac{c\lambda_k}{\lambda_k+c}-\lambda_{j-1}P(L_{j-1}(Y)=k)-c\sum_{i=k+1}^{j-1}P(L_i(Y)=k)\right)+f'_k\\
	&=\frac{c\lambda_k}{\lambda_k+c}\sum_{j=k+2}^\infty v_j-\sum_{n=k+1}^\infty \lambda_n v_{n+1}P(L_n(Y)=k)-c\sum_{i=k+1}^\infty P(L_i(Y)=k)\sum_{j=i+1}^\infty v_j+f'_k.
\end{align*}
It follows from $(\ref{probLkk1})$, $(\ref{vndef1new})$, $(\ref{unsumvnrnl5})$ and $(\ref{splitnewl51})$ that
\begin{align*}
	&\sum_{n=k}^\infty P(T=n)P(L_n(Y)=k)\\
	&\quad\quad\quad\quad=\sum_{n=k}^\infty r_nP(L_n(Y)=k)+\frac{1}{c}\sum_{n=k}^\infty  \lambda_n v_{n+1}P(L_n(Y)=k)-\frac{f'_k}{c}\\
	&\quad\quad\quad\quad\quad\quad+\frac{\lambda_k r_{k+1}}{\lambda_k+c}-\frac{1}{c}\sum_{n=k+1}^\infty  \lambda_n v_{n+1}P(L_n(Y)=k)-\sum_{i=k+1}^\infty P(L_i(Y)=k)r_i+\frac{f_k'}{c}\\
	&\quad\quad\quad\quad=r_kP(L_k(Y)=k)+\frac{1}{c}\lambda_kv_{k+1}P(L_k(Y)=k)+\frac{\lambda_k r_{k+1}}{\lambda_k+c}\\
	&\quad\quad\quad\quad=r_k,~\forall k>N.
\end{align*}

It remains to prove $P(L_T(Y)=n)=r_n$ for $n\leqslant N$. 
By $(\ref{LTYprob1'})$, we have 
$$r_n=\frac{1}{\lambda_n+c}\sum_{i=n+1}^{\infty}  \binom{i}{i-n+1}\lambda_{i, i-n+1}r_i+\frac{c}{\lambda_n+c}P(T=n), ~\forall n\in\mathbb{N}^+.$$
By Lemma \ref{lemmadistLTY}, we have 
$$P(L_T(Y)=n)=\frac{1}{\lambda_n+c}\sum_{i=n+1}^{\infty}  \binom{i}{i-n+1}\lambda_{i, i-n+1}P(L_T(Y)=i)+\frac{c}{\lambda_n+c}P(T=n), ~\forall n\in\mathbb{N}^+.$$
Since $P(L_T(Y)=n)=r_n$ for all $n> N$, by replacing $n$ in the two above equations with $N$, the right-hand sides of the two equations are the same, which implies $P(L_T(Y)=N)=r_{N}$. Therefore, by replacing $n$ in the two equations with $N-1,\ldots,1$, we obtain $P(L_T(Y)=n)=r_n$ for all $n\leqslant N$. The result follows.
\end{proof}

Suppose that $0<c<E[1/X]<\infty$. It is easy to see that $E[(1-(1-X)^\alpha)/X^2]$ is monotonically increasing as a function of $\alpha$ and takes values in $[0,E[1/X]]$ for $\alpha\in[0,1]$. Then there exists a unique constant $\alpha_c\in(0,1)$ such that
\begin{equation}\label{alphac}
	E\left[\frac{1-(1-X)^{\alpha_c}}{X^2}\right]=c.
\end{equation}
We will use this notation in the proofs of Lemma \ref{lemmaincrease}, Lemma \ref{lemmadecrease} and Theorem \ref{Theoconvlambda1}.

\begin{Lemma}
\label{lemmaincrease}
Suppose that $0<c<E[1/X]<\infty$. Then there exists a distribution $\mu_0\in\mathcal{S}$ with infinite mean such that 
\begin{equation}\label{lemma41}
G_c(\mu_0)\succeq \mu_0.
\end{equation}
\end{Lemma}
\begin{proof}
We will show that there exists a distribution with a probability generating function $R$ that has the following property. Suppose the Taylor series of the function $R^2(x)$ that centered at $x=0$ is $R^2(x)=\sum_{n=1}^\infty \tilde{r}_n x^n$. Suppose the Taylor series of the function $S(x)$ that satisfies $(\ref{SRPGF})$ centered at $x=0$ is $S(x)=\sum_{n=1}^\infty s_n x^n$. Then $\sum_{i=1}^ns_i\geqslant \sum_{i=1}^n\tilde{r}_i$ for all $n\in\mathbb{N}^+$. We will then use this property to show that this distribution satisfies $(\ref{lemma41})$.

Firstly, we will derive the Taylor series. Let 
$$R(x)=(1-a+\varepsilon a)x+a(1-(1-x)^{\alpha_c})-\varepsilon a(1-(1-x)^\beta),$$
where $\beta\in(\alpha_c,2\alpha_c\wedge 1)$, $a\in(0,1]$, and $\varepsilon\in[0,\alpha_c/\beta)$ will be chosen later. Then $R(1)=1$ and
\begin{equation}\label{rangelemma28sw}
	1-a+\varepsilon a<1.
\end{equation}
Recall from $(\ref{def11xalpha1})$ that
	$1-(1-x)^{\alpha}=\sum_{n=1}^\infty \gamma_{\alpha,n}x^n.$
Then the Taylor series for the function $R(x)$ centered at $x=0$ is 
\begin{equation}\label{taylorofR1}
	R(x)=\sum_{n=1}^\infty r_n x^n,
\end{equation}
where
\begin{equation}\label{taylorofR2}
r_n=
\begin{cases}	
1-a+\varepsilon a+a\gamma_{\alpha_c,1}-\varepsilon a \gamma_{\beta,1}, & n=1,\\
a\gamma_{\alpha_c,n}-\varepsilon a \gamma_{\beta,n}, & n\geqslant 2.
\end{cases}
\end{equation}
Since $\varepsilon<\alpha_c/\beta$, we have $\gamma_{\alpha_c,1}=\alpha_c>\varepsilon\beta=\varepsilon\gamma_{\beta,1}$. If $\gamma_{\alpha_c,n}>\varepsilon\gamma_{\beta,n}$, then
$$\gamma_{\alpha_c,n+1}=\frac{n-\alpha_c}{n+1}\gamma_{\alpha_c,n}>\frac{n-\beta}{n+1}\varepsilon\gamma_{\beta,n}=\varepsilon\gamma_{\beta,n+1}.$$
It follows by induction that $r_n>0$ for all $n\geqslant 1$, which implies that $R(x)$ is the probability generating function of a positive integer-valued random variable. 

We can also expand $R(X+(1-X)x)$ and $R((1-X)x)$ at $x=0$. We get
\begin{align}\label{taylorofR12}
R(X+(1-X)x)
&=(a-\varepsilon a-a(1-X)^{\alpha_c}+\varepsilon a(1-X)^{\beta})+(1-a+\varepsilon a)(X+(1-X)x)\nonumber\\
&\quad\quad\quad\quad\quad\quad\quad+a\sum_{n=1}^\infty (1-X)^{\alpha_c}\gamma_{\alpha_c,n} x^n -\varepsilon a\sum_{n=1}^\infty (1-X)^{\beta}\gamma_{\beta,n} x^n,	
\end{align}
and
\begin{align}\label{taylorofR11}
R((1-X)x)=(1-a+\varepsilon a)(1-X)x+a\sum_{n=1}^\infty \gamma_{\alpha_c,n} (1-X)^n x^n -\varepsilon a\sum_{n=1}^\infty \gamma_{\beta,n} (1-X)^n x^n.	
\end{align}
Let 
\begin{equation}\label{Sx1}
S(x)=\frac{1}{c}E\left[\frac{R(x)-R((1-X)x)}{X^2}\right]-\frac{x}{c}E\left[\frac{R(X+(1-X)x)-R((1-X)x)}{X^2}\right]+R(x).	
\end{equation}
Note that 
\begin{equation}\label{gamman_1}
\gamma_{\alpha,n-1} =\frac{n}{n-1-\alpha}\gamma_{\alpha,n}=\gamma_{\alpha,n}+\frac{1+\alpha}{n-1-\alpha}\gamma_{\alpha,n}.
\end{equation}
By $(\ref{taylorofR1})$, $(\ref{taylorofR2})$, $(\ref{taylorofR12})$, $(\ref{taylorofR11})$, $(\ref{Sx1})$ and $(\ref{gamman_1})$, the Taylor series for the function $S(x)$ centered at $x=0$ is
\begin{align*}
	S(x)=\sum_{n=1}^\infty s_n x^n,
\end{align*}
where 
\begin{equation*}
s_n=
\begin{cases}
	\frac{a}{c}E\left[\frac{(1-X)^{\alpha_c}+\alpha_c X-1}{X^2}\right]-\frac{\varepsilon a}{c}E\left[\frac{(1-X)^{\beta}+\beta X-1}{X^2}\right]+a\gamma_{\alpha_c,1}-\varepsilon a\gamma_{\beta,1}+1-a+\varepsilon a, & n=1,\\
	\frac{a}{c}E\left[\frac{1-(1-X)^n-\frac{n(1-X)^{\alpha_c}}{n-1-\alpha_c}+\frac{n(1-X)^{n-1}}{n-1-\alpha_c}}{X^2}\right]\gamma_{\alpha_c,n}-\frac{\varepsilon a}{c}E\left[\frac{1-(1-X)^n-\frac{n(1-X)^{\beta}}{n-1-\beta}+\frac{n(1-X)^{n-1}}{n-1-\beta}}{X^2}\right]\gamma_{\beta,n}\\
	\quad\quad\quad\quad\quad\quad\quad\quad\quad\quad\quad\quad\quad\quad\quad\quad\quad\quad\quad\quad +a\gamma_{\alpha_c,n}-\varepsilon a\gamma_{\beta,n}. & n\geqslant 2.
\end{cases}
\end{equation*}
We also have 
\begin{align*}
	R(x)^2&=\left((1-a+\varepsilon a)x+a-\varepsilon a-a(1-x)^{\alpha_c}+\varepsilon a(1-x)^\beta\right)^2\\
	      &=(1-a+\varepsilon a)^2x^2+2a(a-\varepsilon a)\left(1-(1-x)^{\alpha_c}\right)+2a(1-a+\varepsilon a)x\left(1-(1-x)^{\alpha_c}\right)\\
	      &\quad -2\varepsilon a(a-\varepsilon a)\left(1-(1-x)^\beta\right)-2\varepsilon a(1-a+\varepsilon a)x\left(1-(1-x)^\beta\right) \\
	      &\quad -a^2\left(1-(1-x)^{2\alpha_c}\right)+2\varepsilon a^2\left(1-(1-x)^{\alpha_c+\beta}\right)-\varepsilon^2 a^2 \left(1-(1-x)^{2\beta}\right).
\end{align*}
Note that the Taylor series of $x\left(1-(1-x)^\alpha\right)$ at $x=0$ is $\sum_{n=2}^\infty \gamma_{\alpha,n-1} x^n$.
Let 
\begin{align}\label{eq123lemma281}
	\tilde{r}'_n\coloneqq2a\gamma_{\alpha_c,n}+&\frac{2a(1-a+\varepsilon a)(1+\alpha_c)}{n-1-\alpha_c}\gamma_{\alpha_c,n}-2\varepsilon a\gamma_{\beta,n}-\frac{2\varepsilon a(1-a+\varepsilon a)(1+\beta)}{n-1-\beta}\gamma_{\beta,n}\nonumber\\
	&-a^2\gamma_{2\alpha_c,n}+2\varepsilon a^2\gamma_{\alpha_c+\beta,n}-\varepsilon^2 a^2\gamma_{2\beta,n}.
\end{align}
Then the Taylor series for the function $R(x)^2$ centered at $x=0$ is
\begin{align*}
	R(x)^2&=\sum_{n=1}^\infty \tilde{r}_n x^n,
\end{align*}
where
\begin{equation}\label{eq22slemma282}
\tilde{r}_n=
	\begin{cases}
		0, & n=1,\\
		(1-a+\varepsilon a)^2+\tilde{r}_2', & n=2,\\
		\tilde{r}_n', & n\geqslant3.
	\end{cases}
\end{equation}
Note that $R^2(x)$ is the probability generating function of the sum of two independent random variables with probability generating function $R(x)$, which implies 
$$\tilde{r}_n\geqslant 0,~\forall n\in\mathbb{N}^+.$$ 

Our next goal is to show that for fixed $\beta\in(\alpha_c,2\alpha_c\wedge1)$, there exists $a\in(0,1]$ and $\varepsilon\in[0,\alpha_c/\beta)$ such that
\begin{equation}\label{sumincrease1} 
\sum_{i=1}^ns_i\geqslant \sum_{i=1}^n\tilde{r}_i,~\forall n\in\mathbb{N}^+.
\end{equation} 
Let
$$d_n\coloneqq s_n-\tilde{r}_n,~\forall n\in\mathbb{N}^+.$$
Firstly, we will show $ d_n< 0$ for all $n\geqslant 2$.
For $0<\alpha<1$ and $n\geqslant 2$, let
\begin{equation}\label{defzetal61}
\zeta_{\alpha,n}\coloneqq E\left[\frac{1-(1-X)^n-\frac{n(1-X)^{\alpha}}{n-1-\alpha}+\frac{n(1-X)^{n-1}}{n-1-\alpha}}{X^2}\right]-E\left[\frac{1-(1-X)^\alpha}{X^2}\right].
\end{equation}
Let 
\begin{equation*}\label{Createcbeta}
c_\beta\coloneqq E\left[\frac{1-(1-X)^{\beta}}{X^2}\right].
\end{equation*}
Recall from $(\ref{alphac})$ that $c=E[\left(1-(1-X)^{\alpha_c}\right)/X^2]$.
Let 
\begin{align}\label{diff_n}
d_n'\coloneqq &\frac{a}{c}\left(\zeta_{\alpha_c,n}-\frac{2c(1-a+\varepsilon a)(1+\alpha_c)}{n-1-\alpha_c}\right)\gamma_{\alpha_c,n} -\frac{\varepsilon a}{c}\left(\zeta_{\beta,n}-\frac{2c(1-a+\varepsilon a)(1+\beta)}{n-1-\beta}\right)\gamma_{\beta,n}\nonumber\\
   &\quad\quad\quad\quad-\frac{\varepsilon a}{c}(c_\beta-c)\gamma_{\beta,n}+a^2\gamma_{2\alpha_c,n}-2\varepsilon a^2\gamma_{\alpha_c+\beta,n}+\varepsilon^2 a^2\gamma_{2\beta,n}.
\end{align}
Then 
\begin{equation*}
d_n=
\begin{cases}
	d_2'-(1-a+\varepsilon a)^2, & n=2,\\
	d_n', & n\geqslant 3.
\end{cases}
\end{equation*} 
Therefore, it suffices to show $d_n'<0$ for all $n\geqslant2$.
Note that $\zeta_{\alpha,n}$ can be rewritten as
\begin{align}\label{zeta223}
\zeta_{\alpha,n}=\frac{1}{n-1-\alpha}\bigg((1+\alpha)&E\left[\frac{1-(1-X)^\alpha}{X^2}\right]-\alpha E\left[\frac{1-(1-X)^n}{X^2}\right]\nonumber\\
&-E\left[\frac{1-(1-X)^n-nX(1-X)^{n-1}}{X^2}\right]\bigg).
\end{align}
Recall from (\ref{alphac}) that $c=E[(1-(1-X)^{\alpha_c})/X^2]$. Since $0<\alpha_c<1$ and $0\leqslant X\leqslant 1$, we have
\begin{align*}\label{zeta1}
	\zeta_{\alpha_c,n}
	\leqslant \frac{1+\alpha_c}{n-1-\alpha_c}c,~\forall n\geqslant 2.
\end{align*} 
By choosing
\begin{equation}\label{boundbbb}
	a\in\left(0,\frac{1}{4}\right),
\end{equation}
we can make $1-a+\varepsilon a>3/4$, which implies that for all $n\geqslant2$,
\begin{equation}\label{zetaaupper}
	\zeta_{\alpha_c,n}-\frac{2c(1-a+\varepsilon a)(1+\alpha_c)}{n-1-\alpha_c}\leqslant \frac{(1+\alpha_c)c-2c(1-a+\varepsilon a)(1+\alpha_c)}{n-1-\alpha_c}<-\frac{(1+\alpha_c)c}{2(n-1-\alpha_c)}<0.
\end{equation}
Note that $E[1/X]<\infty$ and
$$0\leqslant \frac{1-(1-X)^n}{(n-1-\alpha)X^2}\leqslant \frac{n}{(n-1-\alpha)X}\leqslant \frac{2}{(1-\alpha)X},~\forall n\geqslant2.$$
Therefore, for any given $\alpha\in(0,1)$, by the Dominated Convergence Theorem, we have
$$\frac{1}{n-1-\alpha}E\left[\frac{1-(1-X)^n}{X^2}\right]\to0,~\text{as }n\to\infty,$$
and
$$\frac{1}{n-1-\alpha}E\left[\frac{1-(1-X)^n-nX(1-X)^{n-1}}{X^2}\right]\to0,~\text{as }n\to\infty.$$
It follows that
\begin{equation}\label{zetato0}
	\zeta_{\alpha,n}\to0,~\text{as }n\to\infty.
\end{equation}
Therefore, 
$$\zeta_n^{\beta}+c_\beta-c-\frac{2c(1+\beta)}{n-1-\beta}\to c_\beta-c,~\text{as }n\to\infty.$$
Since $\beta>\alpha_c$, we have $c_\beta>c$, which implies that there exists $N_1\in\mathbb{N}^+$ such that 
\begin{equation*}\label{zetablower}
	 \zeta_{\beta,n}+c_\beta-c-\frac{2c(1+\beta)}{n-1-\beta} \geqslant0,~\forall n>N_1.
\end{equation*}
It is easy to see that for any $\alpha\in(0,1)$, we have $\gamma_{\alpha,n}>0$ for all $n\in\mathbb{N}^+$. Then 
$$\left(-\frac{(1+\alpha_c)c}{2(n-1-\alpha_c)}\right)\gamma_{\alpha_c,n} -\varepsilon \left(\zeta_{\beta,n}+c_\beta-c-\frac{2c(1+\beta)}{n-1-\beta}\right)\gamma_{\beta,n}<0,~\forall n>N_1,~\forall\varepsilon>0.$$
By choosing $\varepsilon$ such that
\begin{equation}\label{boundvare1}
0<\varepsilon<\frac{\frac{(1+\alpha_c)c}{2(N_1-1-\alpha_c)}\gamma_{\alpha_c,N_1}}{\left|\min_{2\leqslant n\leqslant N_1}  \left(\zeta_{\beta,n}+c_\beta-c-\frac{2c(1+\beta)}{n-1-\beta}\right)\gamma_{\beta,n}\right|},
\end{equation}
we have
$$\left(-\frac{(1+\alpha_c)c}{2(n-1-\alpha_c)}\right)\gamma_{\alpha_c,n} -\varepsilon \left(\zeta_{\beta,n}+c_\beta-c-\frac{2c(1+\beta)}{n-1-\beta}\right)\gamma_{\beta,n}<0,~\forall n\in\{2,\ldots,N_1\}.$$
For $n\geqslant 2$, let
\begin{equation*}\label{etadef}
    \eta_n\coloneqq \left(-\frac{(1+\alpha_c)c}{2(n-1-\alpha_c)}\right)\gamma_{\alpha_c,n} -\varepsilon \left(\zeta_n^{\beta}+c_\beta-c-\frac{2c(1+\beta)}{n-1-\beta}\right)\gamma_{\beta,n}.	
\end{equation*}
Then 
\begin{equation}\label{etabound}
	\eta_n <0,~\forall n\geqslant 2.
\end{equation}
It follows from (\ref{rangelemma28sw})  and $(\ref{zetaaupper})$ that for all $n\geqslant2$,
\begin{align}\label{etabound3}
	\eta_n &>\left(\zeta_{\alpha_c,n}-\frac{2c(1-a+\varepsilon a)(1+\alpha_c)}{n-1-\alpha_c}\right)\gamma_{\alpha_c,n}\nonumber\\
	&\quad\quad\quad\quad-\varepsilon \left(\zeta_{\beta,n}+c_\beta-c-\frac{2c(1-a+\varepsilon a)(1+\beta)}{n-1-\beta}\right)\gamma_{\beta,n}.
\end{align}
It follows from $(\ref{zetato0})$ that there exists $N_2\in\mathbb{N}^+$ such that for all $n>N_2$, we have
$$\frac{1}{c\gamma_{\beta,n}}\eta_n=\left(-\frac{(1+\alpha_c)c}{2(n-1-\alpha_c)}\right)\frac{\gamma_{\alpha_c,n}} {c\gamma_{\beta,n}}-\frac{\varepsilon }{c}\left(\zeta_{\beta,n}+c_\beta-c-\frac{2c(1+\beta)}{n-1-\beta}\right)<\frac{-\varepsilon(c_\beta-c)}{2c},$$
which implies 
$$\ell_1\coloneqq \inf_{n\geqslant 2}\left\{\left|\frac{1}{c\gamma_{\beta,n}}\eta_n\right|\right\}\neq0.$$
Recall from (\ref{gammasimre}) that $\gamma_{\alpha,n}=O(n^{-1-\alpha})$ for $\alpha>0$ and $\alpha\notin\mathbb{Z}$. Note that $\gamma_{1,n}=0$ for all $n\geqslant 2$. Therefore,
for $0<\alpha_1<\alpha_2<2$ and $\alpha_1\neq 1$, 
\begin{equation}\label{limitgammaratio}
\frac{\gamma_{\alpha_2,n}}{\gamma_{\alpha_1,n}}
\to0,~\text{as }n\to\infty.
\end{equation}
Then for $\beta\in(\alpha_c,2\alpha_c\wedge1)$, we have
$$\frac{\gamma_{2\alpha_c,n}-2\varepsilon \gamma_{\alpha_c+\beta,n}+\varepsilon^2 \gamma_{2\beta,n}}{\gamma_{\beta,n}}\to 0,~\text{as }n\to\infty.$$
It follows that 
$$\ell_2\coloneqq \max_{n\geqslant 2}\left|\frac{\gamma_{2\alpha_c,n}-2\varepsilon \gamma_{\alpha_c+\beta,n}+\varepsilon^2 \gamma_{2\beta,n}}{\gamma_{\beta,n}}\right|<\infty.$$
By choosing $a$ such that (\ref{boundbbb}) holds and
\begin{equation}\label{boundofb2}
	a\in\left(0,\frac{\ell_1}{\ell_2}\right), 
\end{equation}
it follows from $(\ref{etabound})$ that 
\begin{equation}\label{etabound2}
	\frac{1}{c\gamma_{\beta,n}}\eta_n+a\frac{\gamma_{2\alpha_c,n}-2\varepsilon \gamma_{\alpha_c+\beta,n}+\varepsilon^2 \gamma_{2\beta,n}}{\gamma_{\beta,n}}<0, ~\forall n\geqslant2.
\end{equation}
By $(\ref{diff_n})$, $(\ref{etabound3})$ and $(\ref{etabound2})$, we have
$$d_n'< a\gamma_{\beta,n}\left(\frac{1}{c\gamma_{\beta,n}}\eta_n+a\frac{\gamma_{2\alpha_c,n}-2\varepsilon \gamma_{\alpha_c+\beta,n}+\varepsilon^2 \gamma_{2\beta,n}}{\gamma_{\beta,n}}
\right)<0, ~\forall n\geqslant2.$$

Now we will show that $\sum_{n=1}^\infty d_n=0$. Recall that for $n\geqslant 2$, 
$$s_n=\frac{a}{c}(\zeta_{\alpha_c,n}+c)\gamma_{\alpha_c,n}-\frac{\varepsilon a}{c}(\zeta_{\beta,n}+c_\beta)\gamma_{\beta,n}+a\gamma_{\alpha_c,n}- \varepsilon a \gamma_{\beta,n}.$$
By $(\ref{zetato0})$ and $(\ref{limitgammaratio})$, we have, as $n\to\infty$,
$$\frac{s_n}{\gamma_{\alpha_c,n}}=\frac{a}{c}(\zeta_{\alpha_c,n}+c)-\frac{\varepsilon a}{c}(\zeta_{\beta,n}+c_\beta)\frac{\gamma_{\beta,n}}{\gamma_{\alpha_c,n}}+a- \varepsilon a \frac{\gamma_{\beta,n}}{\gamma_{\alpha_c,n}}\to2a>0,$$
which implies that there exists $N_3\in\mathbb{N}^+$ such that $s_n\geqslant0$ for all $n>N_3$. 
Note that $\sum_{n=1}^\infty r_n=1$ and  $(1-x)R'(x)=O((1-x)^{\alpha_c})$ as $x\to1^-$, which implies $(1-x)R'(x)$ is bounded on $[0,1)$. Then $\sum_{n=1}^\infty s_n =1$ because of Lemma \ref{lemmasrfunction} and Lemma \ref{sumsn1lemma}.  
Since $\sum_{n=1}^\infty\tilde{r}_n =R^2(1)=1$, we have 
$$\sum_{n=1}^\infty d_n=\sum_{n=1}^\infty s_n -\sum_{n=1}^\infty \tilde{r}_n=0.$$ 
The result $(\ref{sumincrease1})$ follows.

Now we are going to show that the distribution with probability generating function $R(x)$ satisfies $(\ref{lemma41})$.
Since $s_n-\tilde{r}_n=d_n<0$ for all $n\geqslant 2$, we have  
$$\sum_{n=2}^\infty\max\{0,s_n\}\leqslant \sum_{n=2}^\infty\max\{0,\tilde{r}_n\}=\sum_{n=2}^\infty\tilde{r}_n\leqslant \sum_{n=1}^\infty\tilde{r}_n\leqslant1.$$
Therefore, we can define a random variable $T$ on $\mathbb{N}^+$ with
\begin{equation*}\label{tdist1}
P(T=n)=
\begin{cases}
	1-\sum_{n=2}^\infty \max\{0,s_n\}, & n=1,\\
	\max\{0,s_n\}, & n\geqslant 2.\\
\end{cases}	
\end{equation*}
Let $\{r_n'\}_{n=1}^\infty$ be a sequence of real numbers such that $r_n'=r_n$ for all $n>N_3$, and 
\begin{equation}\label{r'inl6new} 
r_n'=\frac{c}{\lambda_n+c}P(T=n)+\frac{1}{\lambda_n+c}\sum_{i=n+1}^{\infty} \binom{i}{i-n+1}\lambda_{i, i-n+1}r_i',~\forall n\leqslant N_3.
\end{equation} 
 By $(\ref{Sx1})$ and Lemma \ref{lemmasrfunction}, we have 
\begin{equation*}\label{r'inl6new2}
r_n=\frac{c}{\lambda_n+c}s_n+\frac{1}{\lambda_n+c}\sum_{i=n+1}^{\infty} \binom{i}{i-n+1}\lambda_{i, i-n+1}r_i, ~\forall n\in\mathbb{N}^+.
\end{equation*}
Since $s_n\geqslant 0$ for all $n>N_3$, we have $P(T=n)=s_n$ for all $n>N_3$. Then
\begin{equation}\label{r'inl6newl6688} 
 r_n'=\frac{c}{\lambda_n+c}P(T=n)+\frac{1}{\lambda_n+c}\sum_{i=n+1}^{\infty} \binom{i}{i-n+1}\lambda_{i, i-n+1}r_i',~\forall  n>N_3.
 \end{equation}
 Recall from $(\ref{def11xalpha122})$ that $\gamma_{\alpha,n}\coloneqq(-1)^{n+1}\binom{\alpha}{n}$. Then for $n\geqslant 3$,
\begin{align*}
r_{n-1}-r_n&=a(-1)^{n}	\left(\binom{\alpha_c}{n-1}+\binom{\alpha_c}{n}\right)-\varepsilon a(-1)^{n}	\left(\binom{\beta}{n-1}+\binom{\beta}{n}\right)\\
&=-a\gamma_{\alpha_c+1,n}+\varepsilon a \gamma_{\beta+1,n}.
\end{align*}
It follows from (\ref{gammasimre}) and (\ref{limitgammaratio}) that 
\begin{equation}\label{diffrnn1ratio1}
	\frac{r_{n-1}-r_n}{n^{-2-\alpha_c}}=\frac{-a\gamma_{\alpha_c+1,n}+\varepsilon a \gamma_{\beta+1,n}}{n^{-2-\alpha_c}}\to\frac{a}{\Gamma(-1-\alpha_c)},~\text{as }n\to\infty.
\end{equation}
By Lemma \ref{uniqueulemma}, $(\ref{r'inl6new})$, $(\ref{r'inl6newl6688})$ and (\ref{diffrnn1ratio1}), we have
$$P(L_T(Y)=n)=r_n', ~\forall n\in\mathbb{N}^+.$$
Since $r_n'=r_n$ for $n>N_3$ and $P(T=n)\geqslant s_n$ for $n\geqslant2$, we have $r_n'\geqslant r_n>0$ for all $n\geqslant2$. 
Let $W_1$ be a random variable with $P(W_1=n)=r_n$ for $n\in\mathbb{N}^+$. Note that
$$P(L_T(Y)\geqslant 1)=1= P(W_1\geqslant 1).$$
Since $r_n'\geqslant r_n$ for all $n\geqslant 2$, we have
$$P(L_T(Y)\geqslant n)\geqslant P(W_1\geqslant n),~\forall n\geqslant 2,$$
which implies 
\begin{equation}\label{stochdomi1}
	L_T(Y)\succeq W_1.
\end{equation}
Let $W_2$ be a random variable with the same distribution as $W_1$ and independent of $W_1$. Then $R^2(x)$ is the probability generating function of $W_1+W_2$ and $P(W_1+W_2=n)=\tilde{r}_n$ for all $n\geqslant1$. Recall that $d_n=s_n-\tilde{r}_n<0$ for all $n\geqslant2$. Then
$$P(T=n)=\max\{0,s_n\}\leqslant \tilde{r}_n=P(W_1+W_2=n),~\forall n\geqslant 2,$$
which implies 
\begin{equation}\label{stochdomi2}
	T\preceq W_1+W_2.
\end{equation}
By $(\ref{stochdomi1})$ and $(\ref{stochdomi2})$, we have
$$ W_1\preceq L_T(Y)\preceq L_{W_1+W_2}(Y).$$
It follows that by letting $\mu_0$ be a distribution with 
\begin{equation}\label{taylorofRdown2}
\mu_0(\{n\})=r_n=
\begin{cases}	
1-a+\varepsilon a+a\gamma_{\alpha_c,1}-\varepsilon a \gamma_{\beta,1}, & n=1,\\
a\gamma_{\alpha_c,n}-\varepsilon a \gamma_{\beta,n}, & n\geqslant 2.
\end{cases}
\end{equation}
where $\beta\in(\alpha_c,2\alpha_c\wedge1)$, $\varepsilon$ is chosen to satisfy $(\ref{boundvare1})$ and $a$ is chosen to satisfy $(\ref{boundbbb})$ and $(\ref{boundofb2})$,  
we have 
$$\mu_0\preceq G_c(\mu_0).$$
Because $R'(1)=\infty$, the distribution $\mu_0$ has infinite mean.
\end{proof}

\begin{Lemma}
\label{lemmadecrease}
Suppose that $0<c< E[(1-(1-X)^{1/2})/X^2]$ and $E[1/X]<\infty$. If there exists $p<1$ such that
\begin{equation}\label{downtbound}
	n^{-p}E\left[\frac{1-(1-X)^n}{X^2}\right]\to 0\text{ as }n\to\infty,
\end{equation}
then there exists a distribution $\nu_0\in\mathcal{S}$ with infinite mean such that 
\begin{equation}\label{Gdedown}
G_c(\nu_0)\preceq \nu_0.
\end{equation}
\end{Lemma}
\begin{proof}
We will show that there exists a function $R(x)$, which is infinitely differentiable at $x=0$ and has the following property. Suppose the Taylor series of the function $R^2(x)$ that centered at $x=0$ is $R^2(x)=\sum_{n=1}^\infty \tilde{r}_n x^n$. Suppose the Taylor series of the function $S(x)$ that satisfies $(\ref{SRPGF})$ centered at $x=0$ is $S(x)=\sum_{n=1}^\infty s_n x^n$. Then there exists some $M\in\mathbb{N}^+$ such that $\sum_{i=n}^\infty s_i\geqslant \sum_{i=n}^\infty \tilde{r}_i$ for all $n\geqslant M$. We will then use this property to show that this distribution satisfies $(\ref{Gdedown})$. 
Since $0<c< E[(1-(1-X)^{1/2})/X^2]$, we have $\alpha_c<1/2$.
Recall from $(\ref{def11xalpha1})$ that for $\alpha>0$, we have $1-(1-x)^\alpha=\sum_{n=1}^\infty\gamma_{\alpha,n} x^n$.
Let
\begin{equation}\label{lemma29rk}
R(x)=\frac{1}{\sum_{i=k}^\infty\gamma_{\alpha_c,i}+\sum_{i=k}^\infty\gamma_{\beta,i}}\left((1-(1-x)^{\alpha_c})-\sum_{i=1}^{k-1}\gamma_{\alpha_c,i}x^i+(1-(1-x)^\beta)-\sum_{i=1}^{k-1}\gamma_{\beta,i}x^i\right),
\end{equation}
where 
$$\beta\in\left(\alpha_c,\min\{\alpha_c+1-p,2\alpha_c,1/2\}\right),$$ 
and $k$ is a positive integer which will be chosen later.
Then
\begin{equation}\label{lemma29orderbeta}
	0<\alpha_c<\beta<2\alpha_c<\alpha_c+\beta<2\beta<1.
\end{equation}

Firstly, we will derive the Taylor series of $R(x)$, $R^2(x)$ and $S(x)$. The Taylor series for the function $R(x)$ centered at $x=0$ is 
\begin{equation*}\label{taylorofRdown1change}
	R(x)=\sum_{n=1}^\infty r_n x^n,
\end{equation*}
where
\begin{equation*}\label{taylorofRdown2change}
r_n=
\begin{cases}	
0, & n<k,\\
\frac{1}{\sum_{i=k}^\infty\gamma_{\alpha_c,i}+\sum_{i=k}^\infty\gamma_{\beta,i}}\left(\gamma_{\alpha_c,n}+\gamma_{\beta,n}\right), & n\geqslant k.
\end{cases}
\end{equation*}
Let 
\begin{equation}\label{Sx1downchange}
S(x)=\frac{1}{c}E\left[\frac{R(x)-R((1-X)x)}{X^2}\right]-\frac{x}{c}E\left[\frac{R(X+(1-X)x)-R((1-X)x)}{X^2}\right]+R(x).	
\end{equation}
We can do a calculation similar to what we did in Lemma \ref{lemmaincrease} to derive the Taylor series for $S(x)$ and $R^2(x)$. The Taylor series for the function $S(x)$ centered at $x=0$ is
\begin{align*}
	S(x)=\sum_{n=1}^\infty s_n x^n,
\end{align*}
where for $n\geqslant k$, we have
\begin{align}\label{sndownenalphac}
s_n=&\frac{1}{\sum_{i=k}^\infty\gamma_{\alpha_c,i}+\sum_{i=k}^\infty\gamma_{\beta,i}}\Bigg(\frac{1}{c}E\left[\frac{1-(1-X)^n-\frac{n(1-X)^{\alpha_c}}{n-1-\alpha_c}+\frac{n(1-X)^{n-1}}{n-1-\alpha_c}}{X^2}\right]\gamma_{\alpha_c,n}\nonumber\\
&\quad\quad\quad\quad+\frac{1}{c}E\left[\frac{1-(1-X)^n-\frac{n(1-X)^{\beta}}{n-1-\beta}+\frac{n(1-X)^{n-1}}{n-1-\beta}}{X^2}\right]\gamma_{\beta,n}+\gamma_{\alpha_c,n}+\gamma_{\beta,n}\Bigg).
\end{align}
For $R(x)^2$, after taking square of the right-hand side of (\ref{lemma29rk}), we get that
the Taylor series centered at $x=0$ is
\begin{align*}
	R(x)^2&=\sum_{n=1}^\infty \tilde{r}_n x^n,
\end{align*}
where
\begin{align*}\label{rndown1}
\tilde{r}_n=
\begin{cases}
	0, & n<2k,\\
	\frac{1}{(\sum_{i=k}^\infty\gamma_{\alpha_c,i}+\sum_{i=k}^\infty\gamma_{\beta,i})^2}\bigg(4\gamma_{\alpha_c,n}+4\gamma_{\beta,n}-2\sum_{i=1}^{k-1}\gamma_{\alpha_c,i}\gamma_{\alpha_c,n-i}-2\sum_{i=1}^{k-1}\gamma_{\alpha_c,i}\gamma_{\beta,n-i}\\
	\quad\quad-2\sum_{i=1}^{k-1}\gamma_{\beta,i}\gamma_{\alpha_c,n-i}-2\sum_{i=1}^{k-1}\gamma_{\beta,i}\gamma_{\beta,n-i}-\gamma_{2\alpha_c,n}-2\gamma_{\alpha_c+\beta,n}-\gamma_{2\beta,n}\bigg), & n\geqslant2k.
\end{cases}
\end{align*}
For $n\geqslant 2k$, we can rewrite $\tilde{r}_n$ as
\begin{align*}
\tilde{r}_n&= \frac{2\gamma_{\alpha_c,n}+2\gamma_{\beta,n}}{\sum_{i=k}^\infty\gamma_{\alpha_c,i}+\sum_{i=k}^\infty\gamma_{\beta,i}}+\frac{1}{(\sum_{i=k}^\infty\gamma_{\alpha_c,i}+\sum_{i=k}^\infty\gamma_{\beta,i})^2}\bigg(2\sum_{i=1}^{k-1}\gamma_{\alpha_c,i}(\gamma_{\alpha_c,n}-\gamma_{\alpha_c,n-i})\\
&\quad\quad+2\sum_{i=1}^{k-1}\gamma_{\alpha_c,i}(\gamma_{\beta,n}-\gamma_{\beta,n-i})+2\sum_{i=1}^{k-1}\gamma_{\beta,i}(\gamma_{\alpha_c,n}-\gamma_{\alpha_c,n-i})+2\sum_{i=1}^{k-1}\gamma_{\beta,i}(\gamma_{\beta,n}-\gamma_{\beta,n-i})\\
&\quad\quad-\gamma_{2\alpha_c,n}-2\gamma_{\alpha_c+\beta,n}-\gamma_{2\beta,n}\bigg).
\end{align*}

Our next goal is to show that we can find $k>0$ such that 
\begin{equation}\label{diffsrdown2change}
 \sum_{i=n}^\infty s_i> \sum_{i=n}^\infty \tilde{r}_i, \text{ and } s_n>0, ~\forall n\geqslant k, \text{ and } \sum_{n=k+1}^\infty s_n> 1.
\end{equation}
Let 
\begin{equation*}\label{Diffsrdownchange}
d_n\coloneqq s_n-\tilde{r}_n,~\forall n\in\mathbb{N}^+. 
\end{equation*} 
Let 
$$c_{\beta}\coloneqq E\left[\frac{1-(1-X)^{\beta}}{X^2}\right].$$ 
Note that
$$\gamma_{\alpha,n+1}=\frac{n-\alpha}{n+1}\gamma_{\alpha,n}<\gamma_{\alpha,n},~\forall n\geqslant1, ~\forall\alpha\in(0,1),$$
which implies $\gamma_{\alpha,n}-\gamma_{\alpha,n-i}<0$ for all $\alpha\in(0,1)$.
Recall from (\ref{lemma29orderbeta}) that $\alpha_c,\beta,2\alpha_c,\alpha_c+\beta$ and $2\beta$ are all in $(0,1)$, which implies $\gamma_{\alpha_c,n}$, $\gamma_{\beta,n}$, $\gamma_{2\alpha_c,n}$, $\gamma_{\alpha_c+\beta,n}$, $\gamma_{2\beta,n}$ are all positive for all $n\geqslant1$. Recall the definition of $\zeta_{\alpha,n}$ from $(\ref{defzetal61})$. 
Then
\begin{align}\label{Dndown1change}
d_n\geqslant
\begin{cases}
	\frac{1}{\sum_{i=k}^\infty\gamma_{\alpha_c,i}+\sum_{i=k}^\infty\gamma_{\beta,i}}(\frac{1}{c}\zeta_{\alpha_c,n}\gamma_{\alpha_c,n}+\frac{1}{c}\zeta_{\beta,n}\gamma_{\beta,n}+\frac{1}{c}(c_{\beta}-c)\gamma_{\beta,n}+2\gamma_{\alpha_c,n}+2\gamma_{\beta,n}),&k\leqslant n<2k,\\
	\frac{1}{\sum_{i=k}^\infty\gamma_{\alpha_c,i}+\sum_{i=k}^\infty\gamma_{\beta,i}}\left(\frac{1}{c}\zeta_{\alpha_c,n}\gamma_{\alpha_c,n}+\frac{1}{c}\zeta_{\beta,n}\gamma_{\beta,n}+\frac{1}{c}(c_{\beta}-c)\gamma_{\beta,n}\right),&n\geqslant 2k.
\end{cases}
\end{align}
By $(\ref{zeta223})$, we have
$$-\frac{1+\alpha}{n-1-\alpha}E\left[\frac{1-(1-X)^n}{X^2}\right]
\leqslant \zeta_{\alpha,n}\leqslant \frac{1+\alpha}{n-1-\alpha}E\left[\frac{1-(1-X)^\alpha}{X^2}\right],~\forall n>1+\alpha.$$
It then follows from $(\ref{downtbound})$ that
$$n^{1-p}\zeta_{\alpha,n}\to 0,\text{ as } n\to\infty.$$
By ($\ref{gammasimre}$), we have $\gamma_{\alpha_c,n}/\gamma_{\beta,n}=O(n^{\beta-\alpha_c})$.
Since $\beta<\alpha_c+1-p$, we have
\begin{equation*}\label{limiten1change}
	\frac{\zeta_{\alpha_c,n}\gamma_{\alpha_c,n}}{\gamma_{\beta,n}}\to 0, \text{ as }n\to\infty.
\end{equation*}
Since $E[(1-(1-X)^\alpha)/X^2]$ increases as $\alpha>0$ increases and $0<\alpha_c<\beta$, we have $c<c_{\beta}$.
Recall from $(\ref{zetato0})$ that $\zeta_{\alpha,n}\to0$ as $n\to\infty$ for all $\alpha\in(0,1)$. Then there exists $k\in\mathbb{N}^+$ such that 
\begin{equation}\label{bounddowncase12change}
\zeta_{\alpha_c,n}\gamma_{\alpha_c,n}+\zeta_{\beta,n}\gamma_{\beta,n}+(c_{\beta}-c)\gamma_{\beta,n}>0,~\forall n\geqslant k.
\end{equation}
We will choose this value of $k$ to use in (\ref{lemma29rk}).
By (\ref{Dndown1change}) and (\ref{bounddowncase12change}), we have 
$$d_n> 0, ~\forall n\geqslant k.$$
By (\ref{sndownenalphac}) and (\ref{bounddowncase12change}), we have 
\begin{equation}\label{usedinproofoftheo6}
	s_n\geqslant \frac{2\gamma_{\alpha_c,n}+2\gamma_{\beta,n}}{\sum_{i=k}^\infty\gamma_{\alpha_c,i}+\sum_{i=k}^\infty\gamma_{\beta,i}},~\forall n\geqslant k,
\end{equation}
which implies
$$\sum_{n=k}^\infty s_n\geqslant \frac{\sum_{n=k}^\infty(2\gamma_{\alpha_c,n}+2\gamma_{\beta,n})}{\sum_{i=k}^\infty\gamma_{\alpha_c,i}+\sum_{i=k}^\infty\gamma_{\beta,i}}=2.$$
The result $(\ref{diffsrdown2change})$ follows.

Now we are going to construct a distribution which satisfies $(\ref{Gdedown})$. Note that $\sum_{n=1}^\infty r_n=1$ and  $(1-x)R'(x)=O((1-x)^{\alpha_c})$ as $x\to1^-$, which implies $(1-x)R'(x)$ is bounded on $[0,1)$. Then $\sum_{n=1}^\infty s_n =1$ because of Lemma \ref{lemmasrfunction} and Lemma \ref{sumsn1lemma}. By $(\ref{diffsrdown2change})$, there exists $M_1\geqslant k$ such that 
$$\sum_{n=M_1+1}^\infty s_n \leqslant 1 \text{ and } \sum_{n=M_1}^\infty s_n > 1.$$
Let $T$ be a random variable taking its values in $\mathbb{N}^+$ with 
\begin{equation}\label{tdistdown1change}
P(T=n)=
\begin{cases}
	0, & n<M_1,\\
	1-\sum_{i=M_1+1}^\infty s_i, & n=M_1,\\
	s_n, &n>M_1.
\end{cases}	
\end{equation}
Let $\{r_n'\}_{n=1}^\infty$ be a sequence of real numbers such that $r_n'=r_n$ for all $n>M_1$, and
\begin{equation}\label{r'inl6newl7} 
r_n'=\frac{c}{\lambda_n+c}P(T=n)+\frac{1}{\lambda_n+c}\sum_{i=n+1}^{\infty} \binom{i}{i-n+1}\lambda_{i, i-n+1}r_i',~\forall  n\leqslant M_1.
\end{equation} 
It follows from $(\ref{Sx1downchange})$ and Lemma \ref{lemmasrfunction} that
\begin{equation*}\label{r'inl6new2}
r_n=\frac{c}{\lambda_n+c}s_n+\frac{1}{\lambda_n+c}\sum_{i=n+1}^{\infty} \binom{i}{i-n+1}\lambda_{i, i-n+1}r_i, ~\forall n\in\mathbb{N}^+.
\end{equation*}
Since $r_n'=r_n$ and $P(T=n)=s_n$ for all $n>M_1$, we have 
\begin{equation}\label{r'inl6newl788} 
 r_n'=\frac{c}{\lambda_n+c}P(T=n)+\frac{1}{\lambda_n+c}\sum_{i=n+1}^{\infty} \binom{i}{i-n+1}\lambda_{i, i-n+1}r_i',~\forall  n>M_1.
 \end{equation}
By (\ref{gammasimre}), we have
$$\frac{r_{n-1}-r_n}{n^{-2-\alpha_c}}\to\frac{1}{\sum_{i=k}^\infty\gamma_{\alpha_c,i}+\sum_{i=k}^\infty\gamma_{\beta,i}}\frac{1}{\Gamma(-1-\alpha_c)},~\text{as }n\to\infty.$$ 
Therefore, by Lemma \ref{uniqueulemma}, (\ref{r'inl6newl7}) and (\ref{r'inl6newl788}), we have
$$P(L_T(Y)=n)=r_n', ~\forall n\in\mathbb{N}^+.$$
Since $P(T=n)=0< s_n$ for all $n\in\{k,\ldots,M_1-1\}$ and
$$P(T=M_1)=1-\sum_{n=M_1+1}^\infty s_n<\sum_{n=M_1}^\infty s_n-\sum_{n=M_1+1}^\infty s_n =s_{M_1},$$  
we have $r_n'\leqslant r_n$ for all $n\in\{k,\ldots,M_1\}$. 
It follows that $P(L_T(Y)=n)=r_n$ for all $n>M_1$, and $P(L_T(Y)=n)\leqslant r_n$ for all $n\in\{k,\ldots,M_1\}$. Also, $P(L_T(Y)=n)\geqslant0=r_n$ for all $n<k$. Let $W$ be a random variable with probability mass function $P(W=n)=r_n$. Then 
$$ L_T(Y)\preceq W.$$
Let $T'$ be a random variable that has the same distribution as $T$. Let $Y',Y''\sim \text{Exp}(c)$. Let $T$, $T'$, $Y$, $Y'$ and $Y''$ all be independent of one another. Then $L_T(Y)$ and $L_{T'}(Y')$ are independent and identically distributed. Let $W'$ be a random variable that has the same distribution as $W$ and is independent of $W$. Then
$$ L_T(Y)+L_{T'}(Y')\preceq W+W',$$
which implies 
$$P(L_T(Y)+L_{T'}(Y')\geqslant n)\leqslant P(W+W'\geqslant n)=\sum_{i=n}^\infty\tilde{r}_i.$$
By $(\ref{diffsrdown2change})$ and (\ref{tdistdown1change}), we have 
$$\sum_{i=n}^\infty\tilde{r}_i\leqslant \sum_{i=n}^{\infty}s_i=\sum_{i=n}^\infty P(T=i),~\forall n>M_1.$$
For $n<M_1$, we have
$$P(L_T(Y)+L_{T'}(Y')= n)\geqslant 0=P(T=n),~\forall n\in\{1,\ldots, M_1-1\}.$$
Therefore, we have $L_T(Y)+L_{T'}(Y')\preceq T$, which implies $L_{L_T(Y)+L_{T'}(Y')}(Y'')\preceq L_T(Y)$. Let $\nu_0$ be the distribution of $L_T(Y)$. Then 
$$G_c(\nu_0)\preceq \nu_0.$$
The distribution $\nu_0$ has infinite mean because $P(L_T(Y)=n)=r_n$ for all $n>M_1$ and $\sum_{n=1}^\infty nr_n$ diverges.
\end{proof}

\subsection{Proof of Theorem \ref{Theoconvlambda1}}
\begin{proof}[Proof of Theorem \ref{Theoconvlambda1}]
Recall from $(\ref{taylorofRdown2})$ in the proof of Lemma $\ref{lemmaincrease}$ that there exists a distribution $\mu_0$ with
\begin{equation*}
\mu_0(\{n\})=
\begin{cases}	
1-a+\varepsilon a+a\gamma_{\alpha_c,1}-\varepsilon a \gamma_{\beta,1}, & n=1,\\
a\gamma_{\alpha_c,n}-\varepsilon a \gamma_{\beta,n}, & n\geqslant 2,
\end{cases}
\end{equation*}
where $\beta\in(\alpha_c, 2\alpha_c\wedge 1)$ and $a$ and $\varepsilon$ are sufficiently small, such that $G(\mu_0)\succeq \mu_0$. The convolution of $\mu_0$ and itself is $\left(\mu_0\ast\mu_0\right)(\{n\})=\tilde{r}_n$, where $\tilde{r}_n$ was given in (\ref{eq123lemma281}) and (\ref{eq22slemma282}).
Recall that $\gamma_{\alpha,n}>0$ for all $n\in\mathbb{N}^+$ when $0<\alpha<1$, and $\gamma_{\alpha_2,n}/\gamma_{\alpha_1,n}\to 0$ as $n\to\infty$ for $0<\alpha_1<\alpha_2$, $\alpha_1\notin\mathbb{Z}$. Then for some positive constant $\ell_1$, we have
 $$ \left(\mu_0\ast\mu_0\right)(\{n,n+1,\ldots\})\leqslant \sum_{i=n}^\infty \ell_1a\gamma_{\alpha_c,i},~ \forall n\in\mathbb{N}^+.$$
Recall from (\ref{usedinproofoftheo6})  and (\ref{tdistdown1change}) in the proof of Lemma $\ref{lemmadecrease}$ that
there exists a random variable $T$ with 
$$ P(T\geqslant n)\geqslant \min\left\{1,\frac{\sum_{i= n}^\infty(2\gamma_{\alpha_c,i}+2\gamma_{\beta,i})}{\sum_{i=k}^\infty\gamma_{\alpha_c,i}+\sum_{i=k}^\infty\gamma_{\beta,i}} \right\},\forall n\geqslant 1,$$
where $\beta\in(\alpha_c,\min\{\alpha_c+1-p,2\alpha_c,1/2\})$ and $k$ is sufficiently large, such that the distribution of $L_T(Y)$, which will be denoted by $\nu_0$, satisfies $G_c(\nu_0)\preceq \nu_0$. By choosing $a$ small enough such that $\ell_1 a<2$ and choosing $k$ large enough such that $\sum_{i=k}^\infty\gamma_{\alpha_c,i}+\sum_{i=k}^\infty\gamma_{\beta,i}<1$, we have
$$P(T\geqslant n)\geqslant \left(\mu_0\ast\mu_0\right)(\{n,n+1,\ldots\}),~\forall n\in\mathbb{N}^+,$$
which will imply 
\begin{equation}\label{succeq1}
\nu_0\succeq G_c(\mu_0).
\end{equation}
Since $G_c(\mu_0)\succeq \mu_0$ and $G_c$ is monotone, we have that $\{G_c^n(\mu_0)\}_{n=1}^\infty$ is an increasing sequence of distributions and $\lim_{n\to\infty}G_n(\mu_0)$ exists. Similarly, we have that $\{G_c^n(\nu_0)\}_{n=1}^\infty$ is a decreasing sequence of distributions and $\lim_{n\to\infty}G_c^n(\nu_0)$ exists. By $(\ref{succeq1})$ and the monotonicity of $G$, we have
$$\nu_0\succeq \lim_{n\to\infty}G_c^n(\nu_0)\succeq \lim_{n\to\infty}G_c^n(\mu_0)\succeq \mu_0.$$
By a similar argument to the one in (\ref{lemma18proof111}) in Lemma \ref{lemma18nup1}, we have that the limits $\lim_{n\to\infty}G_c^n(\mu_0)$ and $\lim_{n\to\infty}G_c^n(\nu_0)$ are solutions to (\ref{theoremequ1}).
The result follows because distributions $\lim_{n\to\infty}G_c^n(\mu_0)$ and $\lim_{n\to\infty}G_c^n(\nu_0)$ are in $\mathcal{S}$. 
\end{proof}

\section*{Acknowledgements}
The author would like to convey most sincere gratitude and appreciation to Professor Jason Schweinsberg for his guidance, patience and encouragement. His methodological support, constructive suggestions and valuable feedback have played an important role in the development of this paper.

\bibliographystyle{abbrv}
\bibliography{references.bib}

\end{document}